\documentclass[reqno]{amsart}
\usepackage[margin=1in]{geometry}  
\usepackage{graphicx}              
\usepackage{amsmath, amssymb, amsthm, epsfig}               
\usepackage{enumerate}
\usepackage{amsfonts}              
\usepackage{amsthm}                
\usepackage{amsthm}
\usepackage{enumerate}
\theoremstyle{plain}
\usepackage{color}
\usepackage{hyperref}
\hypersetup{colorlinks,citecolor=blue}
\usepackage{enumitem}
\newtheorem{theorem}{Theorem}
\newtheorem{remark}[theorem]{Remark}
\newtheorem{lemma}[theorem]{Lemma}
\newtheorem{proposition}[theorem]{Proposition}
\newtheorem{corollary}[theorem]{Corollary}
\newtheorem{definition}[theorem]{Definition}

\newcommand{\R}{\mathbb{R}}

\newcommand{\weak}{\rightharpoonup}

\renewcommand{\l}{\left}
\renewcommand{\r}{\right}

\def\abs#1{\left|{#1}\right|}

\numberwithin{theorem}{section}
\numberwithin{equation}{section}

scaled\magstep2

\makeatletter
\def\section{\@startsection {section}{1}{\z@}{-3.5ex plus -1ex minus
        -.2ex}{2.3ex plus .2ex}{\large\bf}}
\makeatother \numberwithin{equation}{section}
\DeclareMathOperator{\dist}{dist} 
\begin{document}
\title[Global existence and finite-time blow-Up of solutions]{On Fractional $1$-Laplacian Evolution Equation}
\author[R. Arora]{Rakesh Arora}
\address[R. Arora]{ Department of Mathematical Sciences, Indian Institute of Technology (IIT-BHU) Varanasi, Uttar Pradesh 221005, India}
\email{rakesh.mat@iitbhu.ac.in, arora.npde@gmail.com}
\author[N. K. Maurya]{Nitin Kumar Maurya}
\address[N. K. Maurya]{Department of Mathematical Sciences, Indian Institute of Technology (IIT-BHU) Varanasi, Uttar Pradesh 221005, India}
\email{nitinkumarmaurya.rs.mat23@itbhu.ac.in }
\begin{abstract}
We study a nonlocal evolution problem driven by the fractional $1$-Laplacian with a Carath\'eodory nonlinearity satisfying subcritical growth conditions. We develop a potential well framework to investigate the existence and qualitative behavior of solutions at different initial energy levels. By combining a modified potential well method with Galerkin approximations, we establish the global existence of weak and strong solutions under appropriate conditions on the initial energy and the sign of the associated Nehari functional. To address the singular nature of the fractional $1$-Laplacian, we approximate the problem by a family of fractional $p$-Laplacian equations, derive estimates uniform with respect to $p>1$, and pass to the limit as $p\to 1^{+}$. Furthermore, in the low-dimensional regime $N<2s$, we employ a subdifferential approach to establish the local existence of strong solutions and investigate their qualitative behavior.
\end{abstract}
\maketitle
\maketitle
\tableofcontents
\section{Introduction}
In this article, we study the following evolution equation involving the fractional $1$--Laplacian
\begin{equation} \label{main:problem}
        \left\{
    \begin{aligned}
      u_{t}+(-\Delta)_{1}^{s} u &=  f(x,u) &&
 \text{in } \Omega \times (0, T),\\
 u &= 0 && \text{in}~  \mathbb{R}^N \setminus \Omega \times (0, T),\\
 u(x,0)&= u_0(x) && \text{in}~ \Omega,
\end{aligned}
    \right.
 \end{equation}
 where $s\in (0,1)$, $T>0$ and $\Omega\subseteq\mathbb{R}^N$, $N\geq 1$ be a bounded domain with Lipschitz boundary and $f: \Omega \times \mathbb{R}\longrightarrow \mathbb{R}$ is a Carath\'eodory function satisfying subcritical growth conditions. The fractional $1$-Laplacian is defined as
 \begin{equation}\label{main:operator}
     (-\Delta)_{1}^{s} u(x):= \text{P.V.} \int_{\mathbb{R}^{N}} \frac{u(x)-u(y)}{|u(x)-u(y)|}\frac{d y}{|x-y|^{N+s}} \quad \text { for all } x \in \mathbb{R}^{N},
 \end{equation}
 where P.V. is a commonly used abbreviation for in the ``Principal Value" sense. This operator can be regarded as the limiting case $p=1$ of the fractional $p$-Laplacian
 \begin{equation}\label{p-laplacian operator}
     (-\Delta)_{p}^{s} u(x):= \text{P.V.} \int_{\mathbb{R}^{N}} \frac{{|u(x)-u(y)|}^{p-2}(u(x)-u(y))}{|x-y|^{N+sp}}\,dy \quad \text { for all } x \in \mathbb{R}^{N}.
 \end{equation}
 
 In recent decades, the study of nonlinear equations involving the nonlocal operators e.g. the fractional Laplacian $(-\Delta)^s$, the $p$-fractional Laplacian $(-\Delta)^s_p$, $1 < p< \infty$, has gained significant attention due to its rich analytic structure and wide-ranging applications in fields such as optimization, finance, anomalous diffusion, phase transitions, flame propagation, and minimal surfaces. The fractional Sobolev spaces provide the appropriate functional framework for such equations. For a comprehensive overview of the basic properties of these spaces and the operator $(-\Delta)^s$, as well as their applications to partial differential equations, we refer interested readers to \cite{Valdinoci-2012, Molica-2016} and the references therein.
 
 In recent years, partial differential equations involving the $1$-Laplacian $\Delta_{1}u=\operatorname{div}\left(\frac{Du}{|Du|}\right)$ have received considerable attention mainly due to their strong connections with image processing, geometric flows, and total variation minimization (see \cite{Gilboa-2009, Bogelein-2015}). Andreu et al. \cite{Andreu-2000, Andreu-2001, Andreu-2004} initiated the systematic study of evolution problems driven by $1$-Laplacian, which laid the foundations of the modern theory. From this perspective, problem \eqref{main:problem} could be seen as the nonlocal analogue of the following $1$-Laplacian Dirichlet problem 
 \begin{equation} \label{local case parabolic problem}
        \left\{
    \begin{aligned}
      u_t - \operatorname{div}\left(\frac{Du}{\abs{Du}}\right)u &=  f(x,t) &&
 \text{in } \Omega\times(0,T),\\
 u &= 0 && \text{on}~  \partial\Omega\times(0,T),\\
 u(x,0)&=u_{0} && \text{in } \Omega.
\end{aligned}
    \right.
 \end{equation}
 A major difficulty in the analysis of \eqref{local case parabolic problem} lies in the highly singular nature of the operator, especially at points where the gradient $Du$ vanishes. Andreu et al.  \cite{Andreu-2001} addressed this issue in the homogeneous case $f\equiv 0$ in $\Omega\times(0,T]$ and $u=\varphi$ on $\partial\Omega\times(0,T]$, where $\varphi\in L^{\infty}(\partial\Omega)$. By exploiting the Anzellotti pairing theory \cite{Anzellotti-1983}, they provided an interpretation of the quotient $\frac{Du}{|Du|}$ when $Du$ vanishes. More precisely, they introduced a vector field $z\in L^{\infty}(\Omega, \mathbb{R}^{N})$ with $\|z\|_{L^{\infty}(\Omega, \mathbb{R}^{N})}\leq1$ and $(z,Du)=|Du|$ in the sense of measures, allowing $z$ to play role of $\frac{Du}{|Du|}$. Within this framework, by using the techniques of completely accretive operators and the Crandall-Liggett semigroup generation theorem \cite{Crandall-1971}, they established the existence and uniqueness of an entropy solution. Several authors have also developed alternative approaches to the homogeneous problem. In particular, Hardt et al. \cite{Hardt-1994} used an approximation techniques that consist in working with a class of nondegenerate parabolic approximation problems. By deriving suitable uniform estimates and passing to the limit, they proved the existence of solutions to the original problem. Later, Andreu et al. \cite{Andreu-2002} studied the behavior of the solution to the homogeneous case of the problem \eqref{local case parabolic problem} near the extinction time for the Dirichlet and Neumann boundary conditions.
 
  For the inhomogeneous case $f\neq 0$, fewer results are available in the literature. A notable contribution in this direction is due to De-Le\'on et al. \cite{Leon-2022} who considered $u_{0}\in L^{2}(\Omega)$ and $f\in L^{1}(0,T;L^{2}(\Omega))$. Their approach relies on approximating the problem by a parabolic $p$-Laplacian equation and then passing to the limit as $p\to 1^{+}$, yielding global existence and uniqueness of solutions. In \cite{Hauer-2019}, Hauer et al. extended the work of B\'enilan et al. \cite{Benilan-1981} for the class of homogeneous operators of order $\alpha =0$. They proved that the mild solutions generated by the associated nonlinear semigroup are locally absolutely continuous and satisfy the corresponding evolution equation almost everywhere in time. Most recently, Alves et al. \cite{Tahir-2022} proved the global existence of a weak solution to the problem \eqref{local case parabolic problem} with nonlinearity $f(u)$ instead of $f(x,t)$, where $f$ is a continuous function and satisfies subcritical growth condition, by employing an approximation technique based on the associated $p$-Laplacian problem combined with the potential well method and Galerkin's method, and then passing to the limit as $p \to 1^{+}$. However, their results are restricted to the low initial energy case, and the initial datum is required to satisfy the additional integrability condition $u_0 \in W^{1,p_0}_{0}(\Omega)$ for some fixed $p_0 \in (1,2)$, leaving the critical initial energy case and finite-time blow-up analysis unaddressed.
  
 
 In contrast to the well-developed theory for the local operator, the study of nonlocal problems driven by the fractional $1$-Laplacian remains in its early stages. In the stationary setting, Bucur et al. \cite{Bucur-2023} investigated the existence of $(s,1)$- harmonic functions and their equivalence to the minimizers of associated energy functional. Furthermore, for the problem 
 \begin{equation} \label{nonlocal case elliptic problem}
\left(-\Delta\right)^{s}_{1} u =  f \quad
 \text{in } \Omega, \quad \text{and} \quad u = 0 \quad \text{in}~  \mathbb{R}^{N}\setminus\Omega.
 \end{equation}
 assuming that the source term $f$ belongs to $L^{\frac{N}{s}}(\Omega)$ and its norm is sufficiently small, Bucur \cite{Bucur-2025} investigated the minimizer of the following functional
 \begin{equation*}
     \mathcal{F}^{s}_{1}(u):=\frac{1}{2}\int_{\mathbb{R}^{N}\times\mathbb{R}^{N}}\frac{|u(x)-u(y)|}{|x-y|^{N+s}}\,dx\,dy-\int_{\Omega}fu\,dx
 \end{equation*}
 and proved that the minimizer corresponds to the weak solution of the problem \eqref{nonlocal case elliptic problem}. Moreover, it was shown that when $\|f\|_{L^{\frac{N}{s}}(\Omega)} > \left(2S_{N,s}\right)^{-1}$, where $S_{N,s}$ denotes the sharp constant in the fractional Sobolev inequality, the energy functional may be unbounded from below, and a global minimizer may fail to exist. Furthermore, they proved the flatness results, asserting that any weak solution $u$ satisfies $u(x) = u(y)$ on a set of positive Lebesgue measure in $\mathbb{R}^N \times \mathbb{R}^N$. This shows that the auxiliary antisymmetric function $z \in L^\infty(\mathbb{R}^N \times \mathbb{R}^N)$ with $\|z\|_{L^\infty} \leq 1$, which plays the role of the formal quotient $\frac{u(x)-u(y)}{|u(x)-u(y)|}$ on the set where $u(x) \neq u(y)$, cannot in general be replaced by the pointwise sign function; hence, the auxiliary function $z$ is indispensable in the definition of a weak solution. Regarding the regularity result, we refer the reader to \cite{Novaga-2023}. In \cite{Li-2024}, Li et al. further studied the asymptotic behavior of the renormalized solution for fractional $p$-Laplacian type equations and established the existence of solutions in the sense of renormalization for the problem \eqref{nonlocal case elliptic problem}, where $0\leq f\in L^{1}(\Omega)$. The corresponding parabolic problem driven by the fractional $1$-Laplacian is given by
 \begin{equation} \label{nonlocal case parabolic problem}
        \left\{
    \begin{aligned}
      u_t + \left(-\Delta\right)^{s}_{1} u &=  f &&
 \text{in } \Omega\times(0,T),\\
 u &= 0 && \text{in}~  \mathbb{R}^{N}\setminus\Omega\times(0,T),\\
 u(x,0)&=u_{0} && \text{in } \Omega.
\end{aligned}
    \right.
 \end{equation}
 A rigorous notion of solution for \eqref{nonlocal case parabolic problem} was first introduced by Maz\'on et al. \cite{Mazon-2016}. In analogy with the local theory based on Anzellotti’s pairing, they proved that the interpretation of the quotient $\frac{u(x,t)-u(y,t)}{|u(x,t)-u(y,t)|}$ requires the introduction of an auxiliary function $z\in L^{\infty}(\mathbb{R}^{N}\times\mathbb{R}^{N}\times[0,T])$ such that $z=\frac{u(x,t)-u(y,t)}{|u(x,t)-u(y,t)|}$ in the set $\left\{\mathbb{R}^{N}\times\mathbb{R}^{N}\times[0,T]: u(x,t)\neq u(y,t)\right\}$ precisely and $z(x,y,t)\in[-1,1]$ in the set $\left\{\mathbb{R}^{N}\times\mathbb{R}^{N}\times[0,T]\,:\,u(x,t)= u(y,t)\right\}.$ In \cite{Li-2025}, Li et al. proved the existence and uniqueness of both renormalized and entropy solutions for the fractional $1$-Laplacian equation, and proved the equivalence between renormalized and entropy solutions to the problem \eqref{nonlocal case parabolic problem} with $u_{0}\in L^{1}(\Omega)$ and $f\in L^{1}\left(\Omega\times(0,T]\right)$. In \cite{DLi-2024}, using the Rothe's time discritization method, Li et al. proved the existence of weak solution to the problem \eqref{nonlocal case parabolic problem} assuming $f\in L^{\frac{N}{s}}\left(\Omega\times(0,T]\right)$ and $u_{0}\in W_{0}^{s,1}(\Omega)\cap L^{2}(\Omega)$ and established that the weak solution maintains $\frac{1}{2}$-H\"older continuity in time.
 
Motivated by the above developments, this paper aims to make 
a further contribution to the theory of evolution equations driven by the fractional $1$-Laplacian. Despite the recent progress, the available results for problem \eqref{main:problem} remain notably limited. To be precise, the existing literature on the parabolic fractional $1$-Laplacian has so far been largely limited to source terms of the form $f = f(x,t)$, as in the foundational works of Maz\'on et al. \cite{Mazon-2016}, the subsequent contributions of Li et al. \cite{Li-2024, Li-2025, DLi-2024}, and Hauer et al. \cite{Hauer-2019} considered a source term of the form $f(x,u)$ for the fractional $1$-Laplacian, but under the restrictive assumption that $f$ is Lipschitz continuous in the second variable uniformly with respect to the first variable.

To the best of our knowledge, the study of \eqref{main:problem} with a general Carath\'eodory source term $f(x,u)$ with subcritical growth, together with a systematic analysis of solution behavior across the low, critical and high initial energy levels, remains unexplored. Motivated by this, we investigate the existence of global weak and strong solutions and the qualitative behavior of solutions under suitable assumptions on $f$. A crucial role in our arguments is played by uniform estimates of potential well depth, convergence of potential well depths and critical points in Nehari-manifold analysis. 

This paper is organized as follows. In section \ref{Main results}, we outline the definitions and properties of fractional Sobolev spaces, introduce relevant notation, and state the main results. In section \ref{Nehari functional and Invariant sets}, we prove several results related to potential well theory and recall some results from the subdifferential theory. In section \ref{Regularization and asymptotic analysis}, we study the regularization scheme used to construct the global weak solution of \eqref{main:problem} and derive uniform estimates for the potential depth $d_{p}$. In section \ref{Global existence of weak/strong solutions}, we establish the global existence of a weak solution and strong solution for low and critical initial energy. In section \ref{Wellposedness and Dynamics in low dimension}, we prove the local existence of strong solutions for $N<2s$ utilizing the sub-differential approach and examine the possibility of finite time blow-up and finite time extinction concerning problem \eqref{main:problem}. In section \ref{On the case of high initial energy}, for high initial energy case, we establish the sufficient condition depending on the initial data for the blow-up and extinction of strong solution. 
 \section{Main results}
 \label{Main results}
 In this section, we begin with introducing the function spaces, notion of solution, and the associated notations. Let $s\in(0,1)$, $1\leq p<\infty$ and the fractional critical exponent $p_{s}^{*}$ be defined as $p_{s}^{*}=\frac{Np}{N-sp}$ if $sp<N$ and $p_{s}^{*}=\infty$, otherwise. Define the Banach space
 $$W^{s,p}(\mathbb{R}^{N})=\left\{u\in L^{p}(\mathbb{R}^{N})\,|\, \frac{|u(x)-u(y)|}{|x-y|^{\frac{N}{p}+s}}\in L^{p}(\mathbb{R}^{N}\times\mathbb{R}^{N}) \right\}$$
 endowed with the norm
 $$\|u\|_{W^{s,p}(\mathbb{R}^{N})}:=\left(\int_{\mathbb{R}^{N}}|u|^{p}\,dx+\int_{\mathbb{R}^{N}}\int_{\mathbb{R}^{N}}\frac{|u(x)-u(y)|^{p}}{|x-y|^{N+sp}}\,dx\,dy\right)^{\frac{1}{p}}.$$
 Let $\Omega$ be an open subset of $\mathbb{R}^{N}$ and $Q= (\mathbb{R}^N \times\mathbb{R}^N)\setminus(\mathcal{C}\Omega\times \mathcal{C}\Omega), \quad \mathcal{C}\Omega=\mathbb{R}^N \setminus \Omega$. Denote
 $$W_{0}^{s,p}(\Omega)=\left\{u\,|\,u\in L^{p}(\Omega), u=0 \,\text{ in }\,\mathcal{C}\Omega,\,\,  \frac{|u(x)-u(y)|}{|x-y|^{\frac{N}{p}+s}}\in L^{p}(Q) \right\}.$$
 The space $W_{0}^{s,p}(\Omega)$ is a normed linear subspace of $W^{s,p}(\mathbb{R}^N)$ equipped with the norm
 $$\|u\|_{W_{0}^{s,p}(\Omega)}=\left(\iint_{Q}\frac{|u(x)-u(y)|^{p}}{|x-y|^{N+sp}}\,dx\,dy\right)^{\frac{1}{p}}.$$
Here the function \(f : \Omega \times \mathbb{R} \to \mathbb{R}\) satisfies the following conditions:
\begin{enumerate}[leftmargin=1.5cm,label=\textnormal{($f_0$)},ref=\textnormal{$f_0$}]
    \item \label{Assump:f_0}  There exists $r\in(1,\frac{N}{N-s})$ and $C>0$ such that 
    \[
     |f(x,t)|\leq C|t|^{r-1}\quad \text{for all } \, (x,t) \in \Omega \times \mathbb{R}.
    \]
\end{enumerate}
\begin{enumerate}[leftmargin=1.5cm,label=\textnormal{($f_1$)},ref=\textnormal{$f_1$}]
    \item \label{Assump:f_1} There exists a constant $\theta>1$, such that 
    \[
    0<\theta F(x,t) \leq t f(x,t)\quad\text{for all }\, (x,t) \in \Omega \times \mathbb{R}\setminus\{0\},
    \]
    where \(F(x,t) := \int_{0}^{t} f(x,\tau) \, d\tau\).
\end{enumerate}

\begin{enumerate}[leftmargin=1.5cm,label=\textnormal{($f_2$)},ref=\textnormal{$f_2$}]
    \item \label{Assump:f_2} \(f(x, \cdot) \in C^1(\mathbb{R})\) for all \(x \in \Omega\) and there exists a constant $\Theta>1$ such that 
    \[
   (\Theta-1)f(x,t) t < t^{2} f'(x,t) \quad \text{for all } \,  t \in \mathbb{R}\setminus\{0\}.
    \]
\end{enumerate}
\begin{remark}
Following the same arguments as in \cite[Lemma 2.16 (ii)]{Arora-2025} and using \eqref{Assump:f_1}, there exists $B>0$ such that 
\begin{equation}\label{F:bounds}
         F(x,t) \geq B|t|^{\theta}, \quad \text{and} \quad tf(x,t)\geq B\theta|t|^{\theta} \quad \text{for all } \,\, |t|\geq 1\ \mbox{and }x\in \Omega.
\end{equation}    
\end{remark}
Next, we recall the following continuous and compact embedding properties of the fractional Sobolev spaces $W^{s,1}_0(\Omega)$, which will be used throughout this paper. These results follow from the standard fractional Sobolev embedding theory; see \cite{Valdinoci-2012}.
\begin{proposition}\label{prop:embedding}
Let $\Omega \subset \mathbb{R}^N$ be a bounded open set with Lipschitz 
boundary, and let $s \in (0,1)$. Then the following embedding results hold:
\begin{enumerate}
    \item[\textnormal{(i)}] The embedding $W^{s,1}_0(\Omega) \hookrightarrow L^r(\Omega)$ is continuous for all $r \in \left[1, 1^{*}_{s}\right]$, and there exists a 
    constant $C_{Sob} = C_{Sob}(N,s,\Omega) > 0$ such that
    \[
        \|u\|_{L^r(\Omega)} \leq C_{Sob} \|u\|_{W^{s,1}_0(\Omega)} 
        \quad \text{for all }\, u \in W^{s,1}_0(\Omega).
    \]
    \item[\textnormal{(ii)}] The embedding $W^{s,1}_0(\Omega)\hookrightarrow L^r(\Omega)$ is compact for all $r \in \left[1, 1^{*}_{s}\right)$.
\end{enumerate}
\end{proposition}
To the stationary counterpart of the problem \eqref{main:problem},  we introduce the associated the energy functional $E : W_{0}^{s,1}(\Omega) \to \mathbb{R}$, defined as: 
\[
    E(u) := \iint_{Q} \frac{|u(x)-u(y)|}{|x-y|^{N+s}}\,dx\,dy - \int_{\Omega} F(x,u)\,dx.
\]
From conditions \eqref{Assump:f_0}, \eqref{Assump:f_1}  and Proposition \ref{prop:embedding}(i), it follows that the energy functional $E$ is well-defined. Next, we define the Nehari functional $I : W_{0}^{s,1}(\Omega) \to \mathbb{R}$ and the Nehari manifold as
\[
   I(u) := \iint_{Q} \frac{|u(x)-u(y)|}{|x-y|^{N+s}}\,dx\,dy - \int_{\Omega} f(x,u)u\,dx,
   \]
   \[
    \mathcal{N} := \left\{ u \in W_{0}^{s,1}(\Omega) \setminus \{ 0 \} \mid I(u) = 0 \right\}.
\]
The potential well $W$ and its corresponding set $V$ is defined as
\begin{equation}\label{invaria:sets}
        W := \left\{ u \in W_{0}^{s,1}(\Omega) \mid I(u) > 0, E(u) < d \right\} \cup \{0\} \quad \text{and} \quad
    V := \left\{ u \in W_{0}^{s,1}(\Omega) \mid I(u) < 0, E(u) < d \right\},
\end{equation}
where the depth $d$ of the potential well $W$ is defined as
\[
    d := \inf_{u \in \mathcal{N}} E(u).
\]
By $\|u\|_{s,p,2}$, we denote the usual norm in $W_{0}^{s,p}(\Omega)\cap L^{2}(\Omega)$ by
$$\|u\|_{s,p,2}=\|u\|_{W_{0}^{s,p}(\Omega)}+\|u\|_{L^{2}(\Omega)}.$$
Next, we introduce two notions of solutions according to the energy dissipation they satisfy; weak solution, for which energy inequality is satisfied, and strong solutions for which corresponding energy equality is satisfied.
\begin{definition}
    \label{Def:weak solution}
A function \( u \in L^{\infty}(0,T;W_{0}^{s,1}(\Omega)\cap L^{2}(\Omega)) \) with \( u_{t} \in L^{2}(0,T;L^{2}(\Omega)) \) is said to be a weak solution of the problem \eqref{main:problem} if 
\begin{enumerate}
    \item \( u(\cdot,0) = u_{0} \) a.e in $\Omega$.
    \item There exists $\eta(\cdot,\cdot,t)\in L^{\infty}(\mathbb{R}^{N}\times\mathbb{R}^{N})$ such that $\eta(x,y,t)=-\eta(y,x,t)$ for almost all $(x,y)\in \mathbb{R}^{N}\times\mathbb{R}^{N},\,\|\eta(\cdot,\cdot,t)\|_{L^{\infty}(\mathbb{R}^{N}\times\mathbb{R}^{N})}\leq 1$, 
    $$ \eta(x,y,t)\in \operatorname{sign}(u(x,t)-u(y,t)) \quad a.e. \,(x,y,t)\in \mathbb{R}^{N}\times\mathbb{R}^{N}\times [0,T], $$ and following equality holds:
\begin{equation}
\label{definition of weak sol u1}
\begin{aligned}
   \int_{\Omega}u_{t}\phi \,dx+\iint_{Q}\eta(x,y,t)\frac{\left(\phi(x)-\phi(y)\right)}{|x-y|^{N+s}}\,dx\,dy=\int_{\Omega}f(x,u)\phi\,dx
\end{aligned}
\end{equation}
for all
$\phi\in W_{0}^{s,1}(\Omega)\cap L^{2}(\Omega)$ and a.e. $t\in[0,T]$.
\item  $u \in C(0, T; L^q(\Omega))$, where $q\in [1,1^{*}_{s})$ and the following energy  relation is satisfied:
\begin{equation}
\label{sol: u}
\begin{aligned}
    \int_{0}^{t} \|u_t(\cdot, \tau)\|^{2}_{L^{2}(\Omega)} \, d\tau + E(u(\cdot,t))
    &\leq E(u_{0}) \quad \text{a.e. } t \in [0,T).
\end{aligned}
\end{equation}

\end{enumerate}
\end{definition}
\begin{definition}
\label{def strong solution}
A  weak solution $u$ of the problem \eqref{main:problem} is said to be a strong solution if the following energy conservation law holds
\begin{equation}
    \label{strong solution}
\int_{0}^{t} \|u_t(\cdot, \tau)\|^{2}_{L^{2}(\Omega)} \, d\tau + E(u(\cdot,t))
    = E(u(\cdot,0)),
\end{equation}
for any time interval $[0,t]\subset [0,T).$
\end{definition}

\begin{definition}[Maximal existence time]
    Let \( u \) be a weak or strong solution of the problem \eqref{main:problem}. We define the maximal existence time \( T_{\max} \) of \( u \) as follows:
    \begin{enumerate}
        \item If \( u \) exists for all \( 0 \leq t < +\infty \), then \( T_{\max} = +\infty \).
        \item If there exists \( t_0 \in (0, +\infty) \) such that \( u \) exists for all \( t \in (0, t_0) \) but does not exist at \( t = t_0 \) in the sense that 
        \[
        \|u(\cdot, t)\|_{L^{2}(\Omega)} \to +\infty \quad \text{as} \quad t \to t_0^-,
        \]
        then \( T_{\max} = t_0 \).
    \end{enumerate}
\end{definition}
To state our main results, let $q>1$ and define $s_q:=N+s-\frac{N}{q},$
where $s_q\in(s,1)$.
The first result corresponds to the global existence of weak solution for low initial energy $E(u_{0}) < d$ and $I(u_{0}) >0$.

\begin{theorem}
\label{Main Theorem low initial energy}
     Let conditions \eqref{Assump:f_0}-\eqref{Assump:f_2} hold and $u_{0}\in W_{0}^{s_{q},q}(\Omega)\cap L^{2}(\Omega)$. If $E(u_{0})<d$ and $I(u_{0})>0$, then problem \eqref{main:problem} admits a global weak solution in the sense of Definition~\ref{Def:weak solution}
\end{theorem}
We establish the existence of a global weak solution through an approximation scheme involving fractional $p$-Laplacian evolution problems with $p>1$. More precisely, by combining the Galerkin method with potential well theory, we first obtain global weak solutions $u^{(p)}$ to the corresponding fractional $p$-Laplacian problems. A key step in passing to the limit as $p\to1^{+}$ is to derive estimates that remain uniform with respect to $p$. To this end, we establish new convergence estimates for the associated energy and Nehari functionals (see Lemma~\ref{strong convergence u_0}), together with uniform convergence estimates for the critical points $\lambda_{p_n}$ arising from the Nehari manifold analysis and for the corresponding potential well depths $d_{p_n}$ (see Lemmas~\ref{boundedness of lambda} and~\ref{Dpn converges}). These estimates play a crucial role in transferring the potential well conditions from the approximating fractional $p_n$-Laplacian problems to the limiting fractional $1$-Laplacian problem. Finally, by exploiting the resulting uniform estimates and suitable compactness arguments, we pass to the limit along a sequence $p_n\to1^{+}$ and show that $u^{(p_n)}$ converges to a global weak solution $u$ of \eqref{main:problem} in the sense of Definition~\ref{Def:weak solution}.

Next, we weaken the assumptions imposed in the preceding result by considering initial data $u_0$ in the natural energy space $W_{0}^{s,1}(\Omega)\cap L^{2}(\Omega)$. Under an additional growth restriction on the nonlinearity, we establish the global existence of strong solutions. More precisely, we introduce the following assumption:
\begin{enumerate}[leftmargin=1.5cm,label=\textnormal{(${f_3}$)},ref=\textnormal{${f_3}$}]
\item \label{Assump:f_3} $r<1+\frac{N}{2(N-s)}.$
\end{enumerate}

\begin{theorem}\label{Main Theorem low initial energy-strong}
Let conditions \eqref{Assump:f_0}--\eqref{Assump:f_3} hold and $u_{0}\in W_{0}^{s,1}(\Omega)\cap L^{2}(\Omega)$. If $E(u_{0})<d$ and $I(u_{0})>0$, then problem \eqref{main:problem} admits a global strong solution.
\end{theorem}

The additional condition \eqref{Assump:f_3} provides the integrability required to construct an approximating sequence of global strong solutions. The passage to the limit relies on uniform estimates for these approximating solutions, together with the continuity in time of the energy functional $E(u(\cdot,t))$ and the Nehari functional $I(u(\cdot,t))$ (see Lemma  \ref{continuity of strong solution}). Moreover, the invariance of the corresponding potential well along the evolution (see Lemma ~\ref{u in W delta}) ensures that the solution remains in the stable region determined by the initial data. Combining these properties with the regularity and continuity results, we establish the global existence of a strong solution in the sense of Definition~\ref{def strong solution}. We emphasize that, even in the corresponding local $1$-Laplacian problem \cite{Tahir-2022}, the global existence of strong solutions for initial data $u_0$ belonging merely to the natural energy space has remained an open problem. Moreover, the results of \cite{Tahir-2022} for the local $1$-Laplacian are restricted to the low initial energy case $E(u_0)<d$, and the critical initial energy case $E(u_0)=d$ is left unaddressed.
 
Next, we state our existence result for global weak and strong solution at the critical level $E(u_0)=d$ and $I(u_0) \geq 0$, together with the identification of a finite time vanishing phenomenon if the Nehari functional becomes zero at some time $t^\ast>0$. 
   
    \begin{theorem}\label{main-exist-critical-weak}
    Let conditions \eqref{Assump:f_0}–\eqref{Assump:f_2} hold and \( u_{0} \in W_{0}^{s_{q},q}(\Omega)\cap L^{2}(\Omega) \). If \( E(u_{0}) = d \) and \( I(u_{0}) \geq 0 \), then problem \eqref{main:problem} admits a global weak solution.
    \begin{enumerate}
        \item[\textnormal{(i)}] If there exists \( t^{*} > 0 \) such that \( I(u(\cdot,t)) > 0 \) for \( 0 < t < t^{*} \) and \( I(u(\cdot,t^{*})) = 0 \), then there exists a weak solution \( u(\cdot,t) \) which vanishes in finite time \( t^{*} \).
        \item[\textnormal{(ii)}] If, in addition, \eqref{Assump:f_3} holds, then problem \eqref{main:problem} admits a global strong solution for \( u_0 \in W_{0}^{s,1}(\Omega)\cap L^{2}(\Omega)\). 
    \end{enumerate}
\end{theorem}
The critical initial energy case $E(u_0)=d$ requires a separate argument, since the initial datum $u_0$ no longer belongs to the potential well $W$. To address this issue, we approximate $u_0$ by the scaled initial data
$$
u_{0k}:=\lambda_k u_0, \qquad \lambda_k=1-\frac{1}{k},
$$
which belong to $W$ by the properties of the Nehari functional. Consequently, each approximate problem falls within the low initial energy regime already treated in Theorem~\ref{Main Theorem low initial energy}. Passing to the limit as $k\to\infty$, we obtain a global weak solution corresponding to the critical initial energy level $E(u_0)=d$.

Finally, the existence of a global strong solution can be established by following the same approximation and limiting arguments as in the weak solution case, together with Theorem~\ref{Main Theorem low initial energy-strong}.

     \begin{corollary}
     \label{Uniqueness of weak/strong solution}
        Let conditions \eqref{Assump:f_0}-\eqref{Assump:f_3} hold. If $E(u_{0})\leq d$, $I(u_{0})\geq 0$ and $f(x,\cdot)$ is uniformly Lipschitz for almost all $x\in\Omega$, then the corresponding  weak and strong solutions of \eqref{main:problem} are unique.
     \end{corollary}
    We further investigate the local existence and qualitative behavior of strong solutions in the low dimensional regime $N<2s$. In this case, the compact embedding of $W^{s,1}_0(\Omega)\hookrightarrow L^2(\Omega)$ allows us to employ a subdifferential formulation of the evolution problem in $L^2(\Omega)$. We establish local existence of strong solutions without imposing any restriction on the initial energy and sign of Nehari functional. The existence of a strong solution plays a crucial role in studying the asymptotic behavior and finite time blow-up.  
    \begin{theorem}\label{loc}
    Let conditions \eqref{Assump:f_0}--\eqref{Assump:f_3} hold, $N< 2s$ and \( u_{0}\in W_{0}^{s,1}(\Omega) \). Then, there exists a \( T>0 \) such that the problem \eqref{main:problem} admits a strong solution \( u \) on \( \Omega\times[0,T] \) in the sense of Definition \ref{def strong solution}.
\end{theorem}
We now state the qualitative behavior of strong solutions in the low dimensional regime for low and critical initial energy.
\begin{theorem}\label{low energy asymptotic behaviour}
     Let conditions \eqref{Assump:f_0}--\eqref{Assump:f_3}, hold and \( u_0 \in W_{0}^{s,1}(\Omega)\). Assume that $N \leq 2s$, then there exist constants \( \delta' \in (0,1) \) and \( C_{Sob} > 0 \) such that the following assertions hold:
     \begin{enumerate}
        \item[\textnormal{(i)}] If $E(u_{0})<d$ and $I(u_{0})>0$, then for all $t\geq 0$,
         \begin{equation*}
   \|u(\cdot, t)\|_{L^2(\Omega)} \leq \left(\|u_{0}\|_{L^{2}(\Omega)}+(\delta'-1)C_{Sob}t\right)_{+},
    \end{equation*}
    where \( (z)_{+}:=\max\{z,0\} \). In particular, the solution vanishes in finite time $ t^{*}=\frac{\|u_{0}\|_{L^{2}(\Omega)}}{(1-\delta')C_{Sob}}.$
        \item[\textnormal{(ii)}] If $E(u_0)=d$. Then, either $I(u(\cdot,t))=0$ and
        $$\|u(\cdot,t)\|_{L^{2}(\Omega)}=\|u_{0}\|_{L^{2}(\Omega)} \quad \text{for all } t \in (0, \infty),$$
        or there exists $t_{0}>0$ such that $I(u(\cdot,t_{0}))>0$ and
        \begin{equation*}
   \|u(\cdot, t)\|_{L^2(\Omega)} \leq \left(\|u(\cdot,t_0)\|_{L^{2}(\Omega)}+(\delta'-1)C_{Sob}(t-t_{0})\right)_{+}
    \end{equation*} 
    for all $t\geq t_{0}.$ In particular, the solution vanishes in finite time  
    $t^{**}=t_{0}+\frac{\|u(\cdot,t_0)\|_{L^{2}(\Omega)}}{(1-\delta')C_{Sob}}.$
        \end{enumerate}
    \end{theorem}
    Theorem \ref{low energy asymptotic behaviour} shows that in the low dimensional regime $N\leq 2s$, the
sign of the Nehari functional $I(u_0)$ determines not just global existence but a
sharper phenomenon: the strong solution vanishes identically after a finite time. This finite-time extinction is obtained by combining the invariance of the potential well $W$ with the continuous embedding $W_{0}^{s,1}(\Omega)\hookrightarrow L^2(\Omega)$, which allows us to derive a differential estimate for the $L^2$-norm of the solution. The resulting estimate yields finite-time extinction in the case $E(u_0)<d$. At the critical initial energy level $E(u_0)=d$, the same argument is applied after identifying the time $t_{0}>0$ at which the solution enters in the potential well $W$.

\begin{theorem}
   \label{Blow-up thm low initial energy}
    Let conditions \eqref{Assump:f_0}--\eqref{Assump:f_3} hold, $N<2s$ and $u_{0} \in W_{0}^{s,1}(\Omega)$. If $E(u_{0}) \leq d$ and $I(u_{0}) < 0$, then the strong solution $u$ of problem \eqref{main:problem} exhibits finite time blow-up in the sense that there exists $T^{\ast}>0$ such that
\begin{equation}
    \label{blow up of strong sol u}
    \lim_{t \to T^{*}} \int_{0}^{t} \|u(\cdot, \tau)\|^{2}_{L^{2}(\Omega)} \, d\tau = +\infty.
\end{equation}
More precisely, the following assertions hold true
 \begin{enumerate}
        \item[\textnormal{(i)}] If $E(u_{0})<d$, then $T^\ast= \frac{4\|u_0\|_{L^2(\Omega)}^2(\Theta-1)}{\Theta(\Theta-2)^{2}(d-E(u_{0}))}$, where $\Theta$ is given in \eqref{Assump:f_2}.
        \item[\textnormal{(ii)}] If $E(u_{0})=d$, then there exists a finite time $T^{\ast\ast}>0$ such that the strong solution $u$ blows up in sense of \eqref{blow up of strong sol u}.
       
        \end{enumerate}
\end{theorem}
The proof of above result follows from the classical concavity method of Levine \cite[Theorem I]{Levine-1973}, adapted to the fractional $1$-Laplacian setting, using the invariance of the corresponding set $V$, established in Lemma \ref{u in W delta}, and a suitable auxiliary function of the solution's $L^2$-norm, one shows that this function cannot remain finite for all time, which forces blow-up before some explicit finite time. The assumption $N<2s$ is essential, since it underlies the local existence of the strong solution on which the blow-up argument is built. 

Next. we state the result concerning the case of high initial energy $E(u_{0}) > d$, we introduce the following sets for the strong solution $u$ to the problem \eqref{main:problem}:  
$$\mathcal{N}_{+}:=\{u\in W_{0}^{s,1}(\Omega)\,|\,I(u)> 0\},\qquad \mathcal{N}_{-}:=\{u\in W_{0}^{s,1}(\Omega)\,|\,I(u)< 0\}$$
and
$$O_{\zeta}:=\{u\in W_{0}^{s,1}(\Omega)\,|\,E(u) <\zeta\}.$$
By the definition of $E$, $\mathcal{N}$, $O_{\zeta}$ and $d,$ we get
$$\mathcal{N}_{\zeta}:=\mathcal{N}\cap O_{\zeta}= \{u\in\mathcal{N}\,|\,E(u)<\zeta\}\neq\emptyset \qquad \text{  for all    }\zeta >d.$$
For $\zeta>d$, define
$$\lambda_{\zeta}:=\inf \{\|u\|_{L^{2}(\Omega)}\,|\,u\in\mathcal{N}_{\zeta}\}, \qquad \Lambda_{\zeta}:=\sup \{\|u\|_{L^{2}(\Omega)}\,|\,u\in\mathcal{N}_{\zeta}\}.$$
\[
\mathcal{B} = \left\{ u_{0} \in W_{0}^{s,1}(\Omega)\cap L^{2}(\Omega) \mid \text{ the strong solution } u \text{ blows up (in the } L^2\text{-norm) in finite time} \right\},
\]
\[
\mathcal{G}_{0} = \left\{ u_{0} \in W_{0}^{s,1}(\Omega)\cap L^{2}(\Omega)\mid \text{ the strong solution } u \text{ satisfies } u(\cdot,t) \to 0 \text{ in } W_{0}^{s,1}(\Omega) \text{ as } t \to \infty \right\}.
\]
\begin{theorem}
    \label{Main theorem high initial energy}
    Let conditions \eqref{Assump:f_0}--\eqref{Assump:f_2} hold and $u_{0}\in W_{0}^{s,1}(\Omega).$ Assume further that $N<2s$. If $E(u_{0})>d$, then the following statements hold:
    \begin{enumerate}
    \item[\textnormal{(i)}] If $u_{0}\in \mathcal{N}_{+}$ and $\|u_{0}\|_{L^{2}(\Omega)}\leq \lambda_{E(u_{0})}$, then $u_{0} \in \mathcal{G}_{0}.$
    \item[\textnormal{(ii)}] If $u_{0}\in \mathcal{N}_{-}$ and $\|u_{0}\|_{L^{2}(\Omega)}\geq \Lambda_{E(u_{0})}$, then $u_{0} \in \mathcal{B}.$
    \end{enumerate}
\end{theorem}
For the high initial energy case, we have derived the sufficient condition  on the $u_0$ for the blow-up and finite time extinction for the strong solution to the problem \eqref{main:problem}. In contrast to the low initial energy case, the sign of the Nehari functional alone does not determine the behavior of solution when $E(u_0)>d$, an additional smallness or largeness condition on $\|u_0\|_{L^2(\Omega)}$, relative to the $\lambda_{E(u_0)}$ and $\Lambda_{E(u_0)}$, is required to guarantee extinction or blow-up, respectively. Existence of such conditions on the initial data are justified in Corollaries \ref{high initial blow up} and \ref{finthm}.
 \section{Nehari functional and Invariant sets}
 \label{Nehari functional and Invariant sets}
In this section, we establish several preliminary results concerning the properties of the Nehari functional $I$, the associated invariant sets $W$ and $V$ (defined in \eqref{invaria:sets}), and the realization of the fractional $1$-Laplacian as the subdifferential of a proper, convex, and lower semicontinuous functional. These results are fundamental to the subsequent analysis, particularly for establishing invariance properties and studying the dynamics of weak and strong solutions of the problem \eqref{main:problem}.

 \subsection{Analysis on Nehari functional}
\begin{lemma}
\label{Lemma:2.3}
    Let $f$ satisfy the conditions \eqref{Assump:f_0}--\eqref{Assump:f_2}. Then, for any $u \in W_{0}^{s,1}(\Omega)$ with $\|u\|_{W_{0}^{s,1}(\Omega)} \neq 0$, we have
\begin{enumerate}
    \item[\textnormal{(i)}] $\lim\limits_{\lambda \to 0^{+}} E(\lambda u) = 0$, and $\lim\limits_{\lambda \to +\infty} E(\lambda u) = -\infty$.
    \item[\textnormal{(ii)}] There exists a unique $\lambda = \lambda_{\ast}(u) > 0$ such that $\frac{dE(\lambda u)}{d\lambda} \big|_{\lambda = \lambda_{\ast}} = 0$. Moreover, $E(\lambda u)$ is increasing on $0 < \lambda \leq \lambda_{\ast}$, decreasing on $\lambda_{\ast} \leq \lambda < \infty$, and attains its maximum at $\lambda = \lambda_{\ast}$.
    \item[\textnormal{(iii)}] $I(\lambda u) \geq 0$ for $0 < \lambda \leq \lambda_{\ast}$, $I(\lambda u) < 0$ for $\lambda_{\ast} < \lambda < \infty$, and $I(\lambda_{\ast} u) = 0$.
\end{enumerate}
\end{lemma}
\begin{proof}
Let $\lambda>0$ and $u \in W_{0}^{s,1}(\Omega)$ with $\|u\|_{W_{0}^{s,1}(\Omega)} \neq 0$. Then, by \eqref{F:bounds} we obtain
\begin{align*}
    E(\lambda u) & \leq \lambda \|u\|_{W_{0}^{s,1}(\Omega)}- B\lambda^{\theta}\int_{\Omega\cap\{|u|\geq \frac{1}{\lambda}\}}|u|^{\theta}\,dx \label{Eq:max bound}.
\end{align*}
It follows, due to \eqref{Assump:f_1} and the continuity of $E$, that
\begin{equation}
    \label{lim E(lambda u)}
    \lim_{\lambda \to +\infty}E(\lambda u)=-\infty, \quad \text{and} \quad \lim_{\lambda \to 0^{+}}E(\lambda u)=E(0)= 0.
\end{equation}
Hence, the claim in (i). Next, we show (ii). In light of condition \eqref{Assump:f_0} , we obtain
\begin{align*}
E(\lambda u)&\geq \iint_{Q} \frac{|\lambda(u(x)-u(y))|}{|x-y|^{N+s}}\,dx\,dy -\int_{\Omega}C|\lambda u|^{r}\,dx \geq \lambda \|u\|_{W_{0}^{s,1}(\Omega)}- C\lambda^{r}\|u\|_{L^{r}(\Omega)}^{r},\label{Eq:min bound}
\end{align*}
which further implies that
\begin{equation}
    \label{E(lmabda u)increasing}
    E(\lambda u)>0\,\,\text{ for all }\,\, \lambda\in\left(0, \lambda_u\right), \quad \lambda_u:= \|u\|^\frac{1}{r-1}_{W_{0}^{s,1}(\Omega)} C^\frac{-1}{r-1} \|u\|^\frac{-r}{r-1}_{L^{r}(\Omega)} .
\end{equation}
Therefore, by using \eqref{lim E(lambda u)} and \eqref{E(lmabda u)increasing}, there exists $\lambda=\lambda_{\ast}(u)>0$ such that $\frac{dE(\lambda u)}
   {d\lambda}|_{\lambda=\lambda_{\ast}}=0,$ namely,
   \begin{equation*}
           \iint_{Q}\frac{|u(x)-u(y)|}{|x-y|^{N+s}}\,dx\,dy =\int_{\Omega}f(x,\lambda_{\ast} u)u\,dx.
   \end{equation*}
   On the other hand, by using \eqref{Assump:f_2}, we obtain
   \begin{equation}
       \label{existence of lambda d lambda}
   \begin{aligned}
      \frac{d^{2}E(\lambda u)}{d\lambda^{2}}|_{\lambda=\lambda_{\ast}}&= -\int_{\Omega}f'(x,\lambda_{\ast} u) u^{2}\,dx = -\frac{1}{(\lambda_{\ast})^{2}}\int_{\Omega}f'(x,\lambda_{\ast} u) (\lambda_{\ast} u)^{2}\,dx <0.
     \end{aligned}
   \end{equation}
Next, we prove that $\lambda = \lambda_{\ast}(u)$ is uniquely determined. We proceed by contradiction. Assume that there exist two distinct roots, say $\lambda_1$ and $\lambda_2$, such that

$$\frac{dE(\lambda u)}{d\lambda}|_{\lambda=\lambda_{1}}=0\text{, }\frac{d^{2}E(\lambda u)}{d\lambda^{2}}|_{\lambda=\lambda_{1}}<0, \quad \text{and} \quad \frac{dE(\lambda u)}{d\lambda}|_{\lambda=\lambda_{2}}=0\text{, }\frac{d^{2}E(\lambda u)}{d\lambda^{2}}|_{\lambda=\lambda_{2}}<0.$$
Thus, there exists $\lambda_3$ such that $\lambda_1 < \lambda_3 < \lambda_2$ and $E(\lambda_3 u)$ is the minimum of $E(\lambda u)$ on the interval $[\lambda_1, \lambda_2]$. Therefore, we have
$$
\frac{dE(\lambda u)}{d\lambda} \bigg|_{\lambda = \lambda_3} = 0, \quad \frac{d^2 E(\lambda u)}{d\lambda^2} \bigg|_{\lambda = \lambda_3} \geq 0,
$$
which leads to a contradiction with \eqref{existence of lambda d lambda}. This proves that $E(\lambda u)$ is increasing on $0 < \lambda \leq \lambda_{\ast}$ and decreasing on $\lambda_{\ast} \leq \lambda < \infty$. Hence, the claim in (ii). Finally, (iii) follows from (ii) and by noting the fact that $ I(\lambda u)=\lambda\left(\frac{dE(\lambda u)}{d\lambda}\right).$
\end{proof}

For any $\delta>0$, we define the modified Nehari functional and Nehari manifold as follows:\\
$$I_{\delta}(u):=\delta\iint_{Q}\frac{|u(x)-u(y)|}{|x-y|^{N+s}}\,dx\,dy-\int_{\Omega}f(x,u)u\,dx,$$
and
$$\mathcal{N}_{\delta}:=\{u\in W_{0}^{s,1}(\Omega)\setminus\{0\}\ |\ I_{\delta}(u)=0\}.$$
The corresponding modified potential well $W_{\delta}$ and its corresponding set is defined as
$$W_{\delta}:=\{u\in W_{0}^{s,1}(\Omega)\ |\ I_{\delta}(u)>0,E(u)< d(\delta)\}\cup{\{0\}}\ \text{and} \ V_{\delta}:=\{u\in W_{0}^{s,1}(\Omega)\ |\ I_{\delta}(u)<0,E(u)< d(\delta) \},$$
 where the depth of the modified potential well $W_{\delta}$ is defined as $$d(\delta)=\inf_{u\in\mathcal{N}_{\delta}}E(u).$$ By following the same arguments as in the proof of Lemma \ref{Lemma:2.3} for $I_{\delta}$, we have similar analysis of modified Nehari functional $I_\delta.$
 \begin{corollary}\label{cor}
   Let $f$ satisfy the conditions \eqref{Assump:f_0}--\eqref{Assump:f_2}. Then, for any $u \in W_{0}^{s,1}(\Omega)$ with $\|u\|_{W_{0}^{s,1}(\Omega)} \neq 0$ and $\delta>0 $ there exists a unique $\lambda_{\ast}=\lambda_{\ast}(\delta, u)$ such that
$I_{\delta}(\lambda u) \geq 0$ for $0 < \lambda \leq \lambda_{\ast}$, $I_{\delta}(\lambda u) < 0$ for $\lambda_{\ast} < \lambda < \infty$, and $I_{\delta}(\lambda_{\ast} u) = 0$.
\end{corollary}
Next, we show that depth $d$ of the potential well $W$ is positive and derive the asymptotics of the depth of the modified potential well $d(\delta)$ and its monotonicity properties, which further helps in deriving the sign properties of the modified Nehari functional $I_\delta.$ 
\begin{lemma}\label{Positive:depth}
           Let $f$ satisfy the conditions \eqref{Assump:f_0}-\eqref{Assump:f_1}. Then, the depth $d$ of the potential well $W$ is positive. Moreover, there exists a constant $C_{\theta}>0$ such that $d \geq C_{\theta}$.
\end{lemma}
\begin{proof}
    Let $u\in\mathcal{N}$. Then, from \eqref{Assump:f_0} and Proposition \ref{prop:embedding}(i), we obtain
 \begin{align}
         \label{lower bound of norm u}
   \|u\|_{W_{0}^{s,1}(\Omega)}=\iint_{Q} \frac{|u(x)-u(y)|}{|x-y|^{N+s}} \,dx\,dy &= \int_{\Omega}f(x,u)u\,dx \leq C\|u\|_{L^{r}(\Omega)}^{r}\leq S_1\|u\|_{W_{0}^{s,1}(\Omega)}^{r},
       \end{align}
        where
        $S_1:= C(C_{Sob})^{r}$ and $C$ is defined in \eqref{Assump:f_0}. Therefore, by using \eqref{Assump:f_1} and \eqref{lower bound of norm u} we infer that
       \begin{align*}
        E(u) & \geq {\iint_{Q}}\frac{|u(x)-u(y)|}{|x-y|^{N+s}}\,dx\,dy-\int_{\Omega}\frac{f(x,u)u}{\theta}\,dx = \iint_{Q} \frac{|u(x)-u(y)|}{|x-y|^{N+s}}\,dx\,dy-\frac{1}{\theta}{\iint_{Q}}\frac{|u(x)-u(y)|}{|x-y|^{N+s}}\,dx\,dy\\
        &=\left(1-\frac{1}{\theta}\right)\|u\|_{W_{0}^{s,1}(\Omega)}
        \geq \left(1-\frac{1}{\theta}\right)\left(\frac{1}{S_1}\right)^{\frac{1}{r-1}}.
\end{align*}
Finally, by taking infimum over $\mathcal{N}$, we obtain
$$d=\inf_{u\in\mathcal{N}}E(u)\geq \left(1-\frac{1}{\theta}\right)\left(\frac{1}{S_1}\right)^{\frac{1}{r-1}}=:C_\theta.$$
Thus, the desired result follows.
\end{proof}
\begin{lemma}
\label{prop:d(delta)}
  Let the conditions \eqref{Assump:f_0}--\eqref{Assump:f_2} hold. Then, the function $d(\delta)$ satisfies the following properties:
\begin{enumerate}
    \item[\textnormal{(i)}] $\lim\limits_{\delta \to 0^{+}} d(\delta) = 0$  \ and \ $\lim\limits_{\delta \to \infty} d(\delta) = -\infty$.
    \item[\textnormal{(ii)}] $d(\delta)$ is increasing on $0 < \delta \leq 1$, decreasing on $\delta \geq 1$, and attains its maximum, $d = d(1)$, at $\delta = 1$.
\end{enumerate}
\end{lemma}
\begin{proof}
By invoking Corollary \ref{cor}, for any $u\in W_{0}^{s,1}(\Omega)$ with $\|u\|_{W_{0}^{s,1}(\Omega)} \neq 0$ and $\delta> 0$, there exists a unique $\lambda_{\ast}=\lambda_{\ast}(\delta,u)$ such that $I_{\delta}(\lambda_{\ast} u)=0$. Thus,
$$\delta {\iint_{Q}}\frac{\lambda_{\ast}|u(x)-u(y)|}{|x-y|^{N+s}}\,dx\,dy =\int_{\Omega}f(x, \lambda_{\ast} u) \lambda_{\ast} u\,dx.$$
Therefore, for a fixed $u\in W_{0}^{s,1}(\Omega)$ with $\|u\|_{W_{0}^{s,1}(\Omega)} \neq 0$, to analyze the behavior of $\lambda_{\ast}(\delta, u)$ as $\delta \to 0^+$ and 
$\delta \to +\infty$, we define the auxiliary function $\eta:(0,\infty)\to \mathbb{R}$ as
\begin{equation*}
\eta(\alpha):=\frac{1}{\|u\|_{W_{0}^{s,1}(\Omega)}} \int_{\Omega}f(x, \alpha u) u\,dx.
\end{equation*}
Differentiating with respect to $\alpha$ and using \eqref{Assump:f_2}, we obtain
\begin{align}\label{prop:incre}
\eta'(\alpha)&=\frac{1}{\|u\|_{W_{0}^{s,1}(\Omega)}}\int_{\Omega}f'(x, \alpha u)u^{2}\,dx>0.
\end{align}
Hence, the map $\alpha \mapsto \eta(\alpha)$ is strictly increasing on $(0,+\infty)$ strictly increasing on $(0,+\infty)$. Next, we claim that
\begin{align}\label{prop:asymp}
    \lim_{\alpha \to 0^{+}}\eta(\alpha)=0 \quad \text{and} \quad \lim_{\alpha \to +\infty}\eta(\alpha)=+\infty.
\end{align}
Indeed, using \eqref{Assump:f_0} and \eqref{F:bounds}, we obtain
\begin{align*}
\eta(\alpha) \leq \frac{C \alpha^{r-1}}{\|u\|_{W_{0}^{s,1}(\Omega)}}\int_{\Omega}|u|^{r}\,dx \to 0 \text{ as } \alpha\to 0^{+},
\end{align*}
and
\begin{align*}
    \quad \eta(\alpha)&\geq \frac{B\theta}{\|u\|_{W_{0}^{s,1}(\Omega)}}\int_{\Omega\cap\{|u|\geq\frac{1}{\alpha}\}}\alpha^{\theta-1}|u|^{\theta}\,dx \to \infty \text{ as } \alpha\to +\infty.
    \end{align*}
Hence, the required claim. Combining \eqref{prop:incre} and \eqref{prop:asymp}, the map $\beta \mapsto \eta^{-1}(\beta)$ is strictly increasing on $(0,+\infty)$. Furthermore, together with $\lambda_{\ast}(\delta,u)u\in \mathcal{N}_{\delta}$ and $\lambda_{\ast}(\delta, u) = \eta^{-1}(\delta)$, it follows that the mapping 
\begin{equation*}
  \text{$\delta \mapsto \lambda_{\ast}(\delta, u)$ is strictly increasing on $(0, +\infty)$.}  
\end{equation*} Moreover,
$$\lim_{\delta \to 0^{+}}\lambda_{\ast}(\delta,u)=0 \quad \text{and} \quad \lim_{\delta \to +\infty}\lambda_{\ast}(\delta,u)=+\infty.$$
Hence, by using $\lambda_{\ast}(\delta,u) u \in \mathcal{N_{\delta}}$, the definition of $d(\delta)$ and Lemma \ref{Lemma:2.3}(i), we conclude that
$$0\leq \lim_{\delta \to 0^{+}}d(\delta)\leq \lim_{\delta \to 0^{+}}E\left(\lambda_{\ast}(\delta,u)u\right)=0,\quad \text{and} \quad \lim_{\delta \to +\infty} d(\delta)\leq \lim_{\delta \to +\infty} E\left(\lambda_{\ast}(\delta,u) u\right)=-\infty.$$
This completes the proof of assertion (i). To establish assertion (ii), it suffices to prove that for any \(0 < \delta' < \delta'' < 1\) or \(\delta' > \delta'' > 1\), and for any \(u \in \mathcal{N}_{\delta''}\), there exists a \(v \in \mathcal{N}_{\delta'}\) and a constant \(\epsilon(\delta', \delta'') > 0\) such that
\[
E(u) - E(v) \geq \epsilon(\delta', \delta'').
\]
Indeed, for \(u \in \mathcal{N}_{\delta''}\), we have \(I_{\delta''}(u) = 0\), which implies that  \(\lambda_{\ast}''=\lambda(u,\delta'') = 1\). By Corollary \ref{cor}, there exists a constant \(\lambda_{\ast}'=\lambda(\delta',u) > 0\) such that \(v = \lambda_{\ast}' u \in \mathcal{N}_{\delta'}\).
Let \(g(\lambda) = E(\lambda u)\). Then,
\begin{equation}\label{an}
\begin{split}
    \frac{dg(\lambda)}{d\lambda}&=\frac{d}{d\lambda}\left[\iint_{Q}\frac{\lambda(|u(x)-u(y)|)}{|x-y|^{N+s}}\,dx\,dy-\int_{\Omega}F(x,\lambda u)\,dx \right]\\
    &=\frac{1}{\lambda}\left[ {\iint_{Q}}\frac{\lambda(|u(x)-u(y)|)}{|x-y|^{N+s}}\,dx\,dy-\int_{\Omega}f(x,\lambda u)\lambda u\ dx\right]\\
    &=\frac{1}{\lambda}\left[(1-\delta''){\iint_{Q}}\frac{\lambda(|u(x)-u(y)|)}{|x-y|^{N+s}}\,dx\,dy +  I_{\delta''}(\lambda u) \right]\\
    &=\frac{1}{\lambda}\left[\lambda (1-\delta'')\|u\|_{W_{0}^{s,1}(\Omega)}+I_{\delta''}(\lambda u) \right].
        \end{split}
\end{equation}
Next, we consider two cases :\\
    \textbf{Case $1$}: $0<\delta'<\delta''<1$. Since $\delta\to \lambda_{\ast}(\delta,u)$ is increasing and $\lambda_{\ast}''=1$, then
    \begin{align*}
        E(u)-E(v)=g(1)-g(\lambda_{\ast}')&=\int_{\lambda_{\ast}'}^{1}\frac{dg(\lambda)}{d\lambda}d\lambda.
    \end{align*}
 Since $0<\delta'<\delta''$ we have $0<\lambda_{\ast}'<\lambda_{\ast}''=1$. By Corollary \ref{cor}, $I_{\delta''}(\lambda u) \geq 0$ for all $\lambda_{\ast}' < \lambda < 1$. Therefore, from \eqref{an}, we get
    \begin{align}
    \label{Epsilon delta constant}
        E(u)-E(v) \geq \int_{\lambda_{\ast}'}^{1}(1-\delta'')\|u\|_{W_{0}^{s,1}(\Omega)}d\lambda =(1-\delta'')\left(1-\lambda_{\ast}' \right)\|u\|_{W_{0}^{s,1}(\Omega)}.
    \end{align}
    Since $u\in \mathcal{N}_{\delta''}$, from \eqref{lower bound of norm u} and \eqref{Epsilon delta constant}, we obtain
    \begin{align*}
     E(u)-E(v)\geq \epsilon(\delta'',\delta'), \quad \text{where} \quad \epsilon(\delta'',\delta'):=\left(\frac{\delta''}{S_1}\right)^{\frac{1}{r-1}}(1-\delta'')(1-\lambda_{\ast}').
    \end{align*}
    \textbf{Case 2}: \(\delta' > \delta'' > 1\). As $\delta\to \lambda_{\ast}(\delta,u)$ is increasing, it follows that  \(\lambda_{\ast}' > \lambda >\lambda_{\ast}''=1\).
    Hence, by repeating the same arguments as in \textbf{Case 1}, we obtain
    we get
    $$ E(u)-E(v)>\epsilon(\delta',\delta''), \quad \text{where} \quad \epsilon(\delta',\delta''):=\left(\frac{\delta''}{S_1}\right)^{\frac{1}{r-1}}(\delta''-1)(\lambda(\delta')-1)$$
Consequently, since \( d(\delta) \) is continuous, increasing for \( 0 < \delta \leq 1 \), and decreasing for \( \delta \geq 1 \), it attains its maximum value at \( \delta = 1 \), where \( d(1) = d \).  
\end{proof}
\begin{lemma}
\label{Lemma: 2.5}
Suppose the assumptions of Lemma \ref{prop:d(delta)} are satisfied. Let \(u \in W_{0}^{s,1}(\Omega)\) with \(0 < E(u) < d\), and assume that \(\delta_1 < 1 < \delta_2\), where \(\delta_1\) and \(\delta_2\) satisfy the equation \(d(\delta) = E(u)\). Then, the sign of \(I_{\delta}(u)\) remains unchanged for \(\delta_1 < \delta < \delta_2\).
\end{lemma}
\begin{proof}
   Clearly, since  $E(u) > 0$,  it follows that  $ \|u\|_{W_{0}^{s,1}(\Omega)}\neq 0$.
 If the sign of \(I_{\delta}(u)\) changes for \(\delta_1 < \delta < \delta_2\), then by the continuity of $I_{\delta}(\cdot)$ there exists \(\overline{\delta} \in (\delta_1, \delta_2)\) such that \(I_{\overline{\delta}}(u) = 0\). By the definition of \(d(\delta)\), this implies \(E(u) \geq d(\overline{\delta})\). However, this contradicts the fact that \(E(u) = d(\delta_1) = d(\delta_2) < d(\overline{\delta})\), as established in Lemma~\ref{prop:d(delta)}(ii).
\end{proof}
\begin{lemma}
\label{u in W delta}
Assume that the conditions \eqref{Assump:f_0}--\eqref{Assump:f_1} hold. Let \(u\) be a global  strong solution of problem \eqref{main:problem}  in the sense of Definition \ref{def strong solution}. Then, $u\in C(0,\infty;W_{0}^{s,1}(\Omega))$, and the mappings
    \[
    t\mapsto E(u(\cdot,t))\  \text{and} \ t\mapsto I(u(\cdot,t)) \ \text{are continuous on}\ [0,\infty). 
    \]
Moreover, if \eqref{Assump:f_2} holds and the initial data \(u_0\) satisfies \(E(u_0) < d\) and \(E(u_0) = d(\delta_1) = d(\delta_2)\) for some \(\delta_1 < 1 < \delta_2\) and
\begin{enumerate}
    \item[\textnormal{(i)}] if \(I(u_0) > 0\), then \( u(\cdot,t) \in W_{\delta}\) for all \(\delta_1 < \delta < \delta_2\), \(0 < t < \infty \).
    \item[\textnormal{(ii)}] if \(I(u_0) < 0\), then \(u(\cdot,t) \in V_{\delta}\) for all \(\delta_1 < \delta < \delta_2\), \(0 < t < \infty\).
\end{enumerate}
\end{lemma}
\begin{proof}
By following the same arguments as in \cite[Lemmas 2.28 and 2.29]{Arora-2025}. The invariance of the associated sets $W_{\delta}$ and $V_{\delta}$ follows from these continuity properties of $t\to E(u(\cdot,t))$ and $t\to I(u(\cdot,t))$ and definition of strong solution.
 \end{proof}
\subsection{Tools from subdifferential theory} 
We first recall some definition and results from the subdifferential theory. Let \( H \) be a Hilbert space with inner product \( (\cdot,\cdot) \) and norm \( \|\cdot\|_{H} \). For a functional \( \varphi:H\to (-\infty,+\infty] \), we define the sublevel set and the domain of \( \varphi \) as  
\[
D(\varphi,a):=\left\{u\in H \,:\, \varphi(u)\leq a \right\}, \quad \text{for } a\in \mathbb{R} \quad \text{and} \quad D(\varphi)= \bigcup_{a\in \mathbb{R}} D(\varphi,a). 
\]
\begin{definition}
    Let \( \varphi:H\to (-\infty,+\infty] \) be a functional. The subdifferential \( \partial\varphi \) of \( \varphi \) is defined as  
    \[
    \partial\varphi(u)=\left\{ f\in H \,: \, \varphi(v)-\varphi(u)\geq (f,v-u), \ \forall\, v\in H \right\}
    \]
    for any \( u \in H \).
\end{definition}
It is well known that the subdifferential $\partial\varphi$ is a maximal monotone operator and its domain satisfies $D(\partial \varphi)\subset D(\varphi)$. Next, we recall a key existence result (see, \cite[Theorem 3.4, p.297]{Ishii} ) and the chain rule for subdifferentials (see, \cite[Lemma 3.3, p.73]{Heim}).
\begin{theorem}
\label{wellposedness}
    Let \( \varphi:H \to (-\infty,+\infty] \) be a proper, convex, and lower semicontinuous functional, and let \( \psi: H \to \mathbb{R} \) be another functional. Assume that the following conditions hold:
    \begin{itemize}
        \item[\textnormal{(i)}] For any \( a \in \mathbb{R} \), the sublevel set \( D(\varphi,a) \) is compact in \( H \).
        \item[\textnormal{(ii)}] \( D(\varphi) \subset D(\psi) \).
        \item[\textnormal{(iii)}] The set $\left\{(\partial\psi)^{0}(u) \mid u \in D(\varphi,a) \right\}$ is bounded in \( H \) for any \( a \in \mathbb{R} \), where \( (\partial\psi)^{0} \) denotes the element of minimal norm in the subdifferential \( \partial\psi(u) \).
    \end{itemize}
    Then, for each initial condition \( h \in D(\varphi) \), there exist \( T>0 \) and a strong solution \( u \) to the initial value problem:
    \begin{equation*}
    \begin{cases}
        \frac{du(\cdot, t)}{dt} + \partial\varphi(u(\cdot, t)) - \partial\psi(u(\cdot, t)) \ni 0 & \text{in } H, \quad t \in (0,T), \\[5pt]
        u(\cdot,0) = h(\cdot) & \text{in } \Omega.
    \end{cases}
    \end{equation*}
\end{theorem}
\begin{lemma}
\label{continuity of strong solution}
    Let \( \varphi:H \to (-\infty,+\infty] \) be a proper, convex, and lower semicontinuous functional. Suppose that \( u \in W^{1,1}(0,T;H) \) and that \( u(\cdot,t) \in D(\partial\varphi) \) for almost every \( t \in [0,T] \), where \( T>0 \). If there exists a function \( g \in L^{2}(0,T;H) \) such that $
    g(\cdot,t) \in \partial\varphi(u(\cdot,t))$ for a.e. $t \in [0,T],$ then the function \( t \mapsto \varphi(u(\cdot,t)) \) is absolutely continuous on \( [0,T] \) and satisfies  
    \[
    \frac{d}{dt} \varphi(u(\cdot,t)) = \left( g(\cdot,t), \frac{du(\cdot,t)}{dt} \right) \quad \text{for a.e. } t \in [0,T].
    \]
\end{lemma}
First, we reformulate the system \eqref{main:problem} as a Cauchy problem for an abstract evolution equation in the Hilbert space \( H = L^{2}(\Omega) \). To achieve this, we define the functionals \( \varphi, \psi : H \to (-\infty,+\infty] \) as follows
\[
\varphi(u) =
\begin{cases}
    \displaystyle \iint_{Q}\frac{|u(x)-u(y)|}{|x-y|^{N+s}} \,dx\,dy & \text{if } u\in W_{0}^{s,1}(\Omega),\\[5pt]
    +\infty & \text{if } u\in H\setminus W_{0}^{s,1}(\Omega),
\end{cases}
\quad 
\psi(u) =
\begin{cases}
    \displaystyle \int_{\Omega}F(x,u)\,dx & \text{if } u\in H\cap L^{r}(\Omega),\\[5pt]
     +\infty & \text{if } u\in H\setminus L^{r}(\Omega).
\end{cases}
\]
It is straightforward to verify that the functionals \( \varphi \) and \( \psi \) are proper, convex, and lower semicontinuous.
Next, we consider the operator \( \mathcal{L}: W_{0}^{s,1}(\Omega) \to \left(W_{0}^{s,1}(\Omega)\right)^{\ast} \) defined by
\[
(\mathcal{L}(u),v) := \iint_{Q}\frac{u(x)-u(y)}{|u(x)-u(y)|}\frac{v(x)-v(y)}{|x-y|^{N+s}} \, dx\,dy, \qquad \forall \, u,v \in W_{0}^{s,1}(\Omega),
\]
where \( \left(W_{0}^{s,1}(\Omega)\right)^{\ast} \) denotes the dual space of \( W_{0}^{s,1}(\Omega)\).
The operator \( \mathcal{L}_{H} \), which is the realization of \( \mathcal{L} \) in \( H = L^{2}(\Omega) \), is defined as
\begin{equation*}
 D(\mathcal{L}_{H}) = \left\{ u\in W_{0}^{s,1}(\Omega) \mid \mathcal{L}(u)\in H \right\} 
\quad \text{with} \quad \mathcal{L}_{H}(u) = \mathcal{L}(u), \quad \forall\, u \in D(\mathcal{L}_{H}).
\end{equation*}
The realization of the fractional $1$-Laplacian coincides with the subdifferential of the proper, convex, and lower semicontinuous functional $\varphi$; hence, by general results of convex analysis, it is a maximal monotone operator.
 \begin{lemma}
\label{operator and subdifferential}
Assume that the condition \eqref{Assump:f_0}--\eqref{Assump:f_2} hold. Then,
\[
\partial\varphi(u) = (-\Delta)^{s}_{1} u \quad \text{and} \quad \partial\psi(u) = f(x,u).\]
\end{lemma}

\begin{proof}
Both \( \mathcal{L}_{H} \) and \( \partial\varphi \) are maximal monotone operators in \( L^{2}(\Omega) \). Therefore, it suffices to show that \( \mathcal{L}_{H}(u) \subset \partial\varphi(u) \). Let \( u \in D(\mathcal{L}_{H}) \). Then, for all \( w \in W_{0}^{s,1}(\Omega) \), we have
\begin{align}
\label{v in subdifferential}
(\mathcal{L}_{H}(u), w - u) &= (\mathcal{L}(u), w - u) 
\leq \iint_{Q}\left|\frac{u(x)-u(y)}{|u(x)-u(y)|}\frac{w(x)-w(y)}{|x-y|^{N+s}}\right| \, dx\,dy-\iint_{Q}\frac{|u(x)-u(y)|}{|x-y|^{N+s}} \, dx\,dy\nonumber\\
&\leq \iint_{Q}\frac{|w(x)-w(y)|}{|x-y|^{N+s}} \,dx\,dy-\iint_{Q}\frac{|u(x)-u(y)|}{|x-y|^{N+s}} \,dx\,dy \leq \varphi(w)-\varphi(u).
\end{align}
If \( w \in  L^{2}(\Omega)\setminus W_{0}^{s,1}(\Omega) \), then \( \varphi(w) = +\infty \), and thus \eqref{v in subdifferential} still holds trivially. This proves that \( \mathcal{L}_{H}(u) \in \partial\varphi(u) \), which implies \( \mathcal{L}_{H}(u) = \partial\varphi(u) \). Similarly, by exploiting the convexity of the mapping \( t \mapsto F(x,t) \), we can show that \( \partial\psi(u) = f(x,u) \).
\end{proof}
\section{Regularization and asymptotic analysis}
\label{Regularization and asymptotic analysis}
In this section, first we develop the regularization scheme to construct the global weak solution of \eqref{main:problem} via fractional $p$-Laplacian evolution problem as $p\to 1^{+}$. We establish the uniform estimates (independent of $p$) for the potential depth $d_{p}$ and uniform bounds on the weak solution of approximate problem. We then derive the convergence estimates needed to pass the limit $p\to 1^{+}$. 
\subsection{Approximating problem}
Let us consider the following problem
\begin{equation} \label{main:problem2}
        \left\{
    \begin{aligned}
      u_{t}+(-\Delta)_{p}^{s_p} u &=  f(x,u), &&
 \text{in } \Omega \times (0, \infty),\\
 u &= 0, && \text{in}~  \mathbb{R}^N \setminus \Omega \times (0, \infty),\\
 u(x,0)&= u^{(p)}_0(x), && \text{in}~ \Omega
\end{aligned}
    \right.
 \end{equation}
 where
 \begin{equation}
  \label{choice:sp}
      s_{p}:=N+s-\frac{N}{p} \quad \text{for any} \ 1<p< \frac{N}{N+s-1}.
 \end{equation}
 The energy functional $E^{(p)} : W_{0}^{s_p,p}(\Omega) \to \mathbb{R}$ corresponding to the stationary counterpart of the problem \eqref{main:problem2} is given by
\[
    E^{(p)}(u) = \frac{1}{p}\iint_{Q} \frac{|u(x)-u(y)|^{p}}{|x-y|^{N+s_pp}}\,dx\,dy - \int_{\Omega} F(x,u)\,dx,
\]
and the Nehari functional is  given by
\[ I^{(p)}(u) = \iint_{Q}\frac{|u(x)-u(y)|^{p}}{|x-y|^{N+s_pp}}\,dx\,dy - \int_{\Omega} f(x,u) u\,dx.
\]
Next, we define the Nehari manifold as
\[
    \mathcal{N}^{(p)} = \left\{ u \in W_{0}^{s_p,p}(\Omega) \setminus \{ 0 \} \mid I^{(p)}(u) = 0 \right\}.
\]
The potential well and its corresponding sets are defined as
\[
    W^{(p)} := \left\{ u \in W_{0}^{s_p,p}(\Omega) \mid I^{(p)}(u) > 0, E^{(p)}(u) < d_{p} \right\} \cup \{0\}
\]
and
\[
    V^{(p)} := \left\{ u \in W_{0}^{s_p,p}(\Omega) \mid I^{(p)}(u) < 0, E^{(p)}(u) < d_{p} \right\},
\]
where the depth $d_{p}$ of the potential well $W_p$ is defined as
\begin{equation}\label{depth-p}
     d_{p} := \inf_{u \in \mathcal{N}^{(p)}} E^{(p)}(u).
\end{equation}

\begin{lemma}
\label{existence of lambda*p}
 Let $f$ satisfy the conditions \eqref{Assump:f_0}--\eqref{Assump:f_1} and 
 \begin{equation}\label{f2-p-case}
      f(x, \cdot) \in C^1(\mathbb{R}) \ \text{for a.e.} \ x \in \Omega \quad \text{and} \quad (p-1)f(x,t) t < t^{2} f'(x,t) \ \text{for all} \  t \in \mathbb{R}\setminus\{0\}.
 \end{equation}
 Then, for any $u \in W_{0}^{s_p,p}(\Omega)$ with $\|u\|_{W_{0}^{s_p,p}(\Omega)} \neq 0$, we have
\begin{enumerate}
    \item[\textnormal{(i)}] $\lim_{\lambda_{p} \to 0^{+}} E^{(p)}(\lambda_{p} u) = 0$, \quad $\lim_{\lambda_{p} \to +\infty} E^{(p)}(\lambda_{p} u) = -\infty$.
    \item[\textnormal{(ii)}] There exists a unique $\lambda_{\ast,p} = \lambda_{\ast,p}(u) > 0$ such that $\frac{dE^{(p)}(\lambda_{p} u)}{d\lambda_p} \big|_{\lambda_p = \lambda_{\ast,p}} = 0$. Moreover, $E^{(p)}(\lambda_p u)$ is increasing on $0 < \lambda_p \leq \lambda_{\ast,p}$, decreasing on $\lambda_{\ast,p} \leq \lambda_{p} < \infty$, and attains its maximum at $\lambda_{p} = \lambda_{\ast,p}$.
    \item[\textnormal{(iii)}] $I^{(p)}(\lambda_p u) \geq 0$ for $0 < \lambda_p \leq \lambda_{\ast,p}$, $I^{(p)}(\lambda_p u) < 0$ for $\lambda_{\ast,p} < \lambda_p < \infty$, and $I^{(p)}(\lambda_{\ast,p} u) = 0$.
\end{enumerate}
\end{lemma}
\begin{proof}
    The proof follows the by repeating the same argument as in \cite[Lemma 2.20]{Arora-2025}.
\end{proof}
\begin{remark}\label{rem:condi}
    Note that the assumptions \eqref{Assump:f_1}-\eqref{Assump:f_2} implies \eqref{f2-p-case} for any $p \in (1, \Theta]$ where $\Theta$ is given in \eqref{Assump:f_2}.
\end{remark}
Set
\[
\underline{p}:= \min\big\{\frac{N+1}{N+s}, \Theta, \theta\big\}
\]
where $\theta, \Theta$ are given in \eqref{Assump:f_1} and \eqref{Assump:f_2} respectively.

\begin{lemma}\label{Positive:depth dp}
           Let $f$ satisfy the conditions \eqref{Assump:f_0}-\eqref{Assump:f_1} and $q\in(1,\underline{p})$. Then, for $p\in(1,\min\{q,\frac{r}{\gamma}\})$ and $1< \gamma <r$, the depth $d_{p}$ of the potential well $W^{(p)}$  is positive. Moreover, there exists a constant $C_{q,\gamma}>0$ (independent of $p$) such that $d_{p} \geq C_{q,\gamma}$.
\end{lemma}
\begin{proof}
Let $u\in\mathcal{N}^{(p)}$. Since $r<1^{*}_{s}<p_{s}^{*}$ for all $p>1$ then, from \eqref{Assump:f_0} and Proposition \ref{prop:embedding}(i), we obtain
 \begin{align}
         \label{lower bound of norm up}
   \frac{1}{p}\|u\|_{W_{0}^{s_{p},p}(\Omega)}^{p}=\frac{1}{p}\iint_{Q} \frac{|u(x)-u(y)|^{p}}{|x-y|^{N+s_{p}p}} \,dx\,dy &= \int_{\Omega}f(x,u)u\,dx \leq C\|u\|_{L^{r}(\Omega)}^{r}\leq S_{1}\|u\|_{W_{0}^{s,1}(\Omega)}^{r},
       \end{align}
        where
         $S_1:= C(C_{Sob})^{r}$ and $C_{Sob}$ is embedding constant, and $C$  is given in \eqref{Assump:f_0}. From \cite[Lemma 2.4]{DLi-2024}, we have 
         \begin{equation*}
           \|u\|_{W_{0}^{s,1}(\Omega)}^{r}\leq C_{1}^{\frac{r(1-p)}{p}}\|u\|_{W_{0}^{s_{p},p}(\Omega)}^{r}\leq \max\{1, C_{1}^{\frac{r(1-q)}{q}}\}\|u\|_{W_{0}^{s_{p},p}(\Omega)}^{r}. 
         \end{equation*}
         Combining this with \eqref{lower bound of norm up}, we have 
         \begin{equation}
             \label{lower bound independent of p}
             \|u\|_{W_{0}^{s_{p},p}(\Omega)} \geq \left(C_{q}\right)^{\frac{-1}{r-p}}, \quad \text{where} \quad C_{q}=qS_{1}\max\{1, C_{1}^{\frac{r(1-q)}{q}}\}. 
         \end{equation}
Next, by using \eqref{Assump:f_1} and \eqref{lower bound independent of p} we infer that
       \begin{align*}
        E^{(p)}(u) & \geq \frac{1}{p}{\iint_{Q}}\frac{|u(x)-u(y)|^{p}}{|x-y|^{N+s_{p}p}}\,dx\,dy-\int_{\Omega}\frac{f(x,u)u}{\theta}\,dx\\
        &= \frac{1}{p}\iint_{Q} \frac{|u(x)-u(y)|^{p}}{|x-y|^{N+s_{p}p}}\,dx\,dy-\frac{1}{\theta}{\iint_{Q}}\frac{|u(x)-u(y)|^{p}}{|x-y|^{N+s_{p}p}}\,dx\,dy\\
        &=\left(\frac{1}{p}-\frac{1}{\theta}\right)\|u\|_{W_{0}^{s_{p},p}(\Omega)}^{p}
        \geq \left(\frac{1}{p}-\frac{1}{\theta}\right)\left(\frac{1}{C_{q}}\right)^{\frac{p}{r-p}}\geq \left(\frac{1}{q}-\frac{1}{\theta}\right)\min\left\{1, \left(\frac{1}{C_{q}}\right)^{\frac{1}{\gamma-1}}\right\} .
\end{align*}
Finally, by taking infimum over $\mathcal{N}$, we obtain
$$d_{p}=\inf_{u\in\mathcal{N}}E^{(p)}(u)\geq \left(\frac{1}{q}-\frac{1}{\theta}\right)\min\left\{1, \left(\frac{1}{C_{q}}\right)^{\frac{1}{\gamma-1}}\right\}=:C_{q,\gamma}.$$
Thus, the desired result follows.
\end{proof}
\begin{definition}
    \label{weaksolution 2}
    A function \( u^{(p)} \in L^{\infty}(0,T;W_{0}^{s_p,p}(\Omega)\cap L^{2}(\Omega)) \) with \( u^{(p)}_{t} \in L^{2}(0,T;L^{2}(\Omega)) \) is said to be a weak solution of the problem \eqref{main:problem2} if 
\begin{enumerate}
    \item[\textnormal{(i)}] \( u^{(p)}(\cdot,0) = u^{(p)}_{0} \) a.e in $\Omega$.
    \item[\textnormal{(ii)}] for all
$\phi\in W_{0}^{s_p,p}(\Omega)\cap L^{2}(\Omega)$ and a.e. $t\in [0,T]$ the following equality holds:
\begin{equation}
\label{definition of weak sol up}
\begin{aligned}
   \int_{\Omega}u^{(p)}_{t}(x,t)\phi(x)\,dx &+\iint_{Q}\frac{|u^{(p)}(x,t)-u^{(p)}(y,t)|^{p-2}(u^{(p)}(x,t)-u^{(p)}(y,t))}{|x-y|^{N+s_pp}}(\phi(x)-\phi(y))\,dx\,dy\\
&=\int_{\Omega}f(x,u^{(p)})\phi(x)\,dx.
\end{aligned}
\end{equation}
\item[\textnormal{(iii)}]  $u^{(p)} \in C(0, T; L^k(\Omega))$ for all $k\in [1,p^{*}_{s_p})$ and the following energy  relation is satisfied:
\begin{equation}
\label{sol: up}
\begin{aligned}
    \int_{0}^{t} \|u^{(p)}_t(\cdot, \tau)\|^{2}_{L^{2}(\Omega)} \, d\tau + E^{(p)}(u^{(p)}(\cdot,t))
    &\leq E^{(p)}(u^{(p)}_{0}) \quad \text{a.e. } t \in [0,T).
\end{aligned}
\end{equation}

\end{enumerate}    
\end{definition}
\begin{lemma}
    \label{Approximate problem main thm}
     Assume that conditions \eqref{Assump:f_0}-\eqref{Assump:f_2} hold. Let $p\in (1, q)$ and the initial data $u_{0}^{(p)}$ satisfies $u_{0}^{(p)} \in W_{0}^{s_p,p}(\Omega)\cap L^{2}(\Omega)$ such that $E^{(p)}(u_{0}^{(p)})<d_p$ and $I^{(p)}(u_{0}^{(p)})>0$. Then the problem \eqref{main:problem2} admits a global weak solution
     \[
u^{(p)} \in L^\infty(0, \infty; W_{0}^{s_p,p}(\Omega)\cap L^{2}(\Omega)), \quad u^{(p)}_t \in L^2(0, \infty; L^2(\Omega)).
\] 
Moreover 
\begin{equation*}
 \int_{0}^{t}||u^{(p)}_t(\cdot,\tau) ||^{2}_{L^{2}(\Omega)} ~d\tau <d_{p},\quad
 \text{and}\quad \|u^{(p)}(\cdot,t)\|_{s_p,p,2}\leq \left(\frac{d_{p}p\theta}{\theta-p}\right)^{\frac{1}{p}}+\|u^{(p)}_0\|_{L^{2}(\Omega)}+o(1) \quad\forall\, t\in [0,\infty),
\end{equation*}
where $d_p$ is defined in \eqref{depth-p} and $\theta$ is given in \eqref{Assump:f_1}.
\end{lemma}
 \begin{proof}
We have the Gelfand triple
 $$W_{0}^{s_p,p}(\Omega)\cap L^{2}(\Omega)\hookrightarrow^{c,d}L^{2}(\Omega)\hookrightarrow^{c,d}\left(W_{0}^{s_p,p}(\Omega)\cap L^{2}(\Omega)\right)^{*}.$$
 Here, $\hookrightarrow^{c,d}$ denotes a continuous and dense embedding. Let $\{V_{n}\}_{n\in\mathbb{{N}}}$ be a Galerkin scheme of the separable Banach space $W_{0}^{s_p,p}(\Omega)\cap L^{2}(\Omega)$, {\it i.e.}
 $$V_{n}=span\{e_{1},e_{2},\cdots,e_{n}\},\quad\overline{\bigcup_{n\in \mathbb{N}}V_{n}}^{\|\cdot\|_{s_p,p,2}}=W_{0}^{s_p,p}(\Omega)\cap L^{2}(\Omega),$$
 where $\{e_{j}\}_{j=1}^{\infty}$ is an orthonormal basis of $L^{2}(\Omega).$ Now, consider the Galerkin approximation
\begin{equation*}
u^{[n]}(x, t) = \sum_{j=1}^n c_j^{[n]}(t) e_j(x),
\end{equation*}
where the functions $c_j^{[n]}(t): [0, T] \to \mathbb{R}$ satisfy the following system of ordinary differential equations for $j = 1, 2, \dots, n$:
\begin{equation}
\label{Galerkin:approximate}
\begin{aligned}
    \left(\frac{du^{[n]}}{dt}, e_j\right)_{L^2(\Omega)} + \left(u^{[n]}, e_j\right)_{W_{0}^{s_p,p}(\Omega)} = \int_{\Omega}f(x, u^{[n]}) e_j\,dx, \quad u^{[n]}(x, 0) = \sum_{i=1}^n \left(u_{0}^{(p)}, e_i\right)_{L^2(\Omega)} e_i(x),
\end{aligned}
\end{equation}
with
\begin{equation*}
\left(u^{[n]}, e_j\right)_{W_{0}^{s_p,p}(\Omega)}= \iint_Q \frac{|u^{[n]}(x,t) - u^{[n]}(y,t)|^{p-2}(u^{[n]}(x,t) - u^{[n]}(y,t))(e_j(x) - e_j(y))}{|x-y|^{N+s_pp}}\, dx\,dy,
\end{equation*}
and the initial condition
\begin{equation*}
 u^{[n]}(x, 0) \to u^{(p)}_0 \quad \text{in } \,W_{0}^{s_p,p}(\Omega)\cap L^{2}(\Omega) \quad \text{as } n \to \infty.
\end{equation*}
It follows from \eqref{Galerkin:approximate} that
\begin{equation}\label{aneq}
\begin{split}
   \frac{dc^{[n]}_{j}}{dt}(t) & = -\iint_Q \frac{|u^{[n]}(x,t) - u^{[n]}(y,t)|^{p-2}(u^{[n]}(x,t) - u^{[n]}(y,t))(e_j(x) - e_j(y))}{|x-y|^{N+s_pp}}\,dx\,dy  \\
   & \qquad + \int_{\Omega}f(x, u^{[n]}) e_j\,dx.
   \end{split}
\end{equation}
Define
\begin{align*}
F_j^{[n]}(c^{[n]}(t)) &= -\iint_{Q} \frac{\left|\sum_{i=1}^n c_i^{[n]}(t)e_i(x) - \sum_{i=1}^n c_i^{[n]}(t)e_i(y)\right|^{p-2}\left(\sum_{i=1}^n c_i^{[n]}(t)e_i(x) - \sum_{i=1}^n c_i^{[n]}(t)e_i(y)\right)}{|x-y|^{N+s_pp}} \\
&\qquad\qquad \times(e_j(x) - e_j(y))\,dx\,dy + \int_{\Omega} f\big(x, \sum_{i=1}^n c_i^{[n]}(t)e_i(x)\big) e_j(x) \,dx,
\end{align*}
and
\[
c^{[n]}(\cdot) = \left(c_i^{[n]}(\cdot)\right)_{i=1}^n, \quad F^{[n]}(\cdot) = \left(F_i^{[n]}(\cdot)\right)_{i=1}^n, \quad c^{[n]}_0 = \left((u^{(p)}_0, e_i)_{L^2(\Omega)}\right)_{i=1}^n.
\]
In view of \eqref{aneq}, equation \eqref{Galerkin:approximate} can be rewritten as
\begin{equation}
\label{Galerkin:system}
    \frac{dc^{[n]}(t)}{dt} = F^{[n]}(c^{[n]}(t)), \qquad \text{and} \qquad c^{[n]}(0) = c^{[n]}_0.
\end{equation}
\textbf{Claim $1$}:  Equation \eqref{Galerkin:system} admits a solution in $[0, T_{\max})$. \\
Multiplying the first equality of \eqref{Galerkin:system} by $c^{[n]}(t)$ and using the condition \eqref{Assump:f_0}, we obtain
\begin{align*}
 \frac{1}{2}\frac{d|c^{[n]}(t)|^{2}}{dt}&\leq CC(n)\left(|c^{[n]}(t)|^{r}\right) \quad \text{where} \quad C(n,r)= \sum_{i=1}^{n}\|e_{i}(x)\|_{L^{r}(\Omega)}^{r}.
 \end{align*}
The remainder of the argument solving the resulting ordinary differential inequality 
and applying Peano's theorem iteratively to extend the solution to the maximal interval 
$[0, T_{\max})$ follows by the same line of reasoning as in the proof of 
\cite[Claim 1, Theorem 4.2]{Arora-2025}, with condition \eqref{Assump:f_0} playing the 
role of Corollary 2.17 and Lemma 2.10 therein.\\
\textbf{Claim 2}: \(u^{[n]}(\cdot,t) \in W^{(p)}\) for all $t \in [0, T_{\max})$.\\
For this purpose, we multiply \eqref{Galerkin:approximate} by \(\frac{dc_j^{[n]}}{dt}\), sum over \(j\) from \(1\) to \(n\), and integrate from \(0\) to \(t\).
As a result, we obtain:
\begin{equation*}
\int_{0}^{t}||u^{[n]}_t(\cdot,\tau)||^{2}_{L^{2}(\Omega)} ~d\tau +E^{(p)}(u^{[n]}(\cdot, t))=E^{(p)}(u^{[n]}(\cdot, 0)).
\end{equation*}
Hence, from the fact that $u^{[n]}(\cdot,0)\rightarrow u^{(p)}_{0} \text{ in }W_{0}^{s_p,p}(\Omega)\cap L^{2}(\Omega)$ and the continuity of $E^{(p)}$ and $I^{(p)}$ ensures that 
\[
E^{(p)}(u^{[n]}(\cdot,0))\rightarrow E^{(p)}(u^{(p)}_{0})<d_p \quad \text{and} \quad I^{(p)}(u^{[n]}(\cdot,0))\rightarrow I^{(p)}(u^{(p)}_{0})>0\ \text{as} \ n \to \infty. 
\]
Therefore, for sufficiently large \( n \), we get 
\begin{equation}
\label{for large n}
    \int_{0}^{t} \|u^{[n]}_t(\cdot,\tau)\|^{2}_{L^{2}(\Omega)} \,d\tau + E^{(p)}(u^{[n]}(\cdot, t)) = E^{(p)}(u^{[n]}(\cdot, 0)) < d_{p}, \quad \text{and} \quad I^{(p)}(u^{[n]}(\cdot,0)) > 0,
\end{equation}  
which implies that \( u^{[n]}(\cdot,0) \in W^{(p)} \).  

Now, suppose that \( u^{[n]}(\cdot,t) \notin W^{(p)} \) for sufficiently large \( n \). Then, there exists \( t_0 > 0 \) such that \( u^{[n]}(\cdot,t_0) \notin W^{(p)} \). If \( t_0 \) is not unique, we may assume, without loss of generality, that \( t_0 \) is the first time for which \( u^{[n]}(\cdot, t_{0}) \notin W^{(p)} \). This implies that \( u^{[n]}(\cdot, t) \not\equiv 0 \) for all \( t \in (0, t_0] \), and  
\[
E^{(p)}(u^{[n]}(\cdot,t_0)) = d_{p}, \quad \text{or} \quad E^{(p)}(u^{[n]}(\cdot,t_0)) > d_{p}, \quad \text{or} \quad I^{(p)}(u^{[n]}(\cdot,t_0)) = 0, \quad \text{or} \quad I^{(p)}(u^{[n]}(\cdot,t_0)) < 0.
\]
Clearly, from \eqref{for large n}, we have \(E^{(p)}(u^{[n]}(\cdot,t_0)) < d_{p}\).
Now, suppose \(I^{(p)}(u^{[n]}(\cdot,t_0)) = 0\). This implies that \(u^{[n]}(\cdot,t_0) \in \mathcal{N}^{(p)}\), {\it i.e.}
$$
d_{p} \leq E^{(p)}(u^{[n]}(\cdot,t_0)).
$$
This contradicts \eqref{for large n}. Now, if \( I^{(p)}(u^{[n]}(\cdot,t_0))< 0 \), then, since \( I^{(p)} \) is a continuous functional and \( I^{(p)}(u^{[n]}(\cdot,0))>0 \), there exists \( t_{1}\in (0,t_{0}) \) such that  
\[
I^{(p)}(u^{[n]}(\cdot,t_{1}))=0 \quad \text{and} \quad u^{[n]}(\cdot, t_1) \not\equiv 0.
\]  
This further implies that \( u^{[n]}(\cdot,t_1) \in \mathcal{N}^{(p)} \), \textit{i.e.}  
\[
d_{p} \leq E^{(p)}(u^{[n]}(\cdot,t_1)),
\]  
which contradicts \eqref{for large n} once again, thereby proving \textbf{Claim 2}.
\\
\textbf{Claim 3}: Uniform boundedness of $u^{[n]}$ in $W_{0}^{s_p,p}(\Omega))\cap L^{2}(\Omega)$ and $f(x,u^{[n]})$ in $L^{\frac{r}{r-1}}(\Omega)$ for all $t \in [0, T_{\max})$. \\ From \textbf{Claim 2} and \eqref{Assump:f_1}, we obtain
\begin{align*}
d_{p}>E^{(p)}(u^{[n]})&=\frac{1}{p}\iint_{Q}\frac{|u^{[n]}(x,t)-u^{[n]}(y,t)|^{p}}{|x-y|^{N+s_pp}}\,dx\,dy-\int_{\Omega}F(x,u^{[n]})\,dx\\
&\geq \frac{1}{p}\iint_{Q}\frac{|u^{[n]}(x,t)-u^{[n]}(y,t)|^{p}}{|x-y|^{N+s_pp}}\,dx\,dy - \frac{1}{\theta} \int_{\Omega}  f(x,u^{[n]}) u^{[n]} \,dx\\
&\geq \frac{1}{p}\iint_{Q}\frac{|u^{[n]}(x,t)-u^{[n]}(y,t)|^{p}}{|x-y|^{N+s_pp}}\,dx\,dy-\frac{1}{\theta}\iint_{Q}\frac{|u^{[n]}(x,t)-u^{[n]}(y,t)|^{p}}{|x-y|^{N+s_pp}}\,dx\,dy+\frac{1}{\theta}I(u^{[n]})\\
&\geq\left(\frac{1}{p}-\frac{1}{\theta}\right)\iint_{Q}\frac{|u^{[n]}(x,t)-u^{[n]}(y,t)|^{p}}{|x-y|^{N+s_pp}}\,dx\,dy \geq \left(\frac{1}{p}-\frac{1}{\theta}\right)\|u^{[n]}(\cdot,t)\|^{p}_{W_{0}^{s_p,p}(\Omega)}
\end{align*}
Thus, due to \eqref{for large n}, we conclude that
$$\int_{0}^{t}||u^{[n]}_t(\cdot,\tau) ||^{2}_{L^{2}(\Omega)}\,d \tau+\left(\frac{1}{p}-\frac{1}{\theta}\right)\|u^{[n]}(\cdot,t)\|^{p}_{W_{0}^{s_p,p}(\Omega)}\leq E^{(p)}(u^{[n]}(0))<d_{p},$$
which shows that
\begin{equation}
    \label{u_t uniform bound}
    \int_{0}^{t}||u^{[n]}_t(\cdot,\tau) ||^{2}_{L^{2}(\Omega)} ~d\tau <d_{p}\quad\forall\, t\in [0,T_{\max})
\end{equation}
and
\begin{equation}
    \label{u uniform bound}
    \|u^{[n]}(\cdot,t)\|_{W_{0}^{s_p,p}(\Omega)}\leq \left(\frac{p\theta d_{p}}{\theta-p}\right)^{\frac{1}{p}}\quad\forall\, t\in [0,T_{\max}).
\end{equation}
Next, we multiply \eqref{Galerkin:approximate} by \(c_j^{[n]}(t)\), sum over \(j\) from \(1\) to \(n\) and using \eqref{for large n}, we obtain
$$\frac{1}{2}\frac{d}{dt}\|u^{[n]}(\cdot,t)\|_{L^{2}(\Omega)}^{2}=-I^{(p)}(u^{[n]}(\cdot,t))<0,\quad\forall\, t\in [0,T_{\max}).$$
This implies that 
\begin{align}
    \label{u L2 uniform bound}
    \|u^{[n]}(\cdot,t)\|_{L^{2}(\Omega)}&\leq \|u^{[n]}(\cdot,0)\|_{L^{2}(\Omega)}= \|u^{(p)}_0\|_{L^{2}(\Omega)}+o(1)\quad\forall\, t\in [0,T_{\max}).
\end{align}
Collecting the estimates in \eqref{u uniform bound} and \eqref{u L2 uniform bound}, we infer that
\begin{equation}
    \label{sp,p,2 uniform bound}
    \|u^{[n]}(\cdot,t)\|_{s_p,p,2}\leq \left(\frac{d_{p}p\theta}{\theta-p}\right)^{\frac{1}{p}}+\|u^{(p)}_0\|_{L^{2}(\Omega)}+o(1) \quad\forall\, t\in [0,T_{\max}).
\end{equation}
Hence, the above estimates yield $T_{\max}= +\infty$. Furthermore, using condition \eqref{Assump:f_0} and embedding of $W_{0}^{s_p,p}(\Omega)\hookrightarrow L^{q}(\Omega)$ for all $q\in [1,p^{*}_{s_p}]$, we obtain
\begin{align}
\label{f(x,un) uniform bound}
    \int_{\Omega}|f(x,u^{[n]}(\cdot,t))|^{\frac{r}{r-1}}\,dx&\leq C^{\frac{r}{r-1}}\int_{\Omega}|u^{[n]}(\cdot,t)|^{r}\,dx \leq C_{*} \|u^{[n]}(\cdot,t)\|^{r}_{{W_{0}^{s_p,p}(\Omega)}} \leq C_{*} \left(\frac{d_{p}p\theta}{\theta-p}\right)^{\frac{r}{p}},
\end{align}
 where  $C_{*}:= C^{\frac{r}{r-1}} (C^{*})^{r}$ and $C^\ast$ is the embedding constant. This further gives the existence of $u^{(p)}$ and a subsequence $\{u^{[n]}\}_{n\in \mathbb{N}}$ (still denoted by $\{u^{[n]}\}_{n\in \mathbb{N}}$) such that, as $n \to \infty$, we have
 \begin{equation}
 \label{weak convergence}
 \left.
 \begin{aligned}
     u^{[n]}&\stackrel{\ast}{\rightharpoonup} u^{(p)}&&\text{ in }L^{\infty}(0,+\infty;W_{0}^{s_p,p}(\Omega)\cap L^{2}(\Omega)),\\
 u_{t}^{[n]}&\weak u^{(p)}_{t}&&\text{ in } L^{2}(0,+\infty;L^{2}(\Omega)),\\
 f(x,u^{[n]})&\stackrel{\ast}{\rightharpoonup}\xi&&\text{ in }L^{\infty}(0,+\infty;L^{\frac{r}{r-1}}\Omega)).
 \end{aligned}
 \right\}
 \end{equation}
Then, from Aubin-Lions compactness theorem\cite[Corollary 4, p. 85]{simon} for any $T>0$, we get
 \begin{equation}
 \label{stong:conv}
 \begin{aligned}
 u^{[n]}&\to u^{(p)} \text{ in } C([0,T];L^{k}(\Omega)),\text{  as } n \to \infty, \text{ for all } k\in[1,p^{*}_{s})
 \end{aligned}
  \end{equation}
and so, $\xi=f(x,u^{(p)}).$ By using the weak lower semicontinuity of the norm and uniform bounds in \eqref{u_t uniform bound}, \eqref{sp,p,2 uniform bound}, we obtain
  \begin{equation*}
      \int_{0}^{t}||u^{(p)}_t(\cdot,\tau) ||^{2}_{L^{2}(\Omega)} ~d\tau \leq \liminf_{n\to\infty} \int_{0}^{t}||u^{[n]}_t(\cdot,\tau) ||^{2}_{L^{2}(\Omega)} ~d\tau<d_{p}\quad\text{for all } t\in [0,\infty). 
  \end{equation*} 
  and
  \begin{equation*}
  \label{uniform bound sp,1,2 of up}
     \|u^{(p)}(\cdot,t)\|_{s_p,p,2} \leq \liminf_{n\to\infty}\|u^{[n]}(\cdot,t)\|_{s_p,p,2} \leq \left(\frac{d_{p}p\theta}{\theta-p}\right)^{\frac{1}{p}}+ \|u^{(p)}_0\|_{L^{2}(\Omega)}+o(1)\quad \text{for all } t\in [0,\infty). 
  \end{equation*}
\textbf{Claim 4}: The function $u$ is a weak solution to problem (\ref{main:problem2}). \\
The proof follows by the same arguments as in \cite[Theorem~4.2, Claim~4]{Arora-2025}. Specifically, choosing test functions of the form 
 $$v=\sum_{j=1}^{k}l_{j}(t)e_{j},$$
where $l_{j}(t)\in C^{1}(0,+\infty)$ with $j=1,2,3,\dots,k $ ($k\leq n $). Passing to the limit as $n \to \infty$ via the theory of monotone and hemicontinuous 
operators as in \cite{Arora-2025}, and using the convergences \eqref{weak convergence}-\eqref{stong:conv} together with the lower semicontinuity of the norm, we conclude that $u^{(p)}$ satisfies the weak formulation \eqref{definition of weak sol up} for all $\phi \in W_0^{s_p,p}(\Omega) \cap L^2(\Omega)$ and a.e.\ $t \in (0,+\infty)$, along with the energy inequality \eqref{sol: up}.
\end{proof}
\subsection{Convergence estimates} 
\begin{lemma}
\label{Nonlinearity convergence}
    Assume that conditions \eqref{Assump:f_0}-\eqref{Assump:f_2} hold true. Let $u^{[n]}_{0} \to u_{0}$ in $L^\alpha(\Omega)$ for some $\alpha >r$ where $r$ is given in \eqref{Assump:f_0}. Then,  
    \[
  \lim_{n\to\infty}\int_\Omega F(x,u_0^{[n]})\,dx=\int_\Omega F(x,u_0)\,dx
  \quad \text{and}\quad
  \lim_{n\to\infty}\int_\Omega f(x,u_0^{[n]})u_0^{[n]}\,dx=\int_\Omega f(x,u_0)u_0\,dx.
\]
\end{lemma}
\begin{proof}
Since $u^{[n]}_{0}\to u_{0}$ in $L^{q}(\Omega)$ for every $q\in[1,\alpha]$ and $u_{0}^{[n]}(x)\to u_{0}(x)$ a.e. in $\Omega$. Consequently, for each $q\in[1,\alpha]$, there exists a constant $M_{q}>0$ independent of $n$, such that $\|u_{0}^{[n]}\|_{L^{q}(\Omega)}\leq M_{q}$ for all $n\in\mathbb{N}$. Now, by using the pointwise convergence of the sequence $u_{0}^{[n]}$ and \eqref{Assump:f_2}, we have
   \begin{equation}
       \label{poinwise convergence of nonlinearity}
      f(x,u_{0}^{[n]}(x))u_{0}^{[n]}(x)\to f(x,u_{0}(x))u_{0}(x)\quad\text{a.e. in }\Omega.
   \end{equation}
Moreover, from \eqref{Assump:f_0}, we obtain
the sequence $\{f(x,u_{0}^{[n]})u_{0}^{[n]}\}_{n \in \mathbb{N}}$ is uniformly integrable. Analogously, from \eqref{Assump:f_1} we obtain that the sequence $\{F(x,u_{0}^{[n]})\}_{n \in \mathbb{N}}$ is uniformly integrable and pointwise convergence to $F(x,u_{0})$ a.e. in $\Omega$. Now, by applying the Vitali's converegence theorem, we obtain the required claim. 
\end{proof}
  \begin{lemma}
        \label{strong convergence u_0}
      Assume that conditions \eqref{Assump:f_0}-\eqref{Assump:f_2} hold and $q \in (1, \underline{p})$. If $u^{[n]}_{0}\to u_{0}$ in $W_{0}^{s_{q},q}(\Omega)$ as $n\to \infty$, then
      \begin{equation}
    \label{EQ:BBM formula}
    \frac{1}{p_{n}}\iint_{Q}\frac{|u_{0}^{[n]}(x)-u_{0}^{[n]}(y)|^{p_{n}}}{|x-y|^{N+s_{p_n}p_{n}}}\,dx\,dy \to \iint_{Q}\frac{|u_{0}(x)-u_{0}(y)|}{|x-y|^{N+s}}\,dx\,dy \quad \text{ as } n\to\infty.
\end{equation}
Moreover, the following convergence hold:
        $$E^{(p_n)}(u_{0}^{[n]})\to E(u_{0}) \quad\text{and}\quad I^{(p_n)}(u_{0}^{[n]})\to I(u_{0}) \quad \text{as } \,p_n\to 1^{+}.$$
    \end{lemma}
        \begin{proof}
Setting $s_{p_n}=N+s-\frac{N}{p_n}$, we have $s_{p_{n}}\in (s,1)$. The choice of $s_{p_n}$ allows us to write the integrand as
\[
  \frac{\abs{u_{0}^{[n]}(x)-u_{0}^{[n]}(y)}^{p_{n}}}{\abs{x-y}^{N+s_{p_n}p_n}}=
  \left(\frac{\abs{u_{0}^{[n]}(x)-u_{0}^{[n]}(y)}}{\abs{x-y}^{N+s}}\right)^{p_n}.
\]
Since $u_{0}^{[n]}(x)\to u_{0}(x)$ a.e. in $\Omega$, the above integrand
converges to
$\frac{\abs{u_{0}(x)-u_{0}(y)}}{\abs{x-y}^{N+s}}$ a.e. $(x,y)\in Q$. Moreover, as $u^{[n]}_{0}\in W_{0}^{s_{q},q}(\Omega)$, we split the integral over $Q$ as 
\begin{align*}
    \iint_{Q}\frac{|u_{0}^{[n]}(x)-u_{0}^{[n]}(y)|^{p_{n}}}{|x-y|^{N+s_{p_n}p_{n}}}\,dx\,dy&=\int_{\Omega}\int_{\Omega}\frac{|u_{0}^{[n]}(x)-u_{0}^{[n]}(y)|^{p_{n}}}{|x-y|^{N+s_{p_n}p_{n}}}\,dx\,dy+2\int_{\Omega}\int_{\mathcal{C}\Omega}\frac{|u_{0}^{[n]}(y)|^{p_{n}}}{|x-y|^{N+s_{p_n}p_{n}}}\,dx\,dy.
\end{align*}
Next, we derive the convergence estimates on the right hand integrals separately:\\
\textbf{Claim 1:}
\begin{equation}\label{convergence in Omega}
    \int_{\Omega}\int_{\Omega}\frac{|u_{0}^{[n]}(x)-u_{0}^{[n]}(y)|^{p_{n}}}{|x-y|^{N+s_{p_n}p_{n}}}\,dx\,dy \to \int_{\Omega}\int_{\Omega}\frac{|u_{0}(x)-u_{0}(y)|}{|x-y|^{N+s}}\,dx\,dy \quad \text{as} \ n \to \infty.
\end{equation}
Since $u^{[n]}_{0}\to u_{0}$ in $W_{0}^{s_{q},q}(\Omega)$, there exists $M>0$ such that $\|u^{[n]}_{0}\|_{W_{0}^{s_{q},q}(\Omega)}<M$ for all $n\in \mathbb{N}$. To this end, we can write
\begin{equation}\label{inte:def}
    \frac{|u_{0}^{[n]}(x)-u_{0}^{[n]}(y)|^{p_{n}}}{|x-y|^{N+s_{p_n}p_{n}}}=\frac{|u_{0}^{[n]}(x)-u_{0}^{[n]}(y)|^{p_{n}}}{|x-y|^{\frac{Np_n}{q}+s_qp_n}}\times\frac{1}{|x-y|^{\frac{N(q-p_n)}{q}-p_n(s_q-s_{p_n})}}.
\end{equation}
Next, using H\"older's inequality, and the fact that $\frac{p_nq(s_q-s_{p_n})}{q-p_n}=N$,
we obtain 
\begin{align*}
\int_{\Omega}\int_{\Omega}\frac{|u_{0}^{[n]}(x)-u_{0}^{[n]}(y)|^{p_{n}}}{|x-y|^{N+s_{p_n}p_{n}}}\,dx\,dy &\leq \left(\int_{\Omega}\int_{\Omega} {\abs{x-y}^{\frac{p_nq(s_q-s_{p_n})}{q-p_n}-N}}\,dx\,dy\right)^{\frac{q-p_n}{q}}\\
&\qquad\qquad\times\left(\int_{\Omega}\int_{\Omega}\frac{|u_{0}^{[n]}(x)-u_{0}^{[n]}(y)|^{q}}{|x-y|^{N+s_{q}q}}\,dx\,dy\right)^{\frac{p_n}{q}}\\\\
&\leq \abs{\Omega\times\Omega}^{\frac{q-p_n}{q}} \|u^{[n]}_{0}\|_{W_{0}^{s_{q},q}(\Omega)}^{p_n} \leq \max\{1, \abs{\Omega\times\Omega}^{\frac{q-1}{q}}\}\max\left\{1,M^{q}\right\},
\end{align*}
where $\abs{\Omega\times\Omega}$ denotes the Lebesgue measure of $\Omega\times\Omega$. Hence the sequence of integrands defined in \eqref{inte:def} is uniformly integrable. Combining this fact with the pointwise convergence established above  and applying the Vitali's convergence theorem, we obtain \textbf{Claim 1}. \\
\textbf{Claim 2:}
$$ \lim_{n\to\infty}\int_{\Omega}\int_{\mathcal{C}\Omega}\frac{|u_{0}^{[n]}(y)|^{p_{n}}}{|x-y|^{N+s_{p_n}p_{n}}}\,dx\,dy =  \int_{\Omega}\int_{\mathcal{C}\Omega}\frac{|u_{0}(y)|}{|x-y|^{N+s}}\,dx\,dy.$$
Next, we define a sequence 
$$h^{(n)}(y)=\int_{\mathcal{C}\Omega}\frac{1}{|x-y|^{N+s_{p_n}p_{n}}}\,dx,\qquad y\in\Omega.$$
The integrand $\left(\frac{1}{\abs{x-y}^{N+s}}\right)^{p_n}$ converges pointwise to $\frac{1}{\abs{x-y}^{N+s}}$ for all $x\in\mathcal{C}\Omega$ and $y\in\Omega$.
Since $\Omega$ is an open subset of $\mathbb{R}^{N}$, there exists $\delta_{y}:=\dist(y,\mathcal{C}\Omega)>0$ such that $|x-y|>\delta_{y}$. For a fix $y\in \Omega$, we can split $\mathcal{C}\Omega=Q_1\cup Q_2$, where $Q_1=\{x\in\mathcal{C}\Omega :\delta_{y}<\abs{x-y}\leq 1\}$
and $Q_2=\{x\in \mathcal{C}\Omega:\abs{x-y}>1\}$. For all $y\in\Omega$ and $x\in\mathcal{C}\Omega$, we have  
$$\frac{1}{|x-y|^{N+s_{p_n}p_{n}}}< G_{y}(x):=\frac{1}{\abs{x-y}^{N+q}}\chi_{Q_1}+\frac{1}{\abs{x-y}^{N+s}}\chi_{Q_2},$$
The combined dominating function $G_{y}\in L^1(\mathcal{C}\Omega)$ for all $y\in\Omega$ and it is independent of $n$. By the dominated convergence theorem, we have 
\begin{equation}
    \label{x-y pn converegence}
    \lim_{n\to\infty}\int_{\mathcal{C}\Omega}\frac{1}{|x-y|^{N+s_{p_n}p_{n}}}\,dx=\int_{\mathcal{C}\Omega}\frac{1}{|x-y|^{N+s}}\,dx.
\end{equation}
Due to the fact that $u_{0}^{[n]}(y)\to u_{0}(y)$ a.e. in $\Omega$ combining this with \eqref{x-y pn converegence}, we have $|u_{0}^{[n]}(y)|^{p_n}\int_{\mathcal{C}\Omega}\frac{1}{|x-y|^{N+s_{p_n}p_{n}}}$ converges pointwise to  $|u_{0}(y)|\int_{\mathcal{C}\Omega}\frac{1}{|x-y|^{N+s}}$ a.e $(x,y)\in \mathcal{C}\Omega\times\Omega$. Now it is only remain to prove the uniform integrability of $|u_{0}^{[n]}(y)|^{p_n}h^{(n)}(y)$ over $\Omega$. Further, we have
\begin{align}
\label{hn bound}
 h^{(n)}(y)=\int_{\mathcal{C}\Omega}\frac{1}{|x-y|^{N+s_{p_n}p_{n}}}\,dx& \leq\int_{\abs{z}>\delta_y}\frac{1}{\abs{z}^{N+s_{p_n}p_n}}\,dz=\frac{\omega_{N}}{s_{p_n}p_n}\delta_{y}^{-s_{p_n}p_n}<\frac{\omega_{N}}{s}\delta_{y}^{-s_{p_n}p_n},  
\end{align}
where $\omega_{N}$ denotes the volume of a unit ball in $\mathbb{R}^{N}$. Now, by using \eqref{hn bound}, the H\"older's inequality and fractional Hardy's inequality \cite[Theorem 2.1]{Bucur-2025}, we have 
\begin{align*}
    2\int_{\Omega}|u_{0}^{[n]}(y)|^{p_{n}}\int_{\mathcal{C}\Omega}\frac{1}{|x-y|^{N+s_{p_n}p_{n}}}\,dx\,dy&\leq \frac{2\omega_{N}}{s}\int_{\Omega}|u_{0}^{[n]}(y)|^{p_{n}}\delta_{y}^{-s_{p_n}p_n}\,dy\\
    &\leq\frac{2\omega_{N}}{s}\left(\int_{\Omega}\frac{|u_{0}^{[n]}(y)|^{q}}{\delta_{y}^{s_{q}q}}\,dy\right)^{\frac{p_n}{q}}\left(\int_{\Omega}\delta_{y}^{\frac{p_nq(s_q-s_{p_n})}{q-p_n}}\, dy\right)^{\frac{q-p_n}{q}}\\
    &\leq \frac{2\omega_{N}}{s}(C_{H})^{p_n}\|u_{0}^{[n]}\|_{W^{s_q,q}(\Omega)}^{p_n}\left(\int_{\Omega}\delta_{y}^{N}\, dy\right)^{\frac{q-p_n}{q}}\\
    &\leq \frac{2\omega_{N}}{s}(C_{H})^{p_n}\|u_{0}^{[n]}\|_{W_{0}^{s_q,q}(\Omega)}^{p_n}\left(\abs{\Omega}\mathrm{diam}(\Omega)^{N}\right)^{\frac{q-p_n}{q}}\\
    &\leq \frac{2\omega_{N}}{s}\max\left\{1,\left(C_{H}M\right)^{q}\right\}\max\left\{1,\left(\abs{\Omega}\text{diam}(\Omega)^{N}\right)^{\frac{q-1}{q}}\right\},
\end{align*}
where $C_{H}=C_{H}(N,s_q,q,\Omega)>0$ is a constant appearing in fractional Hardy's inequality. It implies that the sequence of integrand is uniformly integrable over $\Omega$. Combining this fact with the pointwise convergence of $\int_{\mathcal{C}\Omega}\frac{|u_{0}^{[n]}(y)|^{p_n}}{|x-y|^{N+s_{p_n}p_{n}}} \to \int_{\mathcal{C}\Omega}\frac{|u_{0}(y)|}{|x-y|^{N+s}}$ a.e. $(x,y)\in \mathcal{C}\Omega\times\Omega$ and applying the Vitali's convergence theorem once again, we have 
\begin{equation}
    \label{C omega converegnce}
    2\lim_{n\to\infty}\int_{\Omega}|u_{0}^{[n]}(y)|^{p_{n}}\int_{\mathcal{C}\Omega}\frac{1}{|x-y|^{N+s_{p_n}p_{n}}}\,dx\,dy=2 \int_{\Omega}|u_{0}(y)|\int_{\mathcal{C}\Omega}\frac{1}{|x-y|^{N+s}}\,dx\,dy.
\end{equation}
This proves \textbf{Claim 2}. Since $p_n\to 1^{+}$, from \eqref{convergence in Omega} and \eqref{C omega converegnce}, we prove \eqref{EQ:BBM formula}. Next, since $u_{0}^{[n]}\to u_{0}$ in $W_{0}^{s_q,q}(\Omega)$ from \cite[Lemma 2.4]{DLi-2024}, we have $u_{0}^{[n]}\to u_{0}$ in $W_{0}^{s,1}(\Omega)$. Therefore, by Lemma \ref{Nonlinearity convergence}, we have
\[
  \int_\Omega F(x,u_0^{[n]})\,dx\to\int_\Omega F(x,u_0)\,dx
  \quad \text{and}\quad
  \int_\Omega f(x,u_0^{[n]})u_0^{(n)}\,dx\to\int_\Omega f(x,u_0)u_0\,dx.
\]
Combining this with \eqref{EQ:BBM formula}, we obtain
\begin{equation}
    \label{E_p_converge}
    E^{(p_n)}(u_{0}^{[n]})\to E(u_{0})\quad \text{and}\quad I^{(p_n)}(u_{0}^{[n]})\to I(u_{0})\quad \text{as}\,\, n\to \infty.
\end{equation}
This proves our \textbf{Claim}.
\end{proof}
\subsection{Asymptotics of \texorpdfstring{$\lambda_p$}{lambda p} and \texorpdfstring{$d_p$}{d p}} 
    Next, we derive the asymptotics estimates on the critical point $\lambda_p$ obtained in he Nehari manifold analysis in Lemma \ref{existence of lambda*p} and the depth of the potential well $d_p$ as $p \to 1^+$.
   
\begin{lemma}
\label{boundedness of lambda}
  Assume that conditions \eqref{Assump:f_0}-\eqref{Assump:f_2} hold true. Let $q\in(1,\underline{p})$, $v \in \mathcal{N}\cap C_{c}^{\infty}(\Omega)$ and $\{\lambda_{p_n}\}_{n \in \mathbb{N}}$ is a sequence (obtained in Lemma \ref{existence of lambda*p}) such that $\lambda_{p_n}v\in\mathcal{N}^{(p_n)}.$ Then, 
  \begin{enumerate}
      \item[\textnormal{(i)}]  there exists  $m, M>0$ such that $m \leq \lambda_{p_n} \leq M$ for all $n \in \mathbb{N}$ and
      \item[\textnormal{(ii)}]  $\lambda_{p_n}\to 1$ where $p_n \to 1^+$ as $n \to \infty.$
  \end{enumerate} 
\end{lemma}
\begin{proof}
   From Lemma \ref{Lemma:2.3}(iii), it is easy to see that the set $\mathcal{N}\cap C_{c}^{\infty}(\Omega)\neq \emptyset$. Choose a sequence $p_n\in (1, q)$ such that $p_n\to 1^{+}$ as $n\to\infty.$ Since $p_n\to 1^{+}$ as $n\to\infty$, we can assume that $p_n<r$ for all $n\in\mathbb{N}$, where $r$ is given in \eqref{Assump:f_0}. Let $v\in \mathcal{N}\cap C_{c}^{\infty}(\Omega)$. By Lemma \ref{existence of lambda*p}(iii) there exists a unique $\lambda_{p_n}=\lambda_{p_n}(v,p_n)>0$  such that $\lambda_{p_n}v\in \mathcal{N}^{(p_n)}$, {\it i.e.}
\begin{equation}
    \label{lambda conv}
\lambda_{p_n}^{p_n}\iint_{Q}\frac{\abs{v(x)-v(y)}^{p_n}}{\abs{x-y}^{N+s_{p_n}p_n}}\,dx\,dy=\int_{\Omega}f(x,\lambda_{p_n} v)\lambda_{p_n} v\,dx.
\end{equation}
First, we show \text{(i)}. Assume that $\lambda_{p_n}\to\infty$ as $n \to \infty$ (upto a subsequence). From \eqref{lambda conv} and \eqref{F:bounds}, we obtain
\begin{equation*}
    \lambda_{p_n}^{p_n-1}\iint_{Q}\frac{\abs{v(x)-v(y)}^{p_n}}{\abs{x-y}^{N+s_{p_n}p_n}}\,dx\,dy=\int_{\Omega}f(x,\lambda_{p_n} v)v\,dx \geq B\theta\int_{\Omega\cap\{|v|\geq\frac{1}{\lambda_{p_n}}\}}\lambda_{p_n}^{\theta-1}\abs{v}^{\theta}\,dx.
\end{equation*}
 Since $v\in\mathcal{N}\cap C_{c}^{\infty}(\Omega)$, it follows that
 \begin{equation}
     \label{Lambda bounded above}
     \frac{\iint_{Q}\left(\frac{\abs{v(x)-v(y)}}{\abs{x-y}^{N+s}}\right)^{p_n}\,dx\,dy}{B\theta\int_{\Omega\cap\{|v|\geq\frac{1}{\lambda_{p_n}}\}}\abs{v}^{\theta}\,dx}\geq  \lambda_{p_n}^{\theta-p_n}.
 \end{equation}
Since $\theta>p_n$ and using \eqref{EQ:BBM formula}, the left hand side of the inequality \eqref{Lambda bounded above} remains bounded as $n\to \infty$,  while the right hand side diverges to $\infty$. This contradicts our assumption that $\lambda_{p_n}\to \infty$ as $n \to \infty.$ Therefore, that there exists a constant $M>0$ such that $\lambda_{p_n}\leq M$ for all $n\in\mathbb{N}.$
 Next, we assume that $\lambda_{p_n}\to 0$ as $n\to\infty$ (upto a subsequence). Fix $p_0\in(1,r)$ such that there exists $n_{0}\in \mathbb{N}$ with $p_n<p_{0}$ for all $n\geq n_{0}$. Dividing \eqref{lambda conv} by $\lambda_{p_n}^{p_0}$ and using \eqref{Assump:f_0}, we obtain
 \begin{align}
     \label{lambda bounded below}
      \lambda_{p_n}^{p_n-p_0}\iint_{Q}\left(\frac{\abs{v(x)-v(y)}}{\abs{x-y}^{N+s}}\right)^{p_n}\,dx\,dy&=\int_{\Omega}\frac{f(x,\lambda_{p_n} v)}{|\lambda_{p_n}v|^{p_{0}-1}}v|v|^{p_{0}-1}\,dx \leq \int_{\Omega}\frac{|f(x,\lambda_{p_n} v)|}{|\lambda_{p_n}v|^{p_{0}-1}}|v|^{p_{0}}\,dx\nonumber\\
      &\leq C\int_{\Omega}|v|^{p_{0}}|\lambda_{p_n} v|^{r-p_{0}}\, dx \leq C \lambda_{p_n}^{r-p_{0}}\int_{\Omega}|v|^{r}\,dx.
 \end{align}
 Since $v\in C_{c}^{\infty}(\Omega)$, and $p_{0} <r $, it follows that the right hand side of \eqref{lambda bounded below} converges to $0$ as $n\to \infty$. On the other hand, left hand side diverges to $\infty$ because $p_n<p_0$ for all $n\geq n_{0}$. This again contradicts our assumption that $\lambda_{p_n}\to 0$ as $n \to \infty.$ This proves \text{(i)}. Next, we derive \text{(ii)}. Assume, by contradiction, that $\lambda_{p_n} \not\to 1$. Then, there exists a subsequence $\{\lambda_{{p_n}_k}\}$ which does not converges to $1$. As $\{\lambda_{{p_n}_k}\}_{k\in \mathbb{N}} \subset [m, M]$ is bounded, the Bolzano--Weierstrass theorem guarantees the existence of a further subsequence $\{\lambda_{{p_n}_{k_j}}\}$ such that
\[
\lambda_{{p_n}_{k_j}} \to \bar{\lambda} \in [m, M].
\]
We now pass limit in \eqref{lambda conv} along the subsequence $\{\lambda_{n_{k_j}}\}$. For $v\in C_{c}^{\infty}(\Omega)$, by using Lemma \ref{Nonlinearity convergence} and Lemma \ref{strong convergence u_0}, we obtain
\begin{align}
    \label{lambda converges to 1}
    \bar{\lambda}\iint_{Q}\frac{\abs{v(x)-v(y)}}{\abs{x-y}^{N+s}}\,dx\,dy=\int_{\Omega}f(x,\bar{\lambda} v)\bar{\lambda} v\,dx.
\end{align}
On the other hand, since $v \in \mathcal{N}\cap C_{c}^{\infty}(\Omega)$, the above identity holds for $\bar{\lambda} = 1$. Moreover, the map
\[
\lambda \mapsto \int_\Omega f(x,\lambda v)\,v\,dx
\]
is strictly increasing, in view of assumption \eqref{Assump:f_2}, since
\[
\frac{d}{d\lambda} \int_\Omega f(x,\lambda v)\,v\,dx
= \int_\Omega f'(x,\lambda v)\,v^2\,dx > 0.
\]
Therefore, the above equation admits a unique solution, which necessarily implies that $\bar{\lambda} = 1$. This contradicts our assumption. Hence, our claim. 
\end{proof}
  \begin{lemma}
 \label{dp id bounded}
 Let $f$ satisfies the condition \eqref{Assump:f_0}-\eqref{Assump:f_2}, then there exists $q \in (1, \underline{p})$ and $M>0$ such that $d_{p}\leq M$ for all $p\in (1,q]$.
\end{lemma}
\begin{proof}
For $\varphi\in C_{c}^{\infty}(\Omega)\setminus\{0\}$ and $p\in (1, \underline{p})$, first we prove that
\begin{equation*}
    \iint_{Q}\frac{|\varphi(x)-\varphi(y)|^{p}}{|x-y|^{N+s_{p}p}}\,dx\,dy \to \iint_{Q}\frac{|\varphi(x)-\varphi(y)|}{|x-y|^{N+s}}\,dx\,dy \quad \text{ as } p\to 1^{+}.
\end{equation*}
 The choice of $s_{p}$ in \eqref{choice:sp} allows us to write the integrand as
\[
  \frac{\abs{\varphi(x)-\varphi(y)}^{p}}{\abs{x-y}^{N+s_{p}p}}=
  \left(\frac{\abs{\varphi(x)-\varphi(y)}}{\abs{x-y}^{N+s}}\right)^{p} \to \frac{\abs{\varphi(x)-\varphi(y)}}{\abs{x-y}^{N+s}} \quad \text{as} \ p \to 1^+ \quad \text{a.e. in} \ Q.
\]
Let $M_{1}:=\|\nabla \varphi\|_{L^\infty{(\Omega)}}<\infty$.
Split $Q=Q_1\cup Q_2$, where $Q_1=\{(x,y)\in Q:\abs{x-y}\leq 1\}$
and $Q_2=\{(x,y)\in Q:\abs{x-y}>1\}$. For $Q_1$, the mean value theorem gives,
$\abs{\varphi(x)-\varphi(y)}\leq M_{1}\abs{x-y}$. Now, for $p \in (1, \underline{p})$, we can choose $\epsilon>0$ (independent of $p$) such that $N+ p s_{p} -p \leq N-\epsilon$. This further gives
\[
  \frac{\abs{\varphi(x)-\varphi(y)}^{p}}{\abs{x-y}^{N+s_{p}p}}
  \leq\frac{\max\{1,M^{\underline{p}}_{1}\}}{\abs{x-y}^{N-\epsilon}}\in L^{1}(Q_1),
\]
Since $\abs{\varphi(x)-\varphi(y)}\leq 2\|\varphi\|_{L^\infty(\Omega)}=:M_{2}$ and
$s_p\in(s,1)$ for $p\in (1, \underline{p})$. For $(x, y) \in Q_2$, we have 
\[ \abs{x-y}^{N+s_{p}p}\geq\abs{x-y}^{N+s} \quad\text{ and }\quad
   \frac{\abs{\varphi(x)-\varphi(y)}^{p}}{\abs{x-y}^{N+s_{p}p}}\leq\frac{\max\{1,M_{2}^{\underline{p}}\}}{\abs{x-y}^{N+s}}\in L^1(Q_2).
\]
The combined dominating function
$$G(x,y):=\frac{\max\{1,M_{1}^{\underline{p}}\}}{\abs{x-y}^{N-\epsilon}}\chi_{Q_1}
          +\frac{\max\{1,M_{2}^{\underline{p}}\}}{\abs{x-y}^{N+s}}\chi_{Q_2},$$
belongs to $L^1(Q)$ and is independent of $p$.
By applying Lebesgue dominated convergence theorem, we have 
\begin{equation}
    \label{EQ:BBM formula for fix varphi}
    \iint_{Q}\frac{\abs{\varphi(x)-\varphi(y)}^{p}}{|x-y|^{N+s_{p}p}}\,dx\,dy \to \iint_{Q}\frac{\abs{\varphi(x)-\varphi(y)}}{|x-y|^{N+s}}\,dx\,dy \quad \text{ as } p\to 1^{+}.
\end{equation}
For $\varphi \in C_{0}^{\infty}(\Omega)\setminus\{0\}$, by Lemma \ref{existence of lambda*p}(iii), for each $p\in (1,\underline{p})$ in view of Remark \ref{rem:condi}, there exists $\lambda_{*,p}>0$ such that $\lambda_{*,p}\varphi\in \mathcal{N}^{(p)}$ {\it i.e. 
$$(\lambda_{*,p})^p\iint_{Q}\frac{|\varphi(x)-\varphi(y)|^{p}}{|x-y|^{N+s_pp}}\,dx\,dy = \int_{\Omega} f(x,\lambda_{*,p}\varphi) \lambda_{*,p}\varphi \,dx.$$}
From \eqref{F:bounds} and \eqref{Assump:f_1}, we have 
    \begin{align*}
    \frac{\iint_{Q}\left(\frac{|\varphi(x)-\varphi(y)|}{|x-y|^{N+s}}\right)^{p}\,dx\,dy}{B\theta\int_{\Omega\cap\{|\varphi|>\frac{1}{\lambda_{*,p}}\}}|\varphi|^{\theta}\,dx} &\geq \left(\lambda_{*,p}\right)^{\theta-p}.
    \end{align*}
     Since $\theta>p$ and $\varphi \in C_{0}^{\infty}(\Omega)\setminus\{0\}$ together with \eqref{EQ:BBM formula for fix varphi} and Lemma \ref{boundedness of lambda}, it follows that $\lambda_{*,p}$ remains bounded as $p\to 1^{+}$. Now from the definition of $d_{p}$ and \eqref{Assump:f_1}, we have 
    \begin{align}
    \label{alpha p}
        d_{p}\leq E^{(p)}(\lambda_{*,p}\varphi)\leq \frac{(\lambda_{*,p})^{p}}{p}\iint_{Q} \frac{|\varphi(x)-\varphi(y)|^{p}}{|x-y|^{N+s_pp}}\,dx\,dy.
    \end{align}
    Since $\lambda_{*,p}$ is bounded as $p\to 1^{+}$ and $\varphi\in C_{0}^{\infty}(\Omega)\setminus\{0\}$, from \eqref{alpha p} and \eqref{EQ:BBM formula for fix varphi} we deduce that there exists $q\in (1, \underline{p})$ and $M>0$ such that $d_{p}\leq M$ for all $p\in (1,q].$
    \end{proof}
    \begin{lemma}
    \label{Dpn converges}
       Let conditions \eqref{Assump:f_0}-\eqref{Assump:f_2} hold. Then $d_{p_n}\to d$ where $p_n \to 1^{+}$ as $n \to \infty.$
    \end{lemma}
    \begin{proof}
     Choose a sequence $p_n\in (1, q)$ such that $p_n\to 1^{+}$ as $n\to\infty$. Let $v\in\mathcal{N}\cap C_{c}^{\infty}(\Omega)$. From Lemma \ref{existence of lambda*p}(iii), there exists a unique $\lambda_{p_n}>0$ such that $\lambda_{p_n}v\in \mathcal{N}^{(p_n)}$ and
\[
d_{p_n} \leq E^{(p_n)}(\lambda_{p_n} v)
= \frac{\lambda_{p_n}^{p_n}}{p_n}\iint_{Q}\frac{\abs{v(x)-v(y)}^{p_n}}{\abs{x-y}^{N+s_{p_n}p_n}} \,dx\,dy
- \int_\Omega F(x,\lambda_{p_n} v)\,dx.
\]
From Lemma \ref{strong convergence u_0} and Lemma \ref{boundedness of lambda}, we note that $\lambda_{p_n}^{p_n} \to 1$ as $n \to \infty$ and 
\[
E^{(p_n)}(\lambda_{p_n} v)\to E(v)\quad\text{as}\quad n\to\infty.
\]
Consequently,
\begin{equation}
    \label{limsup dpn and Ev}
    \limsup_{n \to \infty} d_{p_n} \leq E(v)
\quad \text{for all } v \in \mathcal{N} \cap C_c^\infty(\Omega).
\end{equation}
Next, we claim that
\[
\inf_{v \in \mathcal{N} \cap C_c^\infty(\Omega)} E(v) = d.
\]
The inequality $\inf_{v\in\mathcal{N} \cap C_c^\infty(\Omega)} E(v) \geq d$ is immediate, since $\mathcal{N} \cap C_c^\infty(\Omega) \subset \mathcal{N}$. 

To prove the reverse inequality, let $\{w_k\}_{k \in \mathbb{N}} \subset \mathcal{N}$ be such that $E(w_k) \to d$.  Since $C_c^\infty(\Omega)$ is dense in $W^{s,1}_0(\Omega)$, for each $k$, approximate $w_k$ by a sequence $\{\phi_{k,j}\}_{j\in\mathbb{N}} \subset C_c^\infty(\Omega)$ such that $\phi_{k,j} \to w_k$ strongly in $W^{s,1}_0(\Omega)$. For each $j$, let $\mu_{k,j} > 0$ be the unique scalar given by Lemma \ref{Lemma:2.3}(iii) such that $\mu_{k,j}\phi_{k,j} \in \mathcal{N}$, {\it i.e.}

\begin{equation}
    \label{ukj convergnece}
    \iint_{Q}\frac{\abs{\phi_{k,j}(x)-\phi_{k,j}(y)}}{\abs{x-y}^{N+s}}\,dx\,dy=\int_{\Omega}f(x, \mu_{k,j}\phi_{k,j})\phi_{k,j}\,dx.
\end{equation}
Since $\phi_{k,j} \to w_k$ strongly in $W^{s,1}_0(\Omega)$ arguing as in the case of $\lambda_{p_n}$ (with $\phi_{k,j} \to w_k$ in place of a fixed $v$) with \eqref{Assump:f_0}, Proposition \ref{prop:embedding}(i) and \eqref{Assump:f_2} we deduce that $\mu_{k,j} \to 1$ as $j \to \infty$. It follows that
\[
E(\mu_{k,j}\phi_{k,j}) \to E(w_k) \quad \text{as } j \to \infty.
\]
Hence, choosing $j = j(k)$ sufficiently large, we obtain
\[
\inf_{v \in \mathcal{N} \cap C_c^\infty(\Omega)} E(v)
\leq E(\mu_{k,j(k)}\phi_{k,j(k)})
\leq E(w_k) + \frac{1}{k} \to d.
\]
Therefore,
\begin{equation}
    \label{lim sup equality}
    \inf_{v \in \mathcal{N} \cap C_c^\infty(\Omega)} E(v) = d.
\end{equation}
Combining \eqref{lim sup equality} with \eqref{limsup dpn and Ev}, we obtain
\begin{equation}
\label{d_p_n:limsup}
\limsup_{n \to \infty} d_{p_n} \leq d.
\end{equation}
Next, for each $n \in \mathbb{N}$, let $u_n \in \mathcal{N}^{(p_n)}$ be such that
\begin{equation}\label{u_n:min sq}
E^{(p_n)}(u_n) \leq d_{p_n} + \frac{1}{n}.
\end{equation}
From \cite[Lemma 2.4]{DLi-2024}, it follows that $u_n \in W^{s,1}_0(\Omega)\setminus\{0\}$.
From \eqref{u_n:min sq}, assumption \eqref{Assump:f_1} and Lemma \ref{dp id bounded} we have
\[
\left(\frac{1}{p_n}-\frac{1}{\theta}\right)\|u_n\|_{W^{s_{p_n},p_n}_{0}(\Omega)}^{p_n} \leq E^{(p_n)}(u_n)\leq d_{p_n}+\frac{1}{n}\leq M+1. \]
Since $p_n\in (1,q)$ and $\theta>q$, it follows that
$$\|u_n\|_{W^{s_{p_n},p_n}_{0}(\Omega)}\leq \left(\frac{(M+1)(\theta p_n)}{\theta-p_n}\right)^{\frac{1}{p_n}}\leq \left(\frac{(M+1)(\theta q)}{\theta-q}\right)=:C_4. $$
Using \cite[Lemma 2.4]{DLi-2024}, it follows that $\{u_n\}_{n\in\mathbb{N}}$ is uniformly bounded in $W^{s,1}_0(\Omega)$.
From Proposition \ref{prop:embedding}(ii) for all $r \in [1,1_s^*)$, there exist a subsequence of $\{u_n\}_{n\in\mathbb{N}}$ (still denoted by $\{u_n\}_{n\in\mathbb{N}}$) and a function $u \in L^r(\Omega)$ such that
\[
u_n \to u \quad \text{in } L^r(\Omega) \qquad\text{and}
\qquad
u_n(x) \to u(x) \quad \text{a.e. in } \Omega.
\]
Next, we show that $u \neq 0.$ Suppose by contradiction that $u\equiv 0$. From the identity $I^{(p_n)}(u_n)=0$ and assumption \eqref{Assump:f_0}, we obtain
\begin{equation}
    \label{u is nonzero}
    \|u_n\|_{W^{s_{p_n},p_n}_{0}(\Omega)}^{p_n}
\leq C\|u_n\|_{L^{r}(\Omega)}^r\to 0,
\end{equation}
then from \eqref{u is nonzero}, Lemma \ref{Nonlinearity convergence} and Lemma \ref{Positive:depth dp}, we have 
$$0< C< d_{p_n}\leq E^{(p_n)}(u_n)=\frac{1}{p_n}\|u_n\|_{W^{s_{p_n},p_n}_{0}(\Omega)}^{p_n}-\int_{\Omega}F(x,u_n)\,dx \to 0.$$
This contradicts the fact that $u_n\in \mathcal{N}^{(p_n)}$. It follows that $u \not\equiv 0$.
Using the fact that $I^{(p_n)}(u_n)=0$, we write
\begin{equation}
    \label{Epn inequality}
    E^{(p_n)}(u_n)
= \frac{1}{p_n} \int_\Omega f(x,u_n)u_n\,dx
- \int_\Omega F(x,u_n)\,dx.
\end{equation}
Using $u_n\to u$ in $L^r(\Omega)$ and Lemma \ref{Nonlinearity convergence}, we obtain
\begin{equation}
    \label{f(x,un) and vitali}
    \int_\Omega f(x,u_n)u_n\,dx \to \int_\Omega f(x,u)u\,dx
\quad \text{and} \quad
\int_\Omega F(x,u_n)\,dx \to \int_\Omega F(x,u)\,dx.
\end{equation}
Since ${p_n} \to 1^{+}$, from \eqref{Epn inequality} and \eqref{f(x,un) and vitali}, it follows that
\begin{equation}\label{Lim :Epn(u_n)}
\lim_{n\to\infty} E^{(p_n)}(u_n)
= \int_\Omega f(x,u)u\,dx - \int_\Omega F(x,u)\,dx.
\end{equation}
Using pointwise convergence and Fatou's lemma, we obtain
\begin{equation}\label{eq:Fatou}
\|u\|_{W^{s,1}_{0}(\Omega)}
\leq \liminf_{n\to\infty} \|u_n\|_{W^{s,1}_{0}(\Omega)}.
\end{equation}
On the other hand, from $I^{(p_n)}(u_n)=0$, we have
\[
\lim_{n\to\infty}\|u_n\|_{W^{s_{p_n},p_n}_{0}(\Omega)}^{p_n}
= \lim_{n\to\infty}\int_\Omega f(x,u_n)u_n\,dx = \int_\Omega f(x,u)u\,dx,
\]
which implies $\lim_{n\to\infty}\|u_n\|_{W^{s_{p_n},p_n}_{0}(\Omega)} = \int_\Omega f(x,u)u\,dx$. Using \cite[Lemma 2.4]{DLi-2024}, we get
\[
\lim_{n\to\infty}\|u_n\|_{W^{s,1}_{0}(\Omega)}
\leq \lim_{n\to\infty}C_{1}^{\frac{1-p_n}{p_n}} \|u_n\|_{W^{s_{p_n},p_n}_{0}(\Omega)} = \int_\Omega f(x,u)u\,dx.
\]
Therefore,
\begin{equation}\label{eq:limsupnorm}
\limsup_{n\to\infty} \|u_n\|_{W^{s,1}_{0}(\Omega)} \leq \int_\Omega f(x,u)u\,dx.
\end{equation}
Combining \eqref{eq:Fatou} and \eqref{eq:limsupnorm}, we conclude that
\[
\|u\|_{W^{s,1}_{0}(\Omega)} \leq  \int_\Omega f(x,u)u\,dx.
\]
Hence $I(u) \leq 0$ and from \eqref{eq:Fatou}, we have  $u \in W^{s,1}_0(\Omega)\setminus\{0\}$. Further Lemma \ref{Lemma:2.3}(iii) ensures the existence of a unique $\mu^*>0$ such that $\mu^* u \in \mathcal{N}$. Moreover,  $I(\mu u)\geq 0$ for all $0\leq \mu\leq \mu^{*}$ and $I(\mu u)\leq 0$ for all $\mu\geq \mu^{*}$, it follows that $\mu^* \leq 1$. Define $\Psi(x,\mu) := \mu u(x)f(x,\mu u(x)) - F(x,\mu u(x))$ for $\mu\in [0,\infty)$ and $x \in \Omega$. Then from condition \eqref{Assump:f_2}, we obtain
\[
\frac{\partial \Psi}{\partial \mu}
= \mu u(x)^2 f'(x,\mu u(x)) \geq 0.
\]
 Hence $\Psi(x,\cdot)$ is non-decreasing for $\mu\in [0,\infty)$. Since $\mu^*u\in \mathcal{N}$ and $\mu^{*} \leq 1$, we have

\begin{align}
\label{eq:monotone}
E(\mu^* u)&= \mu^{*}\iint_{Q}\frac{\abs{u(x)-u(y)}}{\abs{x-y}^{N+s}}\,dx\,dy-\int_{\Omega}F(x,\mu^{*}u)\,dx\nonumber\\
&=\int_{\Omega}\left(f(x,\mu^{*}u)\mu^{*}u - F(x,\mu^{*}u)\right)\,dx \leq \int_\Omega f(x,u)u\,dx - \int_\Omega F(x,u)\,dx.
\end{align}
Combining \eqref{eq:monotone}, \eqref{Lim :Epn(u_n)}, and \eqref{u_n:min sq}, we obtain
\[
d \leq E(\mu^* u)
\leq \int_\Omega f(x,u)u\,dx - \int_\Omega F(x,u)\,dx
= \lim_{n\to\infty} E^{(p_n)}(u_n)
\leq \liminf_{n\to\infty} d_{p_n}.
\]
Therefore, $\liminf_{n\to\infty} d_{p_n} \geq d.$ Combining this with \eqref{d_p_n:limsup}, we conclude that
\begin{equation}
    \label{convergence of dpn to d}
    \lim_{n\to \infty}d_{p_n} = d.
\end{equation}
\end{proof}
\begin{corollary}
\label{u0 in Wpn}
Let conditions \eqref{Assump:f_0}-\eqref{Assump:f_2} hold and $u_{0}\in W_{0}^{s_{q},q}(\Omega)$. Let $\{u^{[n]}_{0}\}_{n\in\mathbb{N}}\subset W_{0}^{s_{q},q}(\Omega)$ such that $u^{[n]}_{0} \to u_0 $ in $W_{0}^{s_{q},q}(\Omega)$, and $E(u_0)<d$ and $I(u_0)>0$. Then there exists $n_{0}\in\mathbb{N}$ such that 
       $$E^{(p_n)}(u_{0}^{[n]})< d_{p_n} \quad\text{and}\quad I^{(p_n)}(u_{0}^{[n]})>0\qquad \forall\, n\geq n_{0}.$$   
\end{corollary}
\begin{proof}
Since $E(u_{0})<d$ and $I(u_{0})>0$, let $\delta := d - E(u_0) > 0$. From Lemma \ref{strong convergence u_0} and Lemma \ref{Dpn converges}, we have $E^{(p_n)}(u_0^{[n]}) \to E(u_0)$ and $d_{p_n} \to d$ as $n\to\infty$. It implies that there exist $n_1,n_2\in \mathbb{N}$ such that
$$\abs{E^{(p_n)}(u_0^{[n]})- E(u_0)}<\frac{\delta}{4} \quad\forall\, n\geq n_1,\qquad \abs{d_{p_n}-d}<\frac{\delta}{4}\quad\forall\, n\geq n_2.$$
For $\overline{n}=\max\{n_1,n_2\}$, we have
\begin{align*}
    E^{(p_n)}(u_0^{[n]})-d_{p_n}&=E^{(p_n)}(u_0^{[n]})-E(u_0)+E(u_0)-d+ d-d_{p_n}<\frac{\delta}{4}-\delta+\frac{\delta}{4}<\frac{-\delta}{2}.
\end{align*}
This implies that for $n\geq \overline{n}$, we have 
\begin{equation*}
    E^{(p_n)}(u_0^{[n]}) < d_{p_n}.
\end{equation*}  
Similarly we can prove that there exists $\tilde{n}\in\mathbb{N}$ such that
$I^{(p_n)}(u_0^{[n]}) > 0,\,\,\forall\, n\geq\tilde{n}$. Hence, the claim.
\end{proof}
\section{Global existence of weak/strong solutions}
\label{Global existence of weak/strong solutions}
\subsection{On the case of low initial energy:  \texorpdfstring{$E(u_{0})<d$}{E(u0)<d}}
In this subsection, we address the case of low initial energy. We establish the global existence of weak solution and strong solution.

 \begin{proof}[\textbf{Proof of the Theorem \ref{Main Theorem low initial energy}}]
Choose a sequence $p_n\in (1, q)$ such that $p_n\to 1^{+}$. Let $u_{0}\in W_{0}^{s_{q},q}(\Omega)\cap L^{2}(\Omega)$ such that $E(u_{0})<d$ and $I(u_{0})>0$. From Lemma \ref{strong convergence u_0} there exists a sequence $\{u_{0}^{[n]}\}_{n\in\mathbb{N}}\subset W_{0}^{s_q,q}(\Omega)\cap L^{2}(\Omega)$ such that $u_{0}^{[n]}\to u_{0}$ in $W_{0}^{s_{q},q}(\Omega)\cap L^{2}(\Omega)$ as $n\to\infty$. Using $u_0 \in W$ and Corollary \ref{u0 in Wpn}, there exists $n_0 \in \mathbb{N}$ such that 
\begin{equation}
    \label{Ep assumption}
    E^{(p_{n})}(u^{[n]}_{0})<d_{p_n}\qquad \text{and}\qquad I^{(p_n)}(u^{[n]}_{0})>0 \quad \text{for all $n\geq n_{0}$}.
    \end{equation}
Now, by applying Lemma \ref{Approximate problem main thm}  for each $n\geq n_{0}$, there exists a global weak solution $u^{(p_{n})}$ to the problem \eqref{main:problem2} for $p=p_n$ such that
$$\int_{0}^{t}||u^{(p_n)}_t(\cdot,\tau) ||^{2}_{L^{2}(\Omega)} ~d\tau <d_{p_{n}},\,\,\|u^{(p_{n})}(\cdot,t)\|_{W_{0}^{s_{p_n}, p_{n}}(\Omega)}\leq \left(\frac{d_{p_{n}}p_{n}\theta}{\theta-p_{n}}\right)^{\frac{1}{p_{n}}} \quad\text{for all } t\in [0,+\infty)$$
and
\begin{equation}
    \label{approximate sol L2 bound}
    \|u^{(p_{n})}(\cdot,t)\|_{L^{2}(\Omega)}\leq \|u_{0}^{[n]}\|_{L^{2}(\Omega)}+o(1) \quad\text{for all } t\in [0,+\infty).
\end{equation}
Using $\theta>q$ and the fact that $d_{p_{n}}$ is bounded, by Lemma \ref{dp id bounded} there exists $M>0$ such that
\begin{equation}
\label{bound upn}
    \int_{0}^{t}||u^{(p_n)}_t(\cdot,\tau) ||^{2}_{L^{2}(\Omega)}\,d\tau <M,\,\,\|u^{(p_{n})}(\cdot,t)\|_{W_{0}^{s_{p_n},p_{n}}(\Omega)}^{p_n} \leq \frac{Mq\theta}{\theta-q}\quad\text{for all } t\in [0,+\infty).
\end{equation}
\textbf{Claim 1:} Uniform boundedness of $u^{(p_{n})}$ in $W_{0}^{s,1}(\Omega)$ for all $t\in (0,+\infty).$\\
For $n \geq n_0$, we have $u^{(p_n)}\in L^{\infty}(0,+\infty;W^{s_{p_n},p_n}_{0}(\Omega)\cap L^{2}(\Omega))$ and from \eqref{bound upn}, we have
\begin{equation}
    \label{bound of upn in Wsp}
\iint_{Q}\frac{|u^{(p_n)}(x,t)-u^{(p_n)}(y,t)|^{p_n}}{|x-y|^{N+s_{p_n}p_n}}\,dx\,dy \leq \frac{Mq\theta}{\theta-q},
\end{equation}
and 
\begin{align}
    \label{existence of up}
     \int_{\Omega}u^{(p_n)}_{t}\phi \,dx+\iint_{Q}\frac{|u^{(p_n)}(x,t)-u^{(p_n)}(y,t)|^{p_n-2}(u^{(p_n)}(x,t)-u^{(p_n)}(y,t))}{|x-y|^{N+s_{p_n}p_n}}&(\phi(x)-\phi(y))\,dx\,dy\nonumber\\
     &=\int_{\Omega}f(x,u^{(p_n)})\phi \,dx
\end{align}
for all
$\phi\in W_{0}^{s_{p_n},p_n}(\Omega)\cap L^{2}(\Omega)$ and a.e. $t\in [0,+\infty).$ From \cite[Lemma 2.4]{DLi-2024} and using the fact that $p_n\in (1,q)$,  we have 
\begin{align}
\label{u in Ws1}
\|u^{(p_{n})}(\cdot,t)\|_{W_{0}^{s,1}(\Omega)}&\leq C^{\frac{1-p_{n}}{p_n}}_{1}\,\|u^{(p_{n})}(\cdot,t)\|_{W_{0}^{s_{p_n},p_{n}}(\Omega)} \leq \max\{1,C^{\frac{1-q}{q}}_{1}\}\,\|u^{(p_{n})}(\cdot,t)\|_{W_{0}^{s_{p_n},p_{n}}(\Omega)}
\end{align}
From \eqref{bound upn}, \eqref{u in Ws1} and \eqref{approximate sol L2 bound}, we have 
$$\|u^{(p_{n})}(\cdot,t)\|_{W_{0}^{s,1}(\Omega)}\leq C_{q,\theta}\quad \text{and}\quad u^{(p_{n})}\in L^{\infty}(0,+\infty;W_{0}^{s,1}(\Omega)\cap L^{2}(\Omega)).$$
where $C_{q,\theta}:= \max\{1,C^{\frac{1-q}{q}}_{1}\}\,\max\left\{1,\left(\frac{Mq\theta}{\theta-q}\right) \right\}$. Furthermore, using condition \eqref{Assump:f_0} and Proposition \ref{prop:embedding}(i), we obtain
\begin{align}
\label{f(x,u^{p_n}) uniform bound}
    \int_{\Omega}|f(x,u^{(p_n)}(\cdot,t))|^{\frac{r}{r-1}}\,dx \leq C^{\frac{r}{r-1}}\int_{\Omega}|u^{(p_n)}(\cdot,t)|^{r}\,dx \leq C^{\frac{r}{r-1}} (C_{Sob})^{r}\|u^{(p_n)}(\cdot,t)\|^{r}_{{W_{0}^{s,1}(\Omega)}} \leq C^{\frac{r}{r-1}} \left(C_{Sob}C_{q,\theta}\right)^{r}.
    \end{align}
    
This establishes the existence of $u \in L^{\infty}(0,+\infty;W_{0}^{s,1}(\Omega)\cap L^{2}(\Omega))$ and a subsequence of $\{u^{(p_n)}\}_{n\in \mathbb{N}}$ (still denoted by $\{u^{(p_n)}\}_{n\in \mathbb{N}}$) such that, as $n \to \infty$, we have
 \begin{equation}
 \label{weak convergence of upn}
 \left.
 \begin{aligned}
 u^{(p_n)}&\stackrel{\ast}{\rightharpoonup} u&&\text{ in }L^{\infty}(0,+\infty;L^{2}(\Omega)),\\
 u_{t}^{(p_n)}&\weak u_{t} &&\text{ in }L^{2}(0,+\infty;L^{2}(\Omega)),\\
 f(x,u^{(p_n)})&\stackrel{\ast}{\rightharpoonup}\xi && \text{ in }L^{\infty}(0,+\infty;L^{\frac{r}{r-1}}(\Omega)).
 \end{aligned}
 \right\}
 \end{equation}
Then, again from Aubin-Lions compactness theorem\cite[Corollary 4, p. 85]{simon} for any $T>0$, we get
 \begin{equation}
 \label{stong:conv of upn}
 \begin{aligned}
 u^{(p_n)}&\to u \text{ in } C([0,T];L^{q}(\Omega)),\text{  as } n \to \infty \text{ for all } q\in[1,1^{*}_{s}),
 \end{aligned}
  \end{equation}
  and so $\xi=f(x,u).$ Moreover, we have
  $$u^{(p_{n})}(x,t)\to u(x,t) \quad \text{a.e.} \quad (x,t)\in \Omega\times(0,+\infty).$$
  \textbf{Claim 2:} There exists $\eta(\cdot,\cdot,t)\in L^{\infty}(\mathbb{R}^{N}\times\mathbb{R}^{N}),\, \eta(x,y,t)=-\eta(y,x,t)$ for almost all $(x,y)\in \mathbb{R}^{N}\times\mathbb{R}^{N},\,\|\eta(\cdot,\cdot,t)\|_{L^{\infty}(\mathbb{R}^{N}\times\mathbb{R}^{N})}\leq 1$, such that following equality holds:
\begin{align*}
   \int_{\Omega}u_{t}\phi\,dx+\iint_{Q}\eta(x,y,t)\frac{\phi(x)-\phi(y)}{|x-y|^{N+s}}\,dx\,dy=\int_{\Omega}f(x,u)\phi \,dx \quad  \text{for all $\phi\in W_{0}^{s,1}(\Omega)\cap L^{2}(\Omega)$}
\end{align*}
and a.e. $t\in[0,+\infty)$. For $k>0$, we set
$$C_{p_{n},k}:=\left\{ (x,y)\in \mathbb{R}^{N}\times\mathbb{R}^{N}:\,\left|\frac{u^{(p_n)}(x,t)-u^{(p_n)}(y,t)}{|x-y|^{N+s}}\right|>k\right\}.$$
Then, by using the Chebyshev's inequality and \eqref{bound of upn in Wsp}, we obtain
\begin{equation}
    \label{measure of C_p}
    |C_{p_{n},k}|\leq\frac{1}{k^{p_n}}\iint_{Q}\abs{\frac{u^{(p_n)}(x,t)-u^{(p_n)}(y,t)}{|x-y|^{N+s}}}^{p_n}\,dx\,dy \leq \frac{C_{3}}{k^{p_{n}}},
\end{equation}
where $C_3:=\frac{Mq\theta}{\theta-q}$. On the other hand, for all $(x,y) \in \R^N \times \R^N $ and $t\geq 0$, we have
\[
\Phi(x,y,t, p_{n}, u^{(p_{n})}) \chi_{C^c_{p_{n},k}}(x,y) :=\left| \frac{u^{(p_{n})}(x,t) - u^{(p_{n})}(y,t)}{|x - y|^{N+s}} \right|^{p_{n}-2} \frac{u^{(p_{n})}(x,t) - u^{(p_{n})}(y,t)}{|x - y|^{N+s}} \chi_{C^c_{p_{n},k}}(x,y) \leq k^{p_{n}-1}.
\]
Therefore, for any $k \in \mathbb{N}$ there exists a subsequence of $\{p_n\}_{n\in \mathbb{N}}$, denoted by $\{p_{n_{j}}\}$, such that
\[
\Phi(x,y,t, p_{n_j}, u^{(p_{n_j})}) \chi_{C^c_{p_{n_j}}, k} (x,y)\stackrel{\ast}{\rightharpoonup} \eta_k(x,y,t),
\]
in $L^\infty (\R^N \times \R^N)$, with $\eta_k$ antisymmetric such that $\|\eta_k(\cdot,\cdot,t)\|_{L^\infty (\R^N \times \R^N)} \leq 1$ for all $t\geq 0$. Now there exists a subsequence $\{\eta_{k_m}\}$ of $\eta_{k}$ and an antisymmetric function $\eta \in L^\infty (\R^N \times \R^N)$ such that
\[
\eta_{k_m}(\cdot,\cdot,t) \stackrel{\ast}{\rightharpoonup} \eta(\cdot,\cdot,,t) \quad \text{in}\quad L^\infty (\R^N \times \R^N), \quad \text{and} \quad  \|\eta(\cdot,\cdot,t)\|_{L^\infty (\R^N \times \R^N)} \leq 1.
\]
In order to pass limit in \eqref{existence of up}, let us choose first $\phi\in C_{c}^{\infty}(\Omega).$ Then, for a fixed $q_{0}\in(1,\frac{N}{N+s-1})$, choose $r_{0}=\frac{(N+s)q_0-N}{q_0}$ such that $(N+s)q_{0}=N+r_{0}q_{0}$ and $\phi\in W_{0}^{r_{0},q_{0}}(\Omega)$. Let us fix $k \in \mathbb{N} $. From \eqref{existence of up}, we have
\begin{align}
\label{eta estimate}
&\iint_{Q} \Phi(x,y,t, p_{n_j}, u^{(p_{n_j})}) \chi_{C^c_{p_{n_j}}, k} (x,y) \frac{\phi(x) - \phi(y)}{|x-y|^{N+s}} \,dx\,dy - \int_{\Omega} (f(x,u^{(p_{n_j})}) - u_{t}^{(p_{n_j})}) \phi\,dx\nonumber\\
&= -\iint_{Q}\Phi(x,y,t, p_{n_j}, u^{(p_{n_j})}) \chi_{C_{p_{n_j}}, k} (x,y) \frac{\phi(x) - \phi(y)}{|x-y|^{N+s}}\,dx\,dy.
\end{align}
We will find the upper bound of the right-hand side of \eqref{eta estimate}. Now, for $p_{n_{j}} < q_0$, first we apply the  H\"older's inequality with the exponents $p_{n_j}$ and $p_{n_j}/(p_{n_j}-1)$, we obtain 
\begin{align}
\label{approximating upnj}
&\left|\iint_{Q} \Phi(x,y,t, p_{n_j}, u^{(p_{n_j})}) \chi_{C_{p_{n_j}}, k} (x,y) \frac{\phi(x) - \phi(y)}{|x-y|^{N+s}} \,dx\,dy\right|\nonumber \\
& \leq \left( \iint_{Q} \left| \frac{u^{(p_{n_{j}})}(x,t) - u^{(p_{n_j})}(y,t)}{|x-y|^{N+s}} \right|^{p_{n_{j}}}\,dx\,dy \right)^{\frac{p_{n_{j}}-1}{p_{n_{j}}}} \times \left( \iint_{C_{p_{n_{j}},k}} \left| \frac{\phi(x) - \phi(y)}{|x-y|^{N+s}} \right|^{p_{n_{j}}} \,dx\,dy \right)^{1/p_{n_{j}}}
\end{align}
Next, we estimate the second factor on the right-hand side of \eqref{approximating upnj}. Using the H\"older's inequality with exponents $q_{0}/p_{n_j}$ and $q_0/(q_{0}-p_{n_j})$ and using the fact that $(N+s)q_{0}=N+r_{0}q_{0}$, $\|\phi\|_{W_{0}^{r_{0},q_0}(\Omega)} \leq C_\phi$ together with the measure estimate $|C_{p_{n_j},k}| \leq \frac{C_3}{k^{p_{n_j}}}$, we get
\begin{align}
\label{approximating phi}
\left( \iint_{C_{p_{n_{j}},k}} \left| \frac{\phi(x) - \phi(y)}{|x-y|^{N+s}} \right|^{p_{n_{j}}} \,dx\,dy \right)^{1/p_{n_{j}}}
\leq\frac{C_{\phi} C_3^{\frac{q_0-p_{n_{j}}}{p_{n_{j}}q_0}}}{k^{\frac{q_0-p_{n_{j}}}{q_0}}}.
\end{align}
Combining \eqref{approximating phi} with \eqref{approximating upnj} and using \eqref{bound of upn in Wsp} in view of the fact that $(N+s)p_{n_j} = N+ s_{p_{n_j}} p_{n_j}$, we obtain
\begin{align*}
 &\left|\iint_{Q} \left| \frac{u^{(p_{n_j})}(x,t) - u^{(p_{n_j})}(y,t)}{|x-y|^{N+s}} \right|^{p_{n_{j}}-2} \frac{u^{(p_{n_j})}(x,t) - u^{(p_{n_j})}(y,t)}{|x-y|^{N+s}} \chi_{C_{p_{n_{j}},k}}(x,y) \frac{\phi(x) - \phi(y)}{|x-y|^{N+s}} \,dx\,dy\right|\\ 
&\qquad\qquad\qquad\qquad\qquad\qquad\qquad\qquad\qquad\qquad\qquad\qquad\qquad\qquad\qquad\qquad \leq\frac{{C_3}^{\frac{p_{n_{j}}-1}{p_{n_{j}}}}C_{\phi}C_{3}^{{\frac{q_0-p_{n_{j}}}{p_{n_{j}}q_0}}}}{k^{\frac{q_0-p_{n_{j}}}{q_0}}}.  
\end{align*}
Therefore, taking limits as $j \to \infty$ in \eqref{eta estimate} in view of \eqref{weak convergence of upn}, we get

\[
\left| \iint_{Q} \frac{\eta_k(x,y,t)(\phi(x) - \phi(y))}{|x-y|^{N+s}}\,dx\,dy - \int_{\Omega} (f(x,u) - u_{t}(x,t)) \phi(x)\,dx \right| \leq \frac{C_{\phi}C_{3}^{\frac{q_{0}-1}{q_{0}}}}{k^{\frac{q_{0}-1}{q_{0}}}}.
\]
In particular,
\[
\left|\iint_{Q}\frac{\eta_{k_{m}}(x,y,t)(\phi(x)-\phi(y))}{|x-y|^{N+s}}\,dx\,dy - \int_{\Omega} (f(x,u)-u_{t}(x,t)) \phi(x)\,dx \right| \leq \frac{C_\phi C_{3}^{\frac{q_{0}-1}{q_{0}}}}{k^{\frac{q_{0}-1}{q_{0}}}_m}.
\]
Therefore, taking the limit as $ m\to\infty $, we obtain that
\begin{equation}
    \label{existence of eta}
\iint_{Q}\frac{\eta(x,y,t)(\phi(x)-\phi(y))}{|x-y|^{N+s}}\,dx\,dy - \int_{\Omega} (f(x,u)-u_{t}(x,t)) \phi(x)\,dx = 0.
\end{equation}
Now suppose that $\phi \in W_0^{s,1}(\Omega) \cap L^2(\Omega)$ there exists $\phi_n \in C_{c}^{\infty}(\Omega)$ such that
\[
\phi_n \to \phi \quad \text{in } W_0^{s,1}(\Omega) \cap L^2(\Omega) \quad \text{as } n \to +\infty.
\]
By Fatou's Lemma and \eqref{existence of eta}, we have
\begin{align*}
&\iint_{Q}\left( \frac{|\phi(x)-\phi(y)| }{|x-y|^{N+s}}- \frac{\eta(x,y,t)(\phi(x)-\phi(y))}{|x-y|^{N+s}}\right)\,dx\,dy\\
&\qquad\qquad\leq \liminf_{n \to \infty}\iint_{Q}\left( \frac{|\phi_n(x)-\phi_n(y)|}{|x-y|^{N+s}} - \frac{\eta(x,y,t)(\phi_n(x)-\phi_n(y))}{|x-y|^{N+s}} \right)\,dx\,dy\\
&\qquad\qquad=\iint_{Q}\frac{|\phi(x)-\phi(y)|}{|x-y|^{N+s}}\,dx\,dy - \int_{\Omega} (f(x,u)-u_{t}(x,t))\phi(x)\,dx.
\end{align*}
It implies that
\[
\iint_{Q} \frac{\eta(x,y,t)(\phi(x)-\phi(y))}{|x-y|^{N+s}}\,dx\,dy \geq \int_{\Omega} (f(x,u)-u_{t}(x,t)) \phi(x)\,dx
\]
for all $\phi \in W_0^{s,1}(\Omega) \cap L^2(\Omega)$. Since the above inequality is also true for $-\phi$, we get the equality
\begin{equation}
    \label{existence of weak sol}
\iint_{Q}\frac{\eta(x,y,t)(\phi(x)-\phi(y))}{|x-y|^{N+s}}\,dx\,dy - \int_{\Omega} (f(x,u)-u_{t}(x))\phi(x)\,dx = 0
\end{equation}
for all $\phi \in W_0^{s,1}(\Omega) \cap L^2(\Omega)$.
To finish the proof, we only need to show that
\begin{equation}
    \label{eta signum}
\eta(x,y,t) \in \operatorname{sign}(u(x,t)-u(y,t)) \quad \text{a.e. } (x,y,t) \in \R^N \times \R^N\times[0,\infty).
\end{equation}
Choose $\phi(\cdot) = u^{(p_{n})}(\cdot,t)$ in \eqref{existence of up} and $\phi(\cdot) = u(\cdot,t)$  in \eqref{existence of weak sol} and integrate both with respect to $t$ from $0$ to $T$, then we have
\begin{align}
\label{eta vector}
\int_{0}^{T} & \iint_{Q}\frac{|u^{(p_{n})}(x,t)-u^{(p_{n})}(y,t)|^{p_n}}{|x-y|^{(N+s)p_n}}\,dx\,dy\,dt = \int_{0}^{T}\int_{\Omega}f(x,u^{(p_{n})})u^{(p_{n})}(\cdot,t)\,dx\,dt -\frac{1}{2}\|u^{(p_{n})}(\cdot,T)\|_{L^{2}(\Omega)}^{2}\nonumber\\
&\qquad + \frac{1}{2}\|u^{(p_{n})}(\cdot,0)\|_{L^{2}(\Omega)}^{2}\nonumber\\
&=\int_{0}^{T}\iint_{Q}\frac{\eta(x,y,t)(u(x,t)-u(y,t))}{|x-y|^{N+s}} \,dx\,dy\,dt +\int_{0}^{T}\int_{\Omega}\left(f(x,u^{(p_{n})})u^{(p_{n})}(x,t)-f(x,u)u(x,t)\right)\,dx\,dt\nonumber\\
&\qquad-\left(\frac{1}{2}\|u^{(p_{n})}(\cdot,T)\|_{L^{2}(\Omega)}^{2}-\frac{1}{2}\|u(\cdot,T)\|_{L^{2}(\Omega)}^{2} \right) + \left( \frac{1}{2}\|u^{(p_{n})}(\cdot,0)\|_{L^{2}(\Omega)}^{2}-\frac{1}{2}\|u(\cdot,0)\|_{L^{2}(\Omega)}^{2}\right).
\end{align}
Letting $n\to \infty $ in \eqref{eta vector}, Lemma \ref{Nonlinearity convergence} and the strong convergence of the initial data in $L^{2}(\Omega)$, we obtain
\begin{equation}
\label{low semi Q}
    \liminf_{n \to \infty}\int_{0}^{T}\iint_{Q}\frac{|u^{(p_{n})}(x,t)-u^{(p_{n})}(y,t)|^{p_n}}{|x-y|^{(N+s)p_n}}\,dx\,dy\,dt
\leq\int_{0}^{T}\iint_{Q}\frac{\eta(x,y,t)(u(x,t)-u(y,t))}{|x-y|^{N+s}} dx\,dy\,dt.
\end{equation}
On the other hand since $\Omega$ is a bounded subset of $\mathbb{R}^{N}$, for given $\epsilon>0$ we can find $\Omega \subset A$ with $ |A|<+\infty $ such that
\[
\int_{\R^N \setminus A} \frac{1}{|x-y|^{N+s}}\,dx \leq \frac{\epsilon}{2C_{6}} \quad \forall\, y \in \Omega,
\]
where $C_{6}>0$ such that $\|u^{(p_{n})}\|_{L^{1}([0,T];L^{1}(\Omega))}\leq C_{6}.$
Then,
\begin{equation}
\label{lower semi}
\begin{split}
\int_{0}^{T}&\iint_{Q}\frac{|u^{(p_{n})}(x,t)-u^{(p_{n})}(y,t)|}{|x-y|^{N+s}}\,dx\,dy\,dt\\
&\leq 2\int_{0}^{T}\int_{\Omega} \int_{\mathbb{R}^{N}\setminus A} \frac{|u^{(p_{n})}(x,t)-u^{(p_{n})}(y,t)|}{|x-y|^{N+s}}\,dx\,dy\,dt +\int_{0}^{T}\int_{A} \int_{A} \frac{|u^{(p_{n})}(x,t)-u^{(p_{n})}(y,t)|}{|x-y|^{N+s}}\,dx\,dy\,dt\\
&= 2\int_{0}^{T}\int_{\Omega} |u^{(p_{n})}(y,t)| \left( \int_{\R^N \setminus A} \frac{1}{|x-y|^{N+s}}\,dx \right) dy\,dt +\int_{0}^{T}\int_{A} \int_{A} \frac{|u^{(p_{n})}(x,t)-u^{(p_{n})}(y,t)|}{|x-y|^{N+s}}\,dx\,dy\,dt\\
&\leq \epsilon +\int_{0}^{T}\int_{A}\int_{A} \frac{|u^{(p_{n})}(x,t)-u^{(p_{n})}(y,t)|}{|x-y|^{N+s}} \,dx\,dy\,dt.
\end{split}
\end{equation}
Using Fatou's Lemma together with \eqref{lower semi} and \eqref{low semi Q}, we obtain
\begin{align*}
\int_{0}^{T}\iint_{Q}\frac{|u(x,t)-u(y,t)|}{|x-y|^{N+s}}\,dx\,dy\,dt
&\leq \liminf_{n \to \infty} \int_{0}^{T}\iint_{Q}\frac{|u^{(p_{n})}(x,t)-u^{(p_{n})}(y,t)|}{|x-y|^{N+s}}\,dx\,dy\,dt\\
&\leq \epsilon + \liminf_{n \to \infty}\int_{0}^{T}\int_{A} \int_{A} \frac{|u^{(p_{n})}(x,t)-u^{(p_{n})}(y,t)|}{|x-y|^{N+s}}\,dx\,dy\,dt\\
&\leq \epsilon + \liminf_{n \to \infty} \left( \int_{0}^{T}\iint_{Q}\frac{|u^{(p_{n})}(x,t)-u^{(p_{n})}(y,t)|^{p_n}}{|x-y|^{(N+s)p_n}}\,dx\,dy\,dt\right)^{1/p_n}\\
&\qquad\qquad\times\left|(A \times A)\times [0,T]\right|^{\frac{p_n-1}{p_n}}\\
&\leq \epsilon + \iint_{Q}\eta(x,y,t)\frac{u(x,t)-u(y,t)}{|x-y|^{N+s}}\,dx\,dy\,dt.
\end{align*}
Therefore,
\[
\int_{0}^{T}\iint_{Q}\frac{|u(x,t)-u(y,t)|}{|x-y|^{N+s}}\,dx\,dy\,dt \leq \epsilon + \int_{0}^{T}\iint_{Q}\eta(x,y,t)\frac{u(x,t)-u(y,t)}{|x-y|^{N+s}}\,dx\,dy\,dt,
\]
Since $\epsilon>0$ and $T>0$ is arbitrary it follows that
$$\eta(x,y,t) \in \operatorname{sign}(u(x,t)-u(y,t)) \quad \text{a.e. } (x,y,t) \in \R^N \times \R^N\times[0,\infty).$$
\textbf{Claim 3:} 
$$\int_{0}^{t} \|u_t(\cdot, \tau)\|^{2}_{L^{2}(\Omega)} \, d\tau + E(u(\cdot,t))
    \leq E(u_{0}) \quad \text{a.e. } t \in [0,\infty).$$
From \eqref{sol: up}, we have 
$$\int_{0}^{t}||u_t^{(p_{n})}(\cdot, \tau)||^{2}_{L^{2}(\Omega)} ~d\tau + E^{(p_n)}(u^{(p_n)}(\cdot,t))\leq E^{(p_n)}(u^{(n)}_{0})\qquad \text{a.e. } t \in (0, +\infty).$$
Then, using \eqref{weak convergence of upn}-\eqref{stong:conv of upn}, the weak lower semicontinuity of the norm, Fatou's Lemma, and Lemmas \ref{Nonlinearity convergence} and \ref{strong convergence u_0}, we obtain
\begin{align*}
   \int_{0}^{t} & ||u_t(\cdot, \tau)||^{2}_{L^{2}(\Omega)}d\tau +\iint_{Q}\frac{|u(x,t)-u(y,t)|}{|x-y|^{N+s}}\,dx\,dy \\
   &\leq \liminf_{n \to \infty}\int_{0}^{t}||u^{(p_n)}_t(\cdot,\tau)||^{2}_{L^{2}(\Omega)}d\tau + \liminf_{n \to \infty}\frac{1}{p_n}\iint_{Q}\frac{|u^{(p_n)}(x,t)-u^{(p_n)}(y,t)|^{p_n}}{|x-y|^{N+s_{p_n}p_n}}\,dx\,dy\\
   &\leq\liminf_{n \to \infty}\left( E^{(p_n)}(u^{(n)}(\cdot,0)) + \int_{\Omega}F(x,u^{(p_n)})\,dx \right)\\
    &=\lim_{n \to \infty}\left( E^{(p_n)}(u^{(n)}(\cdot,0)) + \int_{\Omega}F(x,u^{(p_n)})\,dx \right) =E(u_{0}) + \int_{\Omega}F(x,u)\,dx. 
\end{align*}
Thus, we have \textbf{Claim 3}. This establishes the global existence of weak solutions.
\end{proof}

 \begin{proof}[\textbf{Proof of the Theorem \ref{Main Theorem low initial energy-strong}}]
Since $C_{c}^{\infty}(\Omega)$ is dense in $W_{0}^{s,1}(\Omega)\cap L^{2}(\Omega)$, then there exists a sequence $u_{0}^{(m)}\in C_{c}^{\infty}(\Omega)\subset W_{0}^{s_{q},q}(\Omega)\cap L^{2}(\Omega)$ such that $u_{0}^{(m)}\to u_{0}$ in $W_{0}^{s,1}(\Omega)\cap L^{2}(\Omega)$. Due to the fact that $E(u_{0})<d, I(u_{0})>0$ and both $E$ and $I$ are continuous functionals on $W_{0}^{s,1}(\Omega)$. Therefore, there exists $m_{0}\in \mathbb{N}$ such that $E(u_{0}^{(m)})<d$ and $I(u_{0}^{(m)})>0$ for all $m\geq m_{0}$. For each fixed $m\geq m_{0}$, the existence of global weak solution $u^{(m)}$ for $u_{0}^{(m)}\in W_{0}^{s_{q},q}(\Omega)\cap L^{2}(\Omega)$ follows by \textbf{Claim 1, Claim 2} and \textbf{Claim 3} in the proof of Theorem \ref{Main Theorem low initial energy}. Using \eqref{Assump:f_3}, we obtain $f(\cdot, u^{(m)}) \in L^2(\Omega).$ By taking $g= f(\cdot, u^{(m)}) - u_t^{(m)}$ in  Lemma \ref{continuity of strong solution}, we obtain
       \begin{equation}
           \label{u0m strong solution}
           u^{(m)}\in C(0,\infty ; W_{0}^{s,1}(\Omega))\,\,\text{and}\,\, \int_{0}^{t}\|u_{t}^{(m)}(\cdot,\tau)\|_{L^{2}(\Omega)}^{2}+E(u^{(m)}(\cdot,t))=E(u^{(m)}(\cdot,0)),\, \text{for all } t\in [0,\infty).
       \end{equation}
Hence, $u^{(m)}$ is the strong solution of the problem \eqref{main:problem}. Next, we proceed with the proof in several steps.\\
\textbf{Step 1:} Uniform bounds of $u^{(m)}$  in $W_{0}^{s,1}(\Omega)\cap L^{2}(\Omega)$ .\\
Since $I(u_{0}^{(m)})>0$, by applying Lemma \ref{u in W delta} for $\delta=1$, we have $u^{(m)}(\cdot, t) \in W$ for all $t\geq 0$. Combining this with assumption \eqref{Assump:f_1}, we obtain
\begin{align*}
d>E(u^{(m)}(\cdot,t))&=\iint_{Q}\frac{|u^{(m)}(x,t)-u^{(m)}(y,t)|}{|x-y|^{N+s}}\,dx\,dy-\int_{\Omega}F(x,u^{(m)})\,dx\\
&\geq \iint_{Q}\frac{|u^{(m)}(x,t)-u^{(m)}(y,t)|}{|x-y|^{N+s}}\,dx\,dy-\int_{\Omega}\frac{1}{\theta}u^{(m)}f(x,u^{(m)})\,dx\\
&\geq \left(1-\frac{1}{\theta}\right)\|u^{(m)}(\cdot,t)\|_{W_{0}^{s,1}(\Omega)}+\frac{1}{\theta}I(u^{(m)}(\cdot,t)) \geq \left(1-\frac{1}{\theta}\right)\|u^{(m)}(\cdot,t)\|_{W_{0}^{s,1}(\Omega)}.
\end{align*}
Since $u^{(m)}$ is a strong solution and $u^{(m)}\in W$ for all $t\in[0,\infty)$. Choosing $\phi=u^{(m)}(\cdot,t)$ in \eqref{definition of weak sol u1}, we obtain 
$$\frac{1}{2}\frac{d}{dt}\|u^{(m)}(\cdot,t)\|_{L^{2}(\Omega)}^{2}=-I(u^{(m)}(\cdot,t))<0,\quad\forall\, t\in [0,\infty).$$
This implies that 
\begin{align}
    \label{um L2 uniform bound}
    \|u^{(m)}(\cdot,t)\|_{L^{2}(\Omega)}&\leq \|u^{(m)}(\cdot,0)\|_{L^{2}(\Omega)}= \|u(\cdot,0)\|_{L^{2}(\Omega)}+o(1)\quad\forall\, t\in [0,\infty).
\end{align}
Considering \eqref{u0m strong solution}, we obtain
\begin{equation}
    \label{uniform bound um}
    \int_{0}^{t}\|u_{t}^{(m)}(\cdot,\tau)\|_{L^{2}(\Omega)}^{2}<d \quad\text{and }\quad \|u^{(m)}(\cdot,t)\|_{W_{0}^{s,1}(\Omega)}<\frac{d\theta}{\theta-1}.
\end{equation}
Following the same argument as in \eqref{f(x,u^{p_n}) uniform bound} and from \eqref{uniform bound um}, we obatin
\begin{align}
    \int_{\Omega}|f(x,u^{(m)}(\cdot,t))|^{\frac{r}{r-1}}\,dx&\leq C^{\frac{r}{r-1}}\int_{\Omega}|u^{(m)}(\cdot,t)|^{r}\,dx \leq C^{\frac{r}{r-1}} (C_{Sob})^{r}\|u^{(m)}(\cdot,t)\|^{r}_{{W_{0}^{s,1}(\Omega)}} \leq  \left(\frac{d\theta C_{Sob}C^{\frac{1}{r-1}}}{\theta-1}\right)^{r}.
    \end{align}
    Combining all the uniform estimates on $u^{(m)}$ implies the existence of subsequence of $\{u^{(m)}\}_{n\in\mathbb{N}}$ (still denoted by $\{u^{(m)}\}_{m\in\mathbb{N}}$) and $u \in L^{\infty}(0,+\infty;W_{0}^{s,1}(\Omega)\cap L^{2}(\Omega))$ such that, as $m\to\infty$, we have 
\begin{equation}
\label{weak convergence of um}
 \left.
 \begin{aligned}
 u^{(m)}&\stackrel{\ast}{\rightharpoonup} u\qquad\text{ in }L^{\infty}(0,+\infty;L^{2}(\Omega));\\
 u_{t}^{(m)}&\weak u_{t} \qquad\text{ in }L^{2}(0,+\infty;L^{2}(\Omega));\\
 f(x,u^{(m)})&\stackrel{\ast}{\rightharpoonup}\xi\qquad\text{ in }L^{\infty}(0,+\infty;L^{\frac{r}{r-1}}(\Omega)).
 \end{aligned}
 \right\}
\end{equation}
Then, again from Aubin-Lions compactness theorem\cite[Corollary 4, p. 85]{simon} for any $T>0$, we get
 \begin{equation}
 \label{stong:conv of um}
 \begin{aligned}
 u^{(m)}&\to u \text{ in } C([0,T];L^{q}(\Omega)),\text{  as } n \to \infty, \text{ for all } q\in[1,1^{*}_{s}),
 \end{aligned}
  \end{equation}
  and so $\xi=f(x,u).$ This implies that $u^{(m)}(x,t)\to u(x,t)$ for a.e. 
  $(x,t)\in \Omega\times(0,+\infty).$\\
   \textbf{Step 2:} We show that there exists $\eta(\cdot,\cdot,t)\in L^{\infty}(\mathbb{R}^{N}\times\mathbb{R}^{N}),\, \eta(x,y,t)=-\eta(y,x,t)$ for almost all $(x,y)\in \mathbb{R}^{N}\times\mathbb{R}^{N},\,\|\eta(\cdot,\cdot,t)\|_{L^{\infty}(\mathbb{R}^{N}\times\mathbb{R}^{N})}\leq 1$, such that following equality holds:
\begin{align}
\label{strong solution u0 in W1}
   \int_{\Omega}u_{t}\phi \,dx+\iint_{Q}\eta(x,y,t)\frac{\phi(x)-\phi(y)}{|x-y|^{N+s}}\,dx\,dy=\int_{\Omega}f(x,u)\phi \,dx
\end{align}
for all
$\phi\in W_{0}^{s,1}(\Omega)\cap L^{2}(\Omega)$ and a.e. $t\in[0,+\infty)$.\\
By Definition \ref{definition of weak sol u1}, corresponding to each $u^{(m)}$ there exists $\eta^{(m)}(\cdot,\cdot,t)\in L^{\infty}(\mathbb{R}^{N}\times\mathbb{R}^{N}),\, \eta^{(m)}(x,y,t)=-\eta^{(m)}(y,x,t)$ for almost all $(x,y)\in \mathbb{R}^{N}\times\mathbb{R}^{N},\,\|\eta^{(m)}(\cdot,\cdot,t)\|_{L^{\infty}(\mathbb{R}^{N}\times\mathbb{R}^{N})}\leq 1$, such that following equality holds
\begin{align}
\label{definition weak um}
   \int_{\Omega}u_{t}^{(m)}\phi \,dx+\iint_{Q}\eta^{(m)}(x,y,t)\frac{\phi(x)-\phi(y)}{|x-y|^{N+s}}\,dx\,dy=\int_{\Omega}f(x,u^{(m)})\phi\,dx
\end{align}
for all
$\phi\in W_{0}^{s,1}(\Omega)\cap L^{2}(\Omega)$ and a.e. $t\in[0,+\infty)$. Since $\|\eta^{(m)}(\cdot,\cdot,t)\|_{L^{\infty}(\mathbb{R}^{N}\times\mathbb{R}^{N})}\leq 1$, there exist a subsequence of $\{\eta^{(m)}\}_{m\in\mathbb{N}}$ still denoted by $\{\eta^{(m)}\}_{m\in\mathbb{N}}$  and a function $\eta(\cdot,\cdot,t)\in L^{\infty}(\mathbb{R}^{N}\times\mathbb{R}^{N})$ such that 
$$\eta^{(m)}(\cdot,\cdot,t)\stackrel{\ast}{\rightharpoonup}\eta(\cdot,\cdot,t) \quad \text{in}\quad L^{\infty}(\mathbb{R}^{N}\times\mathbb{R}^{N})\quad\text{and}\quad \|\eta(\cdot,\cdot,t)\|_{L^{\infty}(\mathbb{R}^{N}\times\mathbb{R}^{N})}\leq 1. $$
Moreover, since each $\eta^{(m)}$ is antisymmetric, $\eta$ is also antisymmetric.
 To pass the limit in \eqref{definition weak um}, by using density of $\phi\in C_{c}^{\infty}(\Omega)$ in $ W_{0}^{s,1}(\Omega)\cap L^{2}(\Omega)$ and following the same argument as in \textbf{Claim 2}, as $m\to\infty$, we obtain that
 $$ \int_{\Omega}u_{t}\phi \,dx+\iint_{Q}\eta(x,y,t)\frac{\phi(x)-\phi(y)}{|x-y|^{N+s}}\,dx\,dy=\int_{\Omega}f(x,u)\phi \,dx,$$
 for all $\phi\in W_{0}^{s,1}(\Omega)\cap L^{2}(\Omega)$ and
$\eta(x,y,t)\in \operatorname{sign}(u(x,t)-u(y,t)) \quad a.e. \,(x,y,t)\in \mathbb{R}^{N}\times\mathbb{R}^{N}\times [0,\infty).$\\
\textbf{Step 3:} We show that
$$\int_{0}^{t} \|u_t(\cdot, \tau)\|^{2}_{L^{2}(\Omega)} \, d\tau + E(u(\cdot,t))
    = E(u_{0}) \quad \text{for all } t \in [0,\infty).$$
Following the same argument of \textbf{Claim 3}, we obtain
$$\int_{0}^{t} \|u_t(\cdot, \tau)\|^{2}_{L^{2}(\Omega)} \, d\tau + E(u(\cdot,t))
    \leq E(u_{0}) \quad \text{a.e. } t \in [0,\infty).$$
 It implies that $u$ is a global weak solution of \eqref{main:problem} with initial data $u_{0}\in W_{0}^{s,1}(\Omega)\cap L^{2}(\Omega)$. Using \eqref{Assump:f_3} we obtain $f(\cdot,u)\in L^{2}(\Omega)$. By taking $g=f(\cdot,u)-u_{t}$ in  Lemma \ref{continuity of strong solution}, we obtain
       \begin{equation}
           \label{u strong solution final}
           u\in C(0,\infty ; W_{0}^{s,1}(\Omega))\quad \text{and}\quad \int_{0}^{t}\|u_{t}(\cdot,\tau)\|_{L^{2}(\Omega)}^{2}+E(u(\cdot,t))=E(u(\cdot,0)) \quad \text{for all}\,\,t\in [0,\infty).
       \end{equation}
    \end{proof}
    \subsection{On the case of critical initial energy: \texorpdfstring{$E(u_{0})=d$}{E(u0)=d}}
In this subsection, we address the case of critical initial energy \( E(u_{0}) = d \). Specifically, we will prove that if \( I(u_{0}) \geq 0 \), then problem \eqref{main:problem} admits a global weak solution and strong solution.
\begin{proof}[\textbf{Proof of the Theorem \ref{main-exist-critical-weak}}]
    Let $\lambda_k = 1 - \frac{1}{k}$, for $k \in \mathbb{N}$ and \( u_{0} \in W_{0}^{s_{q},q}(\Omega)\cap L^{2}(\Omega).\) Consider the following initial value problem:   
\begin{equation} \label{modified:problem}
    \left\{
    \begin{aligned}
      u_t + (-\Delta)^{s}_{1} u &= f(x,u) && \text{in } \Omega \times (0, \infty), \\
      u &= 0 && \text{in } \mathbb{R}^N \setminus \Omega \times (0, \infty), \\
      u(x,0) &= \lambda_k u_0(x) := u_{0k} && \text{in } \Omega.
    \end{aligned}
    \right.
\end{equation}
Let $I(u_0) \geq 0$. By  Lemma \ref{Positive:depth} and Lemma \ref{Lemma:2.3}(iii), we deduce that $u_0 \not \equiv 0$ and there exists a unique $\lambda_{\ast} = \lambda_{\ast}(u_0) \geq 1$ such that $I(\lambda_{\ast} u_0) = 0$. Since $\lambda_k < 1 \leq \lambda_{\ast}$, and using Lemma \ref{Lemma:2.3}(ii)-(iii), we obtain the following:
$$ I(u_{0k}) = I(\lambda_k u_0) > 0 \quad \text{and} \quad E(u_{0k}) = E(\lambda_k u_0) < E(u_0) = d. $$
Therefore, by Theorem \ref{Main Theorem low initial energy}, for each $k$, the problem \eqref{modified:problem} has a global weak solution $u^{(k)} \in L^\infty(0, \infty; W_{0}^{s,1}(\Omega)\cap L^{2}(\Omega))$ with $u_t^{(k)} \in L^2(0, \infty; L^2(\Omega))$ satisfying
$$ \int_0^t \| u^{(k)}_t(\cdot, \tau) \|_{L^2(\Omega)}^2 \,d\tau + E(u^{(k)}(\cdot, t)) \leq  E(u_{0k}) < d \quad \text{a.e }\, t\in [0,\infty).$$
Following the same argument as in \textbf{Claim 1} and \textbf{Claim 2} of Theorem \ref{Main Theorem low initial energy}, there exists a subsequence (denoted with the same notation) $\{u^{(k)}\}_{k \in \mathbb{N}}$ that converges weakly to a function $u$  as in \eqref{weak convergence of um}. Moreover, $u\in L^\infty(0, \infty; W_{0}^{s,1}(\Omega)\cap L^{2}(\Omega))$ and $u$ is a weak solution of \eqref{main:problem}. If we assume that $I(u(\cdot,t))>0$ for $0<t<t^{*}$ and $I(u(\cdot,t^{*}))=0$. Then by choosing $\phi= u(\cdot,t)$ in \eqref{definition of weak sol u1}, we obtain
\begin{equation}
    \label{L2 norm decreasing}
     \frac{1}{2}\frac{d}{dt}\|u(\cdot,t)\|^{2}_{L^{2}(\Omega)}=-I(u(\cdot,t)).
\end{equation}
This implies that $u_{t}(\cdot, t)\not \equiv 0$ for $0<t<t^{*}$.
Therefore, by \eqref{sol: u} we have
$$E(u(\cdot,t^{*})) \leq \  d-\int_{0}^{t^{*}}\|u_t(\cdot,\tau)\|^{2}_{L^{2}(\Omega)}\,d\tau < d.$$
By the definition of $d$, we get $\|u(\cdot,t^{*})\|_{W_{0}^{s,1}(\Omega)}=0$. Now, by extending the function $u$ as $u(\cdot,t)\equiv 0$ for $t\geq t^{*}$, we obtain a weak solution vanishing in finite time $t^{*}$. The existence of global strong solution can be established by following the same arguments as in case of weak solution and by using the Theorem \ref{Main Theorem low initial energy-strong}. 
    \end{proof}
We conclude this section by showing that uniqueness does hold under the additional assumption that $f(x,\cdot)$ is
uniformly Lipschitz continuous.
    \begin{proof}[\textbf{Proof of the Corollary \ref{Uniqueness of weak/strong solution}}]
     Assume $u$ and $v$ are weak/strong solutions of \eqref{main:problem}. Then, by Definition~\ref{Def:weak solution}, for any test function $\phi \in W_{0}^{s,1}(\Omega)\cap L^{2}(\Omega)$, we have
     \begin{align*}
         &\int_{\Omega}u_{t}\phi \,dx+\iint_{Q}\eta_{u}(x,y,t)\frac{\phi(x)-\phi(y)}{|x-y|^{N+s}}\,dx\,dy =\int_{\Omega}f(x,u)\phi \,dx,\\
 &\int_{\Omega}v_{t}\phi \,dx+\iint_{Q}\eta_{v}(x,y,t)\frac{\phi(x)-\phi(y)}{|x-y|^{N+s}}\,dx\,dy =\int_{\Omega}f(x,v)\phi \,dx.
     \end{align*}
    $$$$
Subtracting the two equalities above, setting \( \phi = u - v \in W_{0}^{s,1}(\Omega)\cap L^{2}(\Omega)\), and integrating with respect to \( t \) over the interval \([0, t]\), we get
\begin{align*}
&\int_{0}^{t}\int_{\Omega}(u_{t}-v_{t})(u-v)\,dx\,dt
+\int_{0}^{t}\iint_{Q}\left(\eta_{u}(x,y,t)-\eta_{v}(x,y,t)\right)\left(\frac{u(x,t)-v(x,t)-u(y,t)+v(y,t)}{|x-y|^{N+s}}\right)\,dx\,dy\, dt\\
&\qquad\qquad\qquad\qquad\qquad\qquad\qquad\qquad\qquad\qquad\qquad\qquad\qquad=\int_{0}^{t}\int_{\Omega}\left(f(x,u)-f(x,v)\right)(u-v)\,dx\,dt.
\end{align*}
Since $\eta_{u}(x,y,t)$ and $\eta_{v}(x,y,t)$ are sign functions, the second term on the left-hand side of the equality is always non-negative. So, we have
\begin{align*}
\int_{0}^{t}\int_{\Omega}(u_{t}-v_{t})(u-v)\,dx\,dt \leq \int_{0}^{t}\int_{\Omega}\left(f(x,u)-f(x,v)\right)(u-v)\,dx\,dt.
\end{align*}
Thus, since $f(x,\cdot)$ is uniformly Lipschitz for almost all $x \in \Omega$, using this fact, we obtain 
\begin{align*}
\int_{\Omega}\int_{0}^{t} (u_{t} - v_{t}) (u-v) dt\, dx \leq M\int_{0}^{t}\int_{\Omega}\left(u-v\right)^{2}\,dx\,dt,
\end{align*}
 where $M$ is a Lipschitz constant. Thus, by Gronwall's inequality and the fact that $u(x,0)=v(x,0)$, we deduce that $u=v$ a.e. in $\Omega.$
    \end{proof}
\section{Wellposedness and Dynamics in low dimension}
\label{Wellposedness and Dynamics in low dimension}
\subsection{Local existence of strong solution}
In this section, we examine the local existence of strong solution of \eqref{main:problem} in low dimension. In view of Lemma \ref{operator and subdifferential}, the system \eqref{main:problem} can be rewritten as the following abstract Cauchy problem:
\begin{equation}
\label{subdifferenial main prob}
 \frac{du}{dt} + \partial\varphi(u)-\partial\psi(u) \ni 0  \quad \text{in } H,\, 0<t<T, \quad \text{and} \quad u(\cdot,0) = u_{0}(\cdot) \quad  \text{in } \Omega.
\end{equation}
\begin{lemma}
\label{D(phi,r) is compact}
    Let the conditions \eqref{Assump:f_0}--\eqref{Assump:f_1} hold and $N<2s$. Then, the set \( D(\varphi,a) \) is compact in \( L^{2}(\Omega) \) for any \( a\in\mathbb{R} \). Moreover, \( D(\varphi) \subset D(\psi) \).
\end{lemma}

\begin{proof}
    From assumptions that $N<2s$, it follows that \( W_{0}^{s,1}(\Omega) \) is compactly embedded in \( L^{2}(\Omega) \). From Proposition \ref{prop:embedding}(i) we have \( W_{0}^{s,1}(\Omega) \) is continuously embedded in \( L^{r}(\Omega) \) for all $r\in [1,1^{*}_{s}]$. Hence \( D(\varphi) \subset D(\psi) \) and the set \( D(\varphi, a) \) is compact in \( L^{2}(\Omega) \) for any \( a\in\mathbb{R} \).
\end{proof}
\begin{lemma}
\label{del phi is bounded}
   Let the conditions \eqref{Assump:f_0}--\eqref{Assump:f_3}  hold. Then, $
   \left\{(\partial\psi)^{0}(u) \mid u\in D(\varphi, a)\right\} = \{f(\cdot,u)\}$ for any \( a\in \mathbb{R} \) and \( f(\cdot, u) \in L^{2}(\Omega) \).
\end{lemma}

\begin{proof}
    It is enough to consider \( a>0 \). Since \( \psi \in C^{1}(L^{r}(\Omega), \mathbb{R}) \), there exists a unique \( f_{u} \in L^{2}(\Omega) \) such that \( \partial\psi(u) = \{ f_{u} \} \), and so \( (\partial\psi)^{0}(u) = \{ f_{u} \} \). Moreover,
    \[
    \int_{\Omega} f(x,u) v \,dx = (f_{u},v)_{L^{2}(\Omega)} \quad \text{for all } v\in C_{0}^{\infty}(\Omega).
    \]
    The above equality implies \( f(x,u) = f_{u} \) a.e. in \( \Omega \). Now, using the condition \eqref{Assump:f_0} and applying H\"older's inequality, we obtain
    \begin{align}
    \label{bound of least norm sub}
         \left| (f_{u},v)_{L^{2}(\Omega)}\right| &\leq C\int_{\Omega} |u|^{r-1} |v(x)|\,dx \leq C \||u|^{r-1}\|_{L^{2}(\Omega)} \|v\|_{L^{2}(\Omega)}.
    \end{align}
    From \eqref{Assump:f_3} and Proposition \ref{prop:embedding}(i), we have 
    \begin{align}
    \label{bound of least norm sub-1}
        \int_{\Omega} |u|^{2(r-1)}\,dx \leq |\Omega| + \int_{\Omega \cap \{|u| > 1\}} |u|^{1^{\ast}_{s}}\,dx \leq |\Omega| + \left(C_{Sob}\right)^{1^{\ast}_{s}} \|u\|_{W_{0}^{s,1}(\Omega)}^{1^{\ast}_{s}}.
    \end{align}
    Since \( u\in D(\varphi,a) \), there exists a constant \( C_{a}>0\) such that \( \|u\|_{W_{0}^{s,1}(\Omega)}^{1^{\ast}_{s}} \leq C_a \). Now, combining \eqref{bound of least norm sub} and \eqref{bound of least norm sub-1}, we obtain the required claim.
\end{proof}

\begin{proof}[Proof of the Theorem \ref{loc}]
    The existence of a strong solution \( u \) directly follows from Theorem \ref{wellposedness} combined with Lemmas \ref{D(phi,r) is compact} and \ref{del phi is bounded}. Since \( u \) satisfies \eqref{v in subdifferential}, it follows that  $g = f(x,u) - u_t \in \partial \varphi(u).$
Therefore, by setting \( g = f(\cdot,u) - u_t \) in Lemma \ref{continuity of strong solution}, we obtain $u \in C([0,T];W_{0}^{s,1}(\Omega)).$
Moreover, the energy identity holds:
\begin{align*}
    \int_{0}^{t} \|u_{t}(\cdot,\tau)\|_{L^{2}(\Omega)}^{2} + E(u(\cdot,t)) = E(u(\cdot,0)) \qquad \text{for all }\, t\in [0,T].
\end{align*}
\end{proof}
\subsection{Asymptotic behavior and finite time blow-up}
Next, we are concerned with the asymptotic behavior and finite time blow-up of the strong solution for low initial energy $E(u_{0})<d$.
    \begin{proof}[\textbf{Proof of the Theorem \ref{low energy asymptotic behaviour}}]
        The existence of global strong solution follows by Theorem \ref{Main Theorem low initial energy}. Now, we derive the upper estimates. By applying Lemma \ref{u in W delta}, we conclude that \(u(\cdot,t) \in W_{\delta}\) for \(0 < t < \infty\) and \(\delta_{1} < \delta < \delta_{2}\), where \(\delta_1 < 1 < \delta_2\). The parameters \(\delta_1\) and \(\delta_2\) are determined as the solutions of the equation \(d(\delta) = E(u_0)\) in view of Lemma \ref{prop:d(delta)}.
Furthermore, \(u(\cdot,t) \in W_{\delta}\) implies that \(I_{\delta'}(u) > 0\) for all \(0 < t < \infty\) and \(\delta_{1} < \delta' < 1 \). Now, by taking $\phi =u(\cdot, t)$ for $0< t< \infty$ in \eqref{definition of weak sol u1}, we obtain
\begin{equation}
\label{u_t is nonzero}
    \begin{aligned}
        \frac{1}{2}\frac{d}{dt}\|u(\cdot,t)\|^{2}_{L^{2}(\Omega)}&=-I(u(\cdot,t)).
    \end{aligned}
\end{equation}
Therefore, for \(\delta_{1} < \delta' < 1\), we obtain
\begin{align*}
    \frac{1}{2}\frac{d}{dt}\|u(\cdot,t)\|^{2}_{L^{2}(\Omega)}&=-I(u(\cdot,t)) =(\delta'-1){\iint_{Q}}\frac{\abs{u(x,t)-u(y,t)}}{\abs{x-y}^{N+s}}\,dx\,dy-I_{\delta'}(u(\cdot,t))
    \leq (\delta'-1)\|u(\cdot,t)\|_{W_{0}^{s,1}(\Omega)}.
\end{align*}
Since $N\leq 2s$, by applying Proposition \ref{prop:embedding}(i), we  obtain 
\begin{equation*}
    \frac{1}{2}\frac{d}{dt}\|u(\cdot,t)\|^{2}_{L^{2}(\Omega)} \leq (\delta'-1)C_{Sob}\|u(\cdot,t)\|_{L^{2}(\Omega)}
\end{equation*}
After solving this differential inequality, we obtain
$$\|u(\cdot,t)\|_{L^{2}(\Omega)}\leq \left(\|u_{0}\|_{L^{2}(\Omega)}+(\delta'-1)C_{Sob}t\right)_{+} \quad \text{for all} \   t\in[0,\infty)$$
which further implies that the solution vanishes at a time $t^{*}=\frac{\|u_{0}\|_{L^{2}(\Omega)}}{(1-\delta')C_{Sob}}.$\\
Now, if $E(u_{0})=d$ and $I(u(\cdot, t))= 0$ for all $0\leq t<\infty$, then from \eqref{u_t is nonzero}, we obtain 
$$\frac{d}{dt}\|u(\cdot,t)\|^{2}_{L^{2}(\Omega)}=0\quad\text{for all }t\in[0,\infty).$$
Hence, $\|u(\cdot,t)\|^{2}_{L^{2}(\Omega)}$ is conserved in time. Therefore,
$$\|u(\cdot,t)\|_{L^{2}(\Omega)}=\|u_{0}\|_{L^{2}(\Omega)}\quad\text{for all }t\in[0,\infty).$$
If there exists a $t_{0}>0$ such that $I(u(\cdot,t_{0}))>0$, then by the continuity of the map $t\mapsto I(u(\cdot,t))$, there exists $\epsilon>0$ such that $I(u(\cdot,t))>0$ for all $t\in (t_{0}-\epsilon, t_{0}+\epsilon)$. From \eqref{u_t is nonzero}, we obtain $u_{t}(\cdot, t) \not \equiv 0$ for all $t\in (t_{0}-\epsilon,t_{0}+\epsilon)$. Therefore, by \eqref{strong solution}, for $t_0>0$ we have
 $$E(u(\cdot, t_{0})) = \ d-\int_{0}^{t_0} \|u_t (\cdot, \tau)\|^{2}_{L^{2}(\Omega)}\, d\tau<d.$$
Now, by taking $t=t_{0}$ as the initial time from Lemma \ref{u in W delta}, we know that $u(\cdot,t)\in W_{\delta}$ for all $\delta_{1}<\delta<\delta_{2}$ and $t_{0}<t<\infty$ under the condition $E(u(\cdot, t_{0}))< d$ and $I(u(\cdot,t_{0}))> 0$, where $\delta_{1}<\delta<\delta_{2}$ are two roots of $d(\delta)=E(u(\cdot,t_{0}))$. Thus,  $I_{\delta^{'}}(u(\cdot,t)) > 0$ for $\delta^{'}\in (\delta_{1},1)$ and $t_{0}\leq t<\infty$. Finally, by repeating the same arguments as in the case of $E(u_{0})<d$, we obtain
\begin{equation*}
    \frac{1}{2}\frac{d}{dt}\|u(\cdot,t)\|^{2}_{L^{2}(\Omega)} \leq (\delta'-1)C_{Sob}\|u(\cdot,t)\|_{L^{2}(\Omega)}\quad \text{for all } t\in[t_{0},\infty).
\end{equation*}
Integrating it with respect to $t$ from $t_{0}$ to $t$, we get
$$\|u(\cdot,t)\|_{L^{2}(\Omega)}\leq \left(\|u(\cdot,t_{0})\|_{L^{2}(\Omega)}+(\delta'-1)C_{Sob}(t-t_{0})\right)_{+} \quad \text{for all} \   t\in[t_{0},\infty)$$
which further implies that the solution vanishes at a time $t^{**}=t_{0}+\frac{\|u(\cdot,t_{0})\|_{L^{2}(\Omega)}}{(1-\delta')C_{Sob}}.$

    \end{proof}  
\begin{proof}[\textbf{Proof of the Theorem \ref{Blow-up thm low initial energy}}] Local existence of strong solution follows from Theorem \ref{loc}. Arguing by contradiction, we assume that the solution is global in time, {\it i.e.} $T_{max} = +\infty.$ For $T>0$, we define the auxiliary function $M:[0, T] \rightarrow (0,\infty)$ as
\begin{equation}
\label{def of M}
\begin{aligned}
  M(t)&:= \int_{0}^{t}\|u(\cdot, \tau)\|_{L^2(\Omega)}^{2}\,d\tau+(T-t)\|u_{0}\|_{L^2(\Omega)}^{2}+b(t+a)^{2} \qquad \text{for all } t \in [0, T],
  \end{aligned}
\end{equation}
where $a$ and $b$ are positive constants satisfying appropriate conditions which will be stated later. By differentiating $M$ with respect to $t$ and choosing $\phi(\cdot)=u(\cdot,t)$ in \eqref{definition of weak sol u1}, we obtain
    \begin{equation}
    \label{DM(t):u}
    \begin{aligned}
     & M'(t)=\|u(\cdot, t)\|^{2}_{L^{2}(\Omega)}-\|u_{0}\|^{2}_{L^{2}(\Omega)}+2b(t+a) \ \ \text{and} \ \ M''(t)=-2I(u(\cdot, t))+2b \quad\text{for all } t \in [0, T].
    \end{aligned}
\end{equation}
Now, by using the fact that $u_{0}\in V$ and applying Lemma \ref{u in W delta}(ii) for $\delta=1$, we get 
    \begin{equation}
        \label{u in V}
I(u(\cdot,t))<0\,\text{ for all }\,t\in [0,\infty) \ \text{and the maps} \ t \longmapsto M(t),  t \longmapsto M'(t) \ \text{are strictly increasing in} \ [0, T].
     \end{equation}
Moreover, by the definition of $d$ and Lemma \ref{Lemma:2.3}(iii), there exists a $\lambda_{\ast}\in (0,1)$ such that 
     \begin{equation}\label{lambda-star-est}
    I(\lambda_{\ast}u(\cdot,t))=0\quad\text{and} \quad d\leq E(\lambda_{\ast}u(\cdot,t)) \quad\text{for all } t\in [0,\infty).
     \end{equation}
Now, for a fixed $t \in [0, T]$ and $\beta \in \big[\frac{1}{\Theta}, \frac{1}{2}\big)$, where $\Theta$ is defined in \eqref{Assump:f_2} and $\Theta>2$. We define a function $g:[\lambda_{\ast},1]\rightarrow (0,\infty)$ as
    \[g(\lambda):= E(\lambda u(\cdot,t))-\beta I(\lambda u(\cdot,t)).\]
By differentiating the function $g$ with respect to $\lambda$ and using \eqref{existence of lambda d lambda}, definition of $I$ and \eqref{Assump:f_2}, we obtain
\begin{align*}
    \frac{dg(\lambda)}{d\lambda} &= \frac{d}{d\lambda}E(\lambda u(\cdot,t))-\frac{d}{d\lambda}\left(\beta I(\lambda u(\cdot,t)) \right) =\frac{I(\lambda u(\cdot,t))}{\lambda}-\beta \left(\|u\|_{W_{0}^{s,1}(\Omega)}-\frac{d}{d\lambda}\int_{\Omega}f(x,\lambda u)\lambda u \,dx\right)\\
    &= \frac{I(\lambda u(\cdot,t))}{\lambda}-\beta\left(\|u\|_{W_{0}^{s,1}(\Omega)}-\int_{\Omega}f'(x,\lambda u)\lambda u^{2}\,dx-\int_{\Omega}f(x,\lambda u)u\,dx\right)\\
    &=\left(1-\beta\right)\|u\|_{W_{0}^{s,1}(\Omega)}+ \frac{\beta}{\lambda}\left(\int_{\Omega} f'(x,\lambda u)(\lambda u)^{2}\,dx-\left(\frac{1}{\beta}-1\right)\int_{\Omega}f(x,\lambda u)\lambda u\,dx\right)\\
    &\geq \left(1-\beta\right)\|u\|_{W_{0}^{s,1}(\Omega)} \geq 0.
\end{align*}
This implies that $g$ is an increasing function in $[\lambda_{\ast},1]$. Moreover, in view of \eqref{lambda-star-est}, we have
     \begin{align}
     \label{estimate for I}
         I(u(\cdot,t))&\leq \frac{1}{\beta}\left(E(u(\cdot,t))-d\right) \quad\text{for all } t \in [0, T].
     \end{align}
By repeating the same argument as in \cite[Theorem 4.4]{Arora-2025}, we obtain 

\begin{equation}
    \label{final estimate}
    \begin{aligned}
       M(t)M''(t)-\frac{1}{2\beta}\left(M'(t)\right)^{2}\geq M(t)\left[\frac{2}{\beta}\left( d-E(u_{0})\right)-2b\left(\frac{1}{\beta}-1\right)\right] \geq 0 \qquad \text{for all } t \in [0, T],   
    \end{aligned}
\end{equation}
where 
\begin{equation}\label{cond:b}
   \frac{1}{\Theta} \leq \beta <1 \quad \text{and} \quad b \leq \frac{d-E(u_{0})}{1-\beta}.
\end{equation}
Next, by following the concavity method introduced by Levine \cite[Theorem I]{Levine-1973}, we observe that \eqref{u in V}, \eqref{final estimate} and $M'(0) = 2ab>0$ it follows that
\begin{equation}
    \label{Tmax}
    \begin{aligned}
        &M(t) \to +\infty \ \text{as} \ t \to  \frac{M(0)}{\gamma M'(0)} = \frac{T \|u_0\|_{L^2(\Omega)}^2 + ba^2}{2ab\l(\frac{1}{2\beta}-1\r)} \leq T
    \end{aligned}
\end{equation}
where the last inequality follows by choosing $a$ and $T$ large enough such that 
\begin{equation}
    \label{T-ast}
    \begin{aligned}
        \frac{\|u_0\|_{L^2(\Omega)}^2}{\l(\frac{1}{2\beta}-1\r)2a} < b \leq \frac{d-E(u_0)}{1-\beta} \qquad \text{and} \qquad T_a(b, \beta):=\frac{a^{2}b}{2ab\l(\frac{1}{2\beta}-1\r)-\|u_0\|_{L^2(\Omega)}^2}\leq T.
    \end{aligned}
\end{equation}
Moreover, by taking $T= T_a(b, \beta)$ in \eqref{Tmax}, we have
\[
M(t) \to +\infty \ \text{as} \ t \to  T_a(b, \beta):= \frac{a^{2}b}{2ab\l(\frac{1}{2\beta}-1\r)-\|u_0\|_{L^2(\Omega)}^2}.
\]
This is a contradiction of $u$ being a global strong solution. Hence, $T_{\max} < +\infty.$ 

Finally, by minimizing the blow-up time $T_a(b, \beta)$ with respect to the parameters $a, b$ and $\beta$ such that $(b, \beta) \in \mathcal{R}_a:= \bigg(\frac{\|u_0\|_{L^2(\Omega)}^2}{\l(\frac{1}{2\beta}-1\r)2a}, \frac{d-E(u_0)}{1-\beta} \bigg] \times \bigg[\frac{1}{\Theta}, \frac{1}{2}\bigg)$,
the least blow time $T^\ast$ independent of parameters $a, b$ and $\beta$ is given by
\[T^{\ast}:=\frac{4\|u_0\|_{L^2(\Omega)}^2(\Theta-1)}{\Theta(\Theta-2)^{2}(d-E(u_{0}))} \quad \text{such that} \quad \lim_{t \to T^{*}} \int_{0}^{t} \|u(\cdot, \tau)\|^{2}_{L^{2}(\Omega)} \, d\tau = +\infty.\]
Next, we will prove blow-up of strong solution $u$ for $E(u_{0})=d$ and $I(u_{0})< 0$. Applying Lemma \ref{Positive:depth} and Lemma \ref{u in W delta}, we obtain $E(u_{0})=d>0$, and $E(u(\cdot,t))$ and $I(u(\cdot,t))$ are continuous with respect to $t$. Then, there exists a $t_{0}$ such that $E(u(\cdot,t))>0$ and $I(u(\cdot,t))< 0$ for $0<t \leq t_{0}.$ Using \eqref{L2 norm decreasing}, we have $u_{t}(\cdot,t)\not\equiv 0$  for $0<t \leq t_{0}$. From \eqref{strong solution}
   $$ 0< E(u(\cdot,t_{0}))= d-\int_{0}^{t_{0}}\|u_t(\cdot,\tau)\|^{2}_{L^{2}(\Omega)}d\tau < d. $$ 
 Taking $t=t_{0}$ as the initial time we have $E(u(\cdot,t_{0}))<d$ and $I(u(\cdot,t_{0}))<0$ {\it i.e.} $u(\cdot,t_{0})\in V$. By Lemma \ref{u in W delta}(ii) we have $E(u(\cdot,t))<d$ and  $I(u(\cdot,t))<0$ for all $t\geq t_{0}$. The rest of the proof is the same as for $E(u_{0})<d$.
\end{proof}   
\section{On the case of high initial energy: \texorpdfstring{$E(u_{0})> d$}{E(u0)> d}}
\label{On the case of high initial energy}
   This section provides sufficient conditions for the global existence of strong solutions and blow-up in finite time, particularly for high initial energy. Throughout this section, we assume that $N=1$ and $s\in(\frac{1}{2},1)$, so that $N<2s$.
\begin{lemma}
    \label{N+ is bounded}
    Let $f$ satisfies the conditions \eqref{Assump:f_0}-\eqref{Assump:f_1}. Then,
 \begin{enumerate}
 \item[\textnormal{(i)}] $0$ is away from both $\mathcal{N}$ and $\mathcal{N_{-}}$, {\it i.e.} $\dist(0,\mathcal{N})> 0$ and $\dist(0,\mathcal{N_{-}})> 0.$
 \item[\textnormal{(ii)}] For any $\zeta> 0$, the set $O_{\zeta}\cap \mathcal{N_{+}}$ is bounded in $W_{0}^{s,1}(\Omega)$.
\end{enumerate}
\end{lemma}
\begin{proof}
$(i)$     Let $u\in \mathcal{N}$. From conditions \eqref{Assump:f_0}, \eqref{Assump:f_1} and Proposition \ref{prop:embedding}(i), we get
    \begin{align}
    \label{dist(0,N)}
        d\leq E(u)&={\iint_{Q}}\frac{|u(x)-u(y)|}{|x-y|^{N+s}}\,dx\,dy-\int_{\Omega}F(x,u)\,dx \leq \|u\|_{W_{0}^{s,1}(\Omega)}+ C\|u\|_{L^{r}(\Omega)}^{r}\nonumber\\
        &\leq \|u\|_{W_{0}^{s,1}(\Omega)}+ C(C_{sob})^{r}\|u\|_{W_{0}^{s,1}(\Omega)}^{r}.
    \end{align}
Now, if $\|u\|_{W_{0}^{s,1}(\Omega)}\geq 1$, then clearly dist$(0,\mathcal{N})>0$. Otherwise, if $\|u\|_{W_{0}^{s,1}(\Omega)} < 1$, then from \eqref{dist(0,N)}, we find that $  \|u\|_{W_{0}^{s,1}(\Omega)} \geq \left(\frac{d}{1+C(C_{sob})^{r}}\right).$
This implies that there exists a constant $\rho >0$ such that $\text{dist}(0, \mathcal{N})=\inf_{u\in \mathcal{N}}\|u\|_{W_{0}^{s,1}(\Omega)}\geq\rho>0.$ For $u\in \mathcal{N_{-}}$, we get $\|u\|_{W_{0}^{s,1}(\Omega)}\neq 0$. From \eqref{lower bound of norm u} we have
\begin{align}
\label{dist(0,N-)}
\|u\|_{W_{0}^{s,1}(\Omega)} < S_{1}\|u\|_{W_{0}^{s,1}(\Omega)}^{r}.
 \end{align}
\noindent If $\|u\|_{W_{0}^{s,1}(\Omega)}\geq 1$, then clearly dist$(0,\mathcal{N}_{-})>0.$ Otherwise, if $\|u\|_{W_{0}^{s,1}(\Omega)}< 1$, then from \eqref{dist(0,N-)}, we get $\|u\|_{W_{0}^{s,1}(\Omega)} > \left(\frac{1}{S_{1}}\right)^{\frac{1}{r-1}}.$
   This implies that there exists a constant $\delta >0$ such that $\text{dist}(0, \mathcal{N}_{-})=\inf_{u\in \mathcal{N}_{-}}\|u\|_{W_{0}^{s,1}(\Omega)}\geq\delta>0.$ Next, we show
\text{(ii)}. If $u\in O_{\zeta}\cap \mathcal{N_{+}}$, then $E(u)<\zeta$ and $I(u)>0.$ Therefore, from condition \eqref{Assump:f_1}, we get
\begin{align}
\label{E tau is bounded}
      \zeta> E(u)
    \geq \iint_{Q}\frac{|u(x)-u(y)|}{|x-y|^{N+s}}\,dx\,dy-\int_{\Omega}\frac{1}{\theta}f(x,u)u\,dx
     &\geq \left(1-\frac{1}{\theta}\right)\iint_{Q}\frac{|u(x)-u(y)|}{|x-y|^{N+s}}\,dx\,dy +\frac{1}{\theta}I(u) \nonumber\\
     &\geq \left(\frac{\theta-1}{\theta}\right)\|u\|_{W_{0}^{s,1}(\Omega)}. 
\end{align}
 From \eqref{E tau is bounded} we conclude that the set $O_{\zeta}\cap \mathcal{N_{+}}$ is bounded in $W_{0}^{s,1}(\Omega)$.
\end{proof}
\begin{lemma}
\label{lambda{zeta} away from zero}
    Let the conditions \eqref{Assump:f_0}-\eqref{Assump:f_1} hold. Then, for any $\zeta > d$, the constants $\lambda_{\zeta}$ and $\Lambda_{\zeta}$ satisfy $0 < \lambda_{\zeta} \leq \Lambda_{\zeta} < +\infty.$ 
\end{lemma}
\begin{proof}
Using the fact that $W_{0}^{s,1}(\Omega)$ is embedded in $L^{2}(\Omega)$ and \eqref{E tau is bounded}, we have
$$\|u\|_{L^{2}(\Omega)}\leq C_{Sob}\|u\|_{W_{0}^{s,1}(\Omega)}\leq C_{Sob}\left(\frac{\zeta\theta}{\theta-1}\right).$$
It is enough to show that $\lambda_{\zeta}>0$. We proceed by contradiction. Assume that $\lambda_\zeta = 0$, then there exists a sequence $\{u_n\}_{n\in\mathbb{N}} \subset \mathcal{N}_\zeta$ such that $u_n \to 0$ in $L^2(\Omega)$ and a.e. in $\Omega$. By Lemma \ref{N+ is bounded}(i), there exists a constant $\rho > 0$, independent of $n$, such that
\begin{equation} 
    \label{un positive delta}
    \|u_n\|_{W_0^{s,1}(\Omega)} \geq \rho >0.
\end{equation}
On the other hand, from \eqref{E tau is bounded},  the sequence $\{u_n\}$ is uniformly bounded in $W_0^{s,1}(\Omega)$.
Since the embedding $W_0^{s,1}(\Omega) \hookrightarrow L^r(\Omega)$ is compact for all $r\in[1,\frac{N}{N-s})$. Therefore, up to a subsequence, there exists $u \in L^r(\Omega)$ such that $u_n \to u$ in $L^r(\Omega)$ and a.e in $\Omega$. By the uniqueness of the limit, it follows that $u(x)= 0$ a.e in $\Omega$. In particular, $ \|u_n\|_{L^r(\Omega)} \to 0.$
This further implies together with \eqref{Assump:f_0}, 
\begin{align*}
  \|u_n\|_{W_0^{s,1}(\Omega)} = \int_\Omega f(x,u_n)\,u_n\,dx \leq C \int_\Omega |u_n|^r\,dx = C \|u_n\|_{L^r(\Omega)}^r \to 0
\end{align*}
and contradicts \eqref{un positive delta}. Hence, the required claim.
\end{proof} 
Next, we define the $\omega$-limit set $\omega(u_{0})$ of the initial data $u_{0}\in W_{0}^{s,1}(\Omega)$ by
\[\omega(u_{0})=\bigcap_{\ell \geq 0} \overline{\left\{u(\cdot,t)\,|\, t \geq \ell \right\}}^{W_{0}^{s,1}(\Omega)}.\]
\begin{proof}[\textbf{Proof of the Theorem \ref{Main theorem high initial energy}}]
    $(i)$ If \( u_{0} \in \mathcal{N}_{+} \) and \( \|u_{0}\|_{L^{2}(\Omega)} \leq \lambda_{E(u_{0})} \), then by Lemma \ref{u in W delta} and following same argument as in \cite[Theorem 4.10]{Arora-2025}, we obtain 
    $ u(\cdot,t) \in \mathcal{N}_{+}\quad \text{for all } t \in [0,T_{\max}).$ Next, from \eqref{L2 norm decreasing} and \eqref{strong solution}, we get \( u(\cdot,t) \in O_{E(u_{0})} \) for all \( t \in [0,T_{\max}) \). By Lemma \ref{N+ is bounded}, we conclude that \( u(\cdot,t) \) is bounded in $W_{0}^{s,1}(\Omega)$ and \( E(u(\cdot,t)) < E(u_{0}) \) and \( I(u(\cdot,t)) > 0 \) for all \( t \in [0,T_{\max}) \).\\
\textbf{Claim 1: }$T_{\max}=\infty$. \\
For this purpose suppose by contradiction $T_{\max}<+\infty$, {\it i.e.}
\begin{equation}
    \label{Tmax infty}
    \|u(\cdot,t)\|_{L^{2}(\Omega)}\to \infty\quad \text{as }\, t\to T_{\max}^{-}.
\end{equation}
Since $u(\cdot,t)\in \mathcal{N}_{+}\cap \mathcal{O}_{E(u_{0})}$ for all $t\in [0,T_{\max})$, from \eqref{L2 norm decreasing}, we have 
 $\|u(\cdot,t)\|_{L^{2}(\Omega)}\leq \|u_{0}\|_{L^{2}(\Omega)}.$
Which contradicts \eqref{Tmax infty}. Hence $T_{\max}=\infty.$
This implies that \( u \in \mathcal{N}_{+} \cap O_{E(u_{0})} \) for all \( 0 \leq t < \infty \).  
Since \( I(u(\cdot,t)) > 0 \) for all \( 0 \leq t < \infty \), it follows from \eqref{L2 norm decreasing} that \( \|u(\cdot,t)\|_{L^{2}(\Omega)} \) is decreasing for \( 0 \leq t < \infty \) and \( E(u(\cdot,t)) < E(u_{0}) \) for all \( 0 \leq t < \infty \).
Therefore, for any \( w \in \omega(u_{0}) \) using the definition of $\omega$-limit and the lower semicontinuity of the norm, we obtain
\begin{align}
\label{w limit}
\|w\|_{L^{2}(\Omega)} = \lim_{t \to \infty} \|u(\cdot,t)\|_{L^{2}(\Omega)} < \lambda_{E(u_{0})}\quad\mbox{and}  
\quad E(w) \leq \liminf_{t \to \infty} E(u(\cdot,t)) < E(u_{0}).
\end{align}
Again applying the arguments of \cite[Theorem 4.10]{Arora-2025}, we obtain $ \omega(u_{0}) = \{0\}$ , {\it i.e.} $ u_{0} \in \mathcal{G}_{0}$.\\
$(ii)$ If \( u_{0} \in \mathcal{N}_{-} \) and \( \|u_{0}\|_{L^{2}(\Omega)} \geq \Lambda_{E(u_{0})} \), 
then\ by Lemma \ref{u in W delta} and using the same argument as in \cite[Theorem 4.10]{Arora-2025}, we get
$ u \in \mathcal{N}_{-} \cap O_{E(u_{0})}\quad \text{for all }  0 \leq t < T_{\max}.$ Suppose \( T_{\max} = \infty \). This implies that \( u \in \mathcal{N}_{-} \cap O_{E(u_{0})} \) for all \( 0 \leq t < \infty \). Since \( I(u(\cdot,t)) < 0 \) for all \( 0 \leq t < \infty \), it follows from \eqref{L2 norm decreasing} that \( \|u(\cdot,t)\|_{L^{2}(\Omega)} \) is strictly increasing for \( 0 \leq t < \infty \), and \( E(u(\cdot,t)) < E(u_{0}) \) for all \( 0 \leq t < \infty \). Therefore, for any \( w \in \omega(u_{0}) \), we obtain
\begin{align}
\label{w'' limit}
\|w\|_{L^{2}(\Omega)} = \lim_{t \to \infty} \|u(\cdot,t)\|_{L^{2}(\Omega)} > \Lambda_{E(u_{0})}, \ \ \mbox{and} \ 
E(w) \leq \liminf_{t \to \infty} E(u(\cdot,t)) < E(u_{0}).
\end{align}
Following the same argument as in \cite[Theorem 4.10]{Arora-2025}, we obain $\omega(u_{0}) = \emptyset$, which contradicts \( T_{\max} = \infty \). 
\end{proof}

Next, we show that the set $\{u_0 \in \mathcal{N}_- : \|u_0\|_{L^2(\Omega)} \geq \Lambda_{E(u_0)}\}$ is nonempty for all high initial energy data $u_0$ satisfying $E(u_0)>d.$
\begin{corollary}
\label{high initial blow up}
    Let the conditions \eqref{Assump:f_0}--\eqref{Assump:f_1} hold and $u_{0}\in W_{0}^{s,1}(\Omega)$. If
    \begin{align}
        \label{E(u0)<d}
         &d<E(u_{0})< \left(\frac{\theta-1}{C_{Sob}\theta}\right)\|u_{0}\|_{L^{2}(\Omega)},
    \end{align}
    then $u_{0}\in \mathcal{N}_{-}$ and $\|u_{0}\|_{L^{2}(\Omega)}\geq \Lambda_{E(u_{0})}$. 
\end{corollary}
\begin{proof}
From the condition \eqref{Assump:f_1}, \eqref{E(u0)<d} and the embedding of $W_{0}^{s,1}(\Omega)\hookrightarrow L^{2}(\Omega)$, we obtain
\begin{align}
\label{NE(U0)}
    E(u_0) &\geq \left(1-\frac{1}{\theta}\right)\|u_{0}\|_{W_{0}^{s,1}(\Omega)}+ \frac{1}{\theta} I(u_0)
    \geq  \left(\frac{\theta-1}{C_{Sob}\theta}\right)\|u_{0}\|_{L^{2}(\Omega)} + \frac{1}{\theta} I(u_0)
    \geq E(u_0)+ \frac{1}{\theta} I(u_0).
\end{align}
This implies that \( I(u_0) < 0 \). Therefore, we conclude that \( u_0 \in \mathcal{N}_- \).
Next, let \( u \in \mathcal{N}_{E(u_0)} \). From \eqref{NE(U0)} for $u$ (in place of $u_0$) and using the fact that $I(u) =0$, we obtain
\begin{equation}
\label{E(u0)}
\begin{aligned}
    E(u)
    &\geq \left(\frac{\theta-1}{C_{Sob}\theta}\right)\|u\|_{L^{2}(\Omega)}+\frac{1}{\theta} I(u) \geq \left(\frac{\theta-1}{C_{Sob}\theta}\right)\|u\|_{L^{2}(\Omega)}.
\end{aligned}
\end{equation}
From \eqref{E(u0)<d}, we have 
$$\left(\frac{\theta-1}{C_{Sob}\theta}\right)\|u\|_{L^{2}(\Omega)}\leq E(u)< E(u_{0})< \left(\frac{\theta-1}{C_{Sob}\theta}\right)\|u_{0}\|_{L^{2}(\Omega)}.$$
Taking the supremum over the set \( \mathcal{N}_{E(u_0)} \), we finish the proof.
\end{proof}
By following the same arguments as in \cite[Theorem 4.12]{Arora-2025}, we can prove the following result:
\begin{corollary}
\label{finthm}
 Let the conditions \eqref{Assump:f_0}--\eqref{Assump:f_1} hold. For any $M> d$, then there exists $u_{M}\in \mathcal{N}_{-}$ such that $E(u_{M})=M$ and $\|u_{M}\|_{L^{2}(\Omega)}\geq \Lambda_{E(u_{M})}$.
    \end{corollary}
\section*{Acknowledgement}
The first author acknowledges the financial support from the Anusandhan National Research Foundation (ANRF), India,
under Grant No. ANRF/ARGM/2025/000272/MTR. The second author is funded by the UGC Senior Research Fellowship with reference no. 211610115269.
\end{document}